\PassOptionsToPackage{unicode}{hyperref}
\PassOptionsToPackage{hyphens}{url}
\documentclass[reqno,11pt]{amsart}

\usepackage{amsmath,amssymb,amsthm,mathtools}
\usepackage{multirow}
\usepackage{caption}
\usepackage{subcaption}
\usepackage{longtable,booktabs,array,tabularx}
\usepackage{pdflscape}
\usepackage{calc}
\usepackage{graphicx}
\usepackage[a4paper,top=1.75cm,bottom=1.5cm,left=1.3cm,right=1.3cm]{geometry}
\usepackage{xcolor}
\usepackage{iftex}
\usepackage{enumitem}
\usepackage{algorithm}
\usepackage{algpseudocode}
\usepackage{tikz}
\usetikzlibrary{arrows.meta,calc,positioning,fit,backgrounds,decorations.pathreplacing}

\ifPDFTeX
  \PackageError{PEJWAK}{This source requires XeLaTeX}{Select XeLaTeX as the compiler.}
\fi
\usepackage{fontspec}
\usepackage{ebgaramond}
\usepackage{unicode-math}
\defaultfontfeatures{Ligatures=TeX,Scale=MatchLowercase}

\IfFileExists{microtype.sty}{\usepackage[]{microtype}\UseMicrotypeSet[protrusion]{basicmath}}{}
\IfFileExists{upquote.sty}{\usepackage{upquote}}{}
\IfFileExists{xurl.sty}{\usepackage{xurl}}{}
\usepackage[nocompress]{cite}
\usepackage{hyperref}
\IfFileExists{bookmark.sty}{\usepackage{bookmark}}{}
\hypersetup{
  hidelinks,
  pdfcreator={LaTeX},
  pdftitle={PEJWAK: A Profile-Echoing Self-Anchored Aggregation Architecture for Multi-Criteria Decision Analysis},
  pdfauthor={Seyyed Ahmad Edalatpanah}
}
\allowdisplaybreaks

\theoremstyle{definition}
\newtheorem{definition}{Definition}

\theoremstyle{plain}
\newtheorem{theorem}{Theorem}
\newtheorem{proposition}{Proposition}
\newtheorem{corollary}{Corollary}
\newtheorem{lemma}{Lemma}
\theoremstyle{remark}
\newtheorem{remark}{Remark}

\title{PEJWAK: A Profile-Echoing Self-Anchored Aggregation Architecture for Multi-Criteria Decision Analysis}
\author{Seyyed Ahmad Edalatpanah}

\begin{document}
\begin{abstract}
\textbf{Purpose:}
Aggregation determines both the final score and the level at which
compensation occurs. Additive rules exchange criterion values directly
in a completed sum, weighted-product rules propagate weakness through
the whole profile and collapse at an active zero, and post hoc hybrids
combine completed indices rather than contextualizing each criterion
before final aggregation. This paper introduces the Profile-Echoing
Jointly Weighted Anchoring Kernel (PEJWAK), a distinct profile-echoing
self-anchored architecture that fills this gap. Each criterion is
geometrically coupled with an endogenous anchor derived from the same
alternative, while one importance vector governs anchor formation,
criterion-specific retention, and final participation without a
separately elicited nonlinear control parameter.

\textbf{Methodology:}
The general formulation separates the anchor function, anchoring kernel,
retention exponents, contribution coefficients, and outer aggregation
rule. The canonical member uses a weighted-arithmetic self-anchor, sets
each retention exponent equal to the corresponding criterion-importance
value, and derives contribution coefficients through a normalized
square-root map. Analytical results establish continuity, strict
monotonicity on active coordinates, diagonal calibration, internality,
full-domain joint continuity, exact boundary behaviour, nonseparability,
and an exact logarithmic retention law. A logarithmic escort family
audits the fixed contribution rule through pairwise indifference
geometry, covariance sensitivity, finite rank phases, and its limiting
regime. A deterministic supplier-selection study compares structurally
distinct protocols, decomposes scores, traces weight transitions, audits
set dependence, and evaluates rank affinity through exact permutation
tails and Spearman's coefficient.

\textbf{Findings:}
Canonical PEJWAK is well defined on the complete score-weight domain and
requires neither score flooring nor epsilon regularization. It reproduces
uniform profiles, remains between the smallest and largest criterion
values, and preserves fully decomposable criterion contributions while
coupling every criterion to the complete alternative profile. The
equal-weight binary restriction violates bisymmetry and is therefore not
a classical quasi-arithmetic mean. The kernel retains exactly the
importance-specified fraction of a criterion's logarithmic deviation
from its anchor. An isolated zero removes only its own active term
whenever the remaining active scores sustain a positive anchor. The
numerical study identifies exact stability intervals, transition
thresholds, finitely many rank phases, and a kernel-determined limiting
order. Set-induced reversals are traced to current-set normalization and
disappear through this pathway under fixed exogenous bounds. Exhaustive
enumeration and automated reruns make the reported numerical results
exact, traceable, and reproducible.

\textbf{Originality/Value:}
PEJWAK establishes a generalizable design language for
profile-conditioned aggregation. It occupies an important methodological
space between direct weighted aggregation and parameter-intensive
nonadditive interaction models: it introduces endogenous
alternative-specific context without subset capacities, preserves
criterion identity, separates retention from final participation, and
supports complete decomposition and analytical auditing. Its value is
not merely a different ranking formula, but a criterion-level mechanism
through which each criterion preserves its identity while its expressed
contribution is conditioned by the profile of the alternative to which
it belongs.

\end{abstract}
\maketitle

\noindent\textbf{Affiliations:} Department of Applied Mathematics, Artificial Intelligence Research Center, Ayandegan University, Tonekabon, Iran; and Department of Mathematics, Saveetha School of Engineering, Saveetha Institute of Medical and Technical Sciences, Chennai 602105, India.\par
\noindent\textbf{E-mails:} \href{mailto:s.a.edalatpanah@aihe.ac.ir}{s.a.edalatpanah@aihe.ac.ir}; \href{mailto:saedalatpanah@gmail.com}{saedalatpanah@gmail.com}.\par
\noindent\textbf{Published version:} \emph{Journal of Decisions and Operations Research}, 11(3) (2026), 267--304. DOI: \href{https://doi.org/10.22105/dmor.2026.581617.2112}{10.22105/dmor.2026.581617.2112}. CC BY 4.0.\par
\noindent\textbf{ORCID:} \href{https://orcid.org/0000-0001-9349-5695}{0000-0001-9349-5695}.
\medskip

\section{Introduction}
\label{sec:introduction}

Score-based multi-criteria decision analysis (MCDA) begins with a
multidimensional performance profile and must ultimately convert that
profile into an overall score by which alternatives can be compared
\cite{Greco2025,Keeney1976}. The aggregation rule is therefore not a
neutral computational device. It determines which differences among
criteria may compensate for one another, how criterion importance shapes
that compensation, and whether the resulting score remains interpretable
on the normalized performance scale
\cite{Bouyssou1986,Beliakov2007}. The contrast is clearest at two familiar
poles. A linear additive rule permits direct exchange among weighted
criterion values, so strong performance in some dimensions may conceal a
severe shortfall in another. A weighted product propagates weakness
through the complete profile, and a zero on any positively important
criterion reduces the entire product to zero. The central design question
is therefore not merely where to place one global compromise between
addition and multiplication. It is whether each criterion can be
interpreted in relation to the performance profile of the alternative to
which it belongs, while preserving criterion identity, maintaining score
traceability, and preventing an isolated boundary value from erasing all
performance accumulated elsewhere.

Established aggregation families provide powerful answers to neighbouring
but distinct modelling questions. Additive and multiplicative
architectures determine direct or global forms of compensation
\cite{Beliakov2007,Grabisch2009}; ordered architectures attach influence
to rank positions \cite{Yager1988}; interaction-aware architectures
represent coalition effects \cite{Grabisch1996}; multiple-aggregation
procedures reconcile several evaluative channels
\cite{Zavadskas2012,Liao2020,Wen2020}; external-reference procedures
compare alternatives with shared benchmarks
\cite{HwangYoon1981,Stevic2020}; and penalty-based constructions derive
or correct whole-profile summaries \cite{Calvo2010,Mariani2024}. These
families locate context, reference, compensation, and weighting at
different stages of evaluation. The comparison developed in
Section~\ref{sec:aggregation-positioning} shows that they do not provide,
in an equivalent form, the particular combination targeted here: fixed
criterion identity, an endogenous reference extracted from the same
alternative, criterion-level coupling with that reference before final
aggregation, and distinct architectural roles for one common source of
criterion-importance information. This combination defines the
methodological space addressed by PEJWAK.

This study introduces the Profile-Echoing Jointly Weighted Anchoring
Kernel (PEJWAK), a profile-echoing self-anchored aggregation architecture.
For each alternative, an anchor function first extracts an internal
reference from its normalized performance profile. A bivariate anchoring
kernel then couples each criterion value with that reference, and an
outer aggregation rule combines the completed criterion-level terms.
The general formulation separates these layers so that the anchor,
retention mechanism, contribution rule, and final operator remain
mathematically distinguishable. In the canonical realization, a single
criterion-importance vector forms the weighted-arithmetic self-anchor,
sets the criterion-specific retention exponents, and generates the final
contribution coefficients through a normalized square-root map. PEJWAK
therefore introduces alternative-specific context without requiring
subset capacities, independently elicited interaction parameters, or a
separately tuned global nonlinear control parameter.

PEJWAK, an acronym for Profile-Echoing Jointly Weighted Anchoring Kernel,
takes its name from the Persian word \textit{pejwak}, meaning an echo or
reflected sound. The form \textit{pezhvāk} gives its formal transliteration. The name expresses the mechanism itself: each
criterion contributes to the formation of the alternative's profile and
is then interpreted against the reference extracted from that profile
without losing its identity. \textit{Profile-Echoing} denotes the return
of the profile to the criterion as context. \textit{Jointly Weighted}
denotes the coordinated but mathematically distinct roles of criterion
importance in anchor formation, retention, and contribution.
\textit{Anchoring Kernel} identifies the bivariate geometric map through
which each criterion value is joined with the profile-derived reference.

The paper makes four connected contributions. First, it establishes a
general profile-anchored aggregation family by formally separating the
anchor function, anchoring kernel, retention exponents, contribution
coefficients, and outer aggregation rule. The canonical PEJWAK operator
is then obtained as a specific importance-linked member of this family,
rather than as an isolated scoring formula. This construction provides a
generalizable design language through which alternative anchors,
retention mappings, and contribution rules can be developed while the
criterion-level architecture remains explicit. Second, the paper
establishes the mathematical foundations required for reliable use and
structural interpretation. At the family level, it proves continuity,
coordinatewise monotonicity, diagonal calibration, and internality. For
the canonical member, it further proves strict monotonicity on active
coordinates, joint continuity on the complete score-weight domain,
zero-weight neutrality, an exact characterization of the zero set,
endogenous profile coupling, nonseparability, failure of classical
bisymmetry in the equal-weight binary restriction, exact logarithmic
retention, and state-dependent marginal substitution.

Third, the paper develops the logarithmic escort geometry of the
contribution rule. Pairwise indifference hyperplanes, a covariance
sensitivity identity, finitely many rank phases, exact transition points,
and the limiting order are derived analytically. The square-root rule
remains the fixed operational specification, while the escort exponent
serves as an ex post audit coordinate that reveals the ranking
consequences of alternative concentration regimes without creating a
new preference parameter. Fourth, the numerical study converts the
architecture into a fully auditable computational object. It compares
prespecified and structurally distinct aggregation protocols, decomposes
every PEJWAK score into complete criterion contributions, traces
continuous criterion-importance paths and exact ranking thresholds,
identifies the normalization transport responsible for set-induced rank
reversal, and distinguishes directional top-weighted rank affinity from
symmetric full-ranking association. Exact permutation enumeration,
certified transition calculations, and a deterministic computational
workflow make the reported results traceable and independently
reproducible.

Across these connected contributions, PEJWAK establishes a
criterion-level aggregation relation for problems in which a performance
value should be interpreted within the profile of the same alternative.
It preserves criterion identity, introduces endogenous profile context,
localizes geometric sensitivity, and separates retention from final
participation while keeping the resulting score fully traceable. This
combination places PEJWAK between direct weighted aggregation and
parameter-intensive nonadditive interaction models without reproducing
either architecture.

The remainder of the paper follows the logic of this architecture.
Section~\ref{sec:aggregation-positioning} locates compensation, reference,
context, and criterion influence across representative aggregation
families and identifies the precise methodological position of PEJWAK.
Section~\ref{sec:design-self-anchored} formalizes the general
profile-anchored family and its canonical realization and establishes
their principal mathematical and computational properties.
Section~\ref{sec:numerical-study} conducts the comparative,
criterion-importance sensitivity, contribution-rule, rank-reversal,
rank-affinity, and reproducibility audits.
Section~\ref{sec:discussion} interprets the architectural advantages and
methodological implications of the results.
Section~\ref{sec:limitations} clarifies the evidential scope and
cross-domain transfer conditions of the architecture, and
Section~\ref{sec:conclusion} consolidates the contribution and concludes
the paper.

\section{Compensation, Reference, and the Geometry of Aggregation}
\label{sec:aggregation-positioning}

Aggregation operators determine more than how a multidimensional
performance profile is reduced to one score. They determine where
compensation occurs, how context enters the evaluation, what the weights
or control parameters act upon, and whether criterion identity survives
the construction \cite{Beliakov2007}. Additive, multiplicative, ordered,
interaction-aware, multiple-aggregation, external-reference, and
penalty-based architectures answer these questions at different levels
\cite{Bouyssou1986,Grabisch2009,Yager1988,Grabisch1996,
HwangYoon1981,Zavadskas2012,Liao2020,Wen2020,Calvo2010}. The relevant
comparison is therefore not a contest among formulas that pursue the
same objective. It is an architectural comparison of where information
is introduced and what evaluative relation each construction is able to
represent. Method selection must follow the preference semantics and
decision purpose of the problem rather than formula novelty alone
\cite{Cinelli2020,Cinelli2022,Bouyssou2000,French2002}.

Compensation itself is not a binary property. It specifies whether gains
on some criteria may offset losses on others, at what stage this exchange
occurs, and what restrictions govern it. Compensatory, noncompensatory,
and intermediate preference structures accordingly represent different
decision logics \cite{Bouyssou1986,Fishburn1976}. Outranking procedures
may prevent a favourable concordance pattern from establishing a
preference when sufficiently strong discordance or veto is present
\cite{Roy1991}; strong-sustainability models may likewise reject full
substitutability across critical dimensions \cite{Munda2008}. Criterion
weights consequently have no context-free interpretation. Depending on
the architecture, they may scale trade-offs, express importance or
priority, govern positional influence, or parameterize coalition effects
\cite{Vansnick1986}. PEJWAK remains compensatory, but it relocates a
substantive part of compensation from the completed score to the
formation of each criterion-level term. Its contribution is therefore
not the elimination of trade-offs, but their profile-conditioned
organization before final aggregation.

This distinction is clearest in direct profile aggregation. Simple
Additive Weighting (SAW) computes
$\sum_j w_jr_j$, so criterion values remain separate until they exchange
directly in the completed sum. The Weighted Product (WP) method computes
$\prod_j r_j^{w_j}$, coupling all active criteria in one global product
and transmitting an active zero to the complete score. Weighted power
means connect arithmetic and geometric behaviour through one common
order parameter, while quasi-arithmetic means apply one common
strictly monotone generator to all criterion values
\cite{Beliakov2007,Grabisch2009}. These families provide coherent global
aggregation geometries, but the same transformation regime governs the
whole profile. PEJWAK introduces a different level of construction: each
criterion is first coupled with an endogenous reference extracted from
the same alternative, and only then participates in the final sum.
Criterion-specific retention therefore replaces neither addition nor
multiplication; it reorganizes their geometric relation at the level of
the individual criterion term. The bisymmetry result established in
Section~\ref{sec:design-self-anchored} confirms that even the equal-weight
binary restriction of canonical PEJWAK is not a classical
quasi-arithmetic mean \cite{Aczel1948}.

Ordered and data-dependent operators place context elsewhere. In Ordered
Weighted Averaging (OWA), the coefficients belong to ordered positions
rather than fixed criterion identities \cite{Yager1988}. Weighted OWA
(WOWA) combines source importance with positional preference
\cite{Torra1997}; dependent OWA operators allow the observed
configuration to affect the effective weights \cite{Xu2006}; and the
Power Average increases the influence of values supported by the
remaining arguments \cite{Yager2001}. These constructions alter which
values receive influence after ordering or support assessment. Canonical
PEJWAK instead preserves fixed criterion identity throughout. The
criterion-importance information determines retention and contribution,
whereas the observed profile constructs the common endogenous reference
that changes the criterion term itself. The contextual response therefore
occurs through anchoring rather than through positional reassignment of
weight, preserving a direct interpretive path from each criterion to its
complete contribution.

Interaction-aware aggregation introduces still richer forms of
dependence. The Choquet integral represents explicit coalition effects
through a capacity on subsets of the criterion set
\cite{Grabisch1996}. Its marginal contributions can encode synergy,
redundancy, complementarity, substitutability, and, under suitable
structures, degrees of veto or favour behaviour
\cite{Marichal2004}. This expressive power requires an explicit
coalitional model; a general normalized capacity contains
$2^n-2$ nontrivial subset values before structural restrictions or
identification procedures are imposed. Bonferroni means couple each
argument with a summary of the remaining arguments
\cite{Bonferroni1950,Yager2009,Beliakov2010}, while Heronian means
introduce context through pairwise product terms \cite{Guan2006}.
PEJWAK generates a different form of dependence. One common reference is
extracted from the complete profile, and every criterion is coupled with
that reference. A change in one coordinate consequently moves the anchor
and propagates through the other criterion-level terms, producing
endogenous profile-mediated cross-effects without requiring a subset
capacity or a separately parameterized interaction matrix. When explicit
coalition semantics are the modelling objective, Choquet-type
constructions provide the appropriate richer language. When the required
context is the alternative's own profile rather than elicited synergy or
veto, PEJWAK supplies a substantially lighter, decomposable, and more
directly auditable architecture.

Multiple-aggregation and global blending procedures integrate information
at a later stage. Weighted Aggregated Sum Product Assessment (WASPAS)
combines completed additive and multiplicative scores
\cite{Zavadskas2012}. Double Normalization-based Multiple Aggregation
(DNMA) reconciles subordinate utilities and rankings generated through
several normalization and aggregation channels \cite{Liao2020}.
Mixed Aggregation by Comprehensive Normalization Technique (MACONT)
integrates multiple normalization representations and evaluates
deviations from a virtual reference alternative \cite{Wen2020}.
MULTIMOORA combines ratio, reference-point, and full multiplicative
systems \cite{Brauers2010}, while CoCoSo reconciles several compromise
appraisal sequences \cite{Yazdani2019}. The compensatory
$\gamma$-operator similarly blends two completed whole-profile
connectives through one global control parameter
\cite{Zimmermann1980}. These methods preserve several evaluative views
and combine them after complete or subordinate assessments have been
formed. PEJWAK integrates at a more elementary level. Arithmetic profile
formation and geometric retention are joined inside every criterion term
before the outer score exists. This pre-aggregation integration is one of
its clearest architectural advantages: the method does not append a
mixing rule to completed indices, but changes how criterion-level
evidence is constructed in the first place.

Reference-based architectures differ primarily in where the benchmark is
formed. TOPSIS evaluates alternatives against common ideal and anti-ideal
profiles \cite{HwangYoon1981}; MARCOS incorporates explicit ideal and
anti-ideal alternatives into a compromise-ranking structure
\cite{Stevic2020}; and MACONT constructs a virtual reference from
candidate-set information \cite{Wen2020}.

\clearpage
\begin{landscape}
\begin{center}
\captionof{table}{
Architectural Comparison of Representative Aggregation Families and the
Methodological Position of PEJWAK
}
\label{tab:conceptual-positioning}
\footnotesize
\renewcommand{\arraystretch}{1.18}
\setlength{\tabcolsep}{4pt}
\begin{tabularx}{\linewidth}{@{}p{3.3cm}p{4.4cm}p{4.8cm}p{4.0cm}X@{}}
\toprule
Aggregation architecture
&
Location of context or reference
&
Meaning of weights or control information
&
Treatment of criterion identity
&
Level of compensation or coupling
\\
\midrule

Additive aggregation, including SAW
&
No separate reference; criteria enter a common additive balance
&
Criterion contribution and, under the assumed value model, tradeoff scaling
&
Preserved through fixed criterion identities
&
Direct exchange among weighted criterion contributions
\\

Multiplicative aggregation, including WP
&
No separate reference; criteria are linked through the complete product
&
Criterion exponents determine multiplicative influence
&
Preserved through fixed criterion identities
&
Across the product; a zero on a positively weighted criterion may collapse the score
\\

Power and quasi-arithmetic means
&
A common order parameter or generator defines the profile geometry
&
Criterion coefficients operate within one shared transformation regime
&
Preserved, but all criteria undergo the same transformation logic
&
Within a common original or transformed coordinate system
\\

OWA, WOWA, and data-dependent ordered operators
&
Context enters through ordered position, observed configuration, or support
&
Positional weights, source importance combined with position, or data-dependent effective weights
&
OWA detaches weights from fixed identities; WOWA reintroduces source importance within an ordered scheme
&
Through ordered or data-dependent weighting of an arithmetic aggregate
\\

Choquet integral
&
Context is encoded by a capacity over criterion coalitions
&
Importance and interaction are represented through subset values
&
Preserved through coalition-dependent marginal contributions
&
At the level of coalitions, synergies, redundancies, and marginal contributions
\\

Bonferroni and Heronian means
&
Context is created through remaining inputs or pairwise products
&
Global order parameters, with criterion weights in weighted variants
&
Preserved within relational terms
&
Through leave-one-out summaries or cross-criterion products
\\

Multiple-aggregation procedures
&
Context is distributed across normalization, aggregation, reference, or comparability channels
&
Criterion weights within channels and method-specific reconciliation rules
&
Method-specific; generally retained within subordinate channels
&
Between complete or partially complete evaluative views
\\

External-reference methods, including TOPSIS, MARCOS, and MACONT
&
Ideal, anti-ideal, or virtual references are constructed across alternatives
&
Weights scale distances, utilities, or positions relative to shared benchmarks
&
Preserved relative to external criterion references
&
Between each alternative and a shared external reference structure
\\

Penalty-based aggregation
&
The aggregate minimizes disagreement between the inputs and a candidate output
&
Penalty or dissimilarity functions determine the notion of centrality
&
Criterion-specific under distinct penalties $p_j$; potentially symmetric otherwise
&
Through minimization of total disagreement around a candidate output
\\

Penalized power means and related indicators
&
Context is supplied by within-unit heterogeneity in a scaled transformed space
&
The order and any weights form the base mean; a data-driven factor corrects the completed score
&
Symmetric in the unweighted form; criterion importance may be retained in weighted variants
&
After formation of the base aggregate, through a row-level correction
\\

PEJWAK
&
An endogenous within-alternative anchor is derived from the same
performance profile and returned to every criterion as context
&
In the canonical realization, one importance vector enters through
distinct anchor-formation, retention, and contribution mappings; no
subset capacity or separately elicited global nonlinear control parameter
&
Fixed criterion identity is preserved throughout, and every complete
criterion contribution remains traceable
&
Criterion-level geometric coupling before final aggregation; endogenous
profile context with a fully decomposable score
\\

\bottomrule
\end{tabularx}
\end{center}
\end{landscape}
\clearpage
These methods ask how an
alternative stands relative to a shared external horizon. Penalty-based
aggregation instead defines the output as a centre minimizing prescribed
disagreement with the inputs \cite{Calvo2010}. Penalized power means first
form a whole-profile summary and then correct it through a
heterogeneity-based factor \cite{Mariani2024}. PEJWAK makes a different
use of reference. Its benchmark is derived from the same alternative and
does not remain external to criterion-term formation. The profile produces
an internal anchor, and that anchor returns to every criterion as context.
PEJWAK therefore neither compares an alternative with a common frontier,
nor selects the final score as an optimization centre, nor corrects a
completed mean after aggregation. It changes the manner in which the
components participate in producing the aggregate.

The defining chain is
\[
r_i
\longmapsto
A(r_i),
\qquad
\bigl(r_{ij},A(r_i)\bigr)
\longmapsto
K_{\alpha_j}\bigl(r_{ij},A(r_i)\bigr),
\qquad
\left\{
K_{\alpha_j}\bigl(r_{ij},A(r_i)\bigr)
\right\}_{j=1}^{n}
\longmapsto
P_{A,\varphi,\alpha}(r_i).
\]
The profile enters twice, but in nonredundant roles. It first produces the
shared reference and then remains represented through the criterion
values that are coupled with that reference. Every criterion thus helps
form the context that subsequently conditions its own interpretation.
This return of the profile to the criterion is the precise mathematical
content of self-anchoring and of the echo mechanism.

The preceding comparison identifies a methodological space that is not
represented equivalently by the established families. Direct weighted
aggregation preserves criterion identity but supplies no
alternative-specific contextual reference before summation. Global
multiplicative aggregation introduces geometric sensitivity but couples
the entire profile and inherits global zero propagation. Ordered
operators introduce context by changing positional influence.
Nonadditive integrals model explicit interactions but require coalition
semantics and substantially richer parameterization. Multiple-aggregation
and reference-based procedures preserve several completed views or invoke
benchmarks beyond the alternative itself. PEJWAK combines four features
within one construction: a within-alternative reference, fixed criterion
identity, criterion-level geometric coupling before final aggregation,
and a formal separation between retention and contribution. No compared
architecture provides this conjunction in an equivalent form.

Its profile dependence belongs to the wider landscape of mixture-type
and input-dependent aggregation \cite{Beliakov2007,Mesiar2006}, but its
specific mechanism is the self-anchored return of one common
profile-derived reference to every fixed criterion identity. Explicit
vetoes, coalitional synergy, positional attitudes, and external-frontier
comparison remain distinct modelling objectives.

Three coordinates summarize this positioning. The first is the source of
reference: absent, externally constructed, output-centred, or derived
from the same alternative. The second is the location of compensation or
coupling: direct exchange, whole-profile transformation, positional or
coalitional interaction, reconciliation after separate evaluations, or
criterion-level anchoring before final aggregation. The third is the
object of weighting: fixed criteria, ordered positions, coalitions,
complete evaluative channels, or distinct components of one
importance-linked architecture. Table~\ref{tab:conceptual-positioning}
condenses these coordinates and makes visible the distinctive conjunction
provided by PEJWAK: preserved criterion identity, an endogenous
within-alternative reference, pre-aggregation geometric coupling, and
separate retention and contribution roles.

\section{PEJWAK Architecture: Design Principles and the Self-Anchored Construction}
\label{sec:design-self-anchored}

PEJWAK is built from four design commitments that jointly define its
mathematical identity. First, each criterion remains attached to its
fixed identity through its observed value, retention behaviour, and
complete contribution to the final score. Second, when one
criterion-importance vector governs several architectural functions,
anchor formation, retention, and outer participation are represented by
distinct mappings rather than compressed into an undifferentiated notion
of weight. Third, the reference against which a criterion is interpreted
is extracted from the normalized performance profile of the same
alternative. Fourth, the completed operator is diagonally calibrated and
internal: it reproduces every uniform profile and remains between the
smallest and largest criterion values of the evaluated alternative.
These commitments place contextualization inside every criterion term
and before final aggregation.

Together, the four commitments establish more than one closed-form
scoring rule. They define a reusable profile-echoing architecture with
three linked mathematical layers. The anchor function extracts the
alternative-specific reference; the anchoring kernel determines how much of each
criterion-specific deviation from that reference is retained; and the
outer aggregation rule converts the anchored terms into fully traceable
criterion contributions and a final score. The layers are coordinated
but noninterchangeable: the anchor is not the kernel, the kernel output
is not yet the complete criterion contribution, and the PEJWAK operator
emerges only after anchoring, retention, and participation have been
joined. Subsection~\ref{subsec:formal-architecture} formalizes this
general design space and identifies the canonical importance-linked
member used throughout the paper.

\subsection{Formal Architecture and Canonical Construction}
\label{subsec:formal-architecture}

Table~\ref{tab:pejwak-notation} fixes the notation for the general profile-anchored architecture, the importance-linked mappings, and the canonical PEJWAK realization.

\begin{table}[htbp]
\caption{Notation for the general profile-anchored architecture and the canonical PEJWAK realization.}
\label{tab:pejwak-notation}
\centering
\small
\renewcommand{\arraystretch}{1.14}
\begin{tabularx}{\textwidth}{@{}p{3.3cm}X@{}}
\toprule
Symbol & Meaning \\
\midrule
$\mathcal A_i$, $C_j$
&
Alternative $i$, $i=1,\ldots,m$, and criterion $j$, $j=1,\ldots,n$
\\
$x_{ij}$, $r_{ij}$
&
Raw and normalized performance of alternative $\mathcal A_i$ on criterion $C_j$
\\
$r_i=(r_{i1},\ldots,r_{in})$
&
Normalized performance profile of alternative $\mathcal A_i$
\\
$\mathcal W_n$
&
Simplex of normalized nonnegative criterion-importance vectors
\\
$w=(w_1,\ldots,w_n)\in\mathcal W_n$
&
Criterion-importance vector
\\
$A_w(r_i)$
&
Admissible anchor derived from $r_i$ under importance profile $w$
\\
$\alpha(w)=(\alpha_1(w),\ldots,\alpha_n(w))$
&
Vector of criterion-specific retention exponents
\\
$\varphi(w)=(\varphi_1(w),\ldots,\varphi_n(w))$
&
Vector of nonnegative contribution coefficients, with $\varphi(w)\in\mathcal W_n$
\\
$K_{\alpha}(x,a)$
&
Boundary-extended anchoring kernel that geometrically couples score $x$ with anchor $a$
\\
$P_{A,\varphi,\alpha}(r_i)$
&
General profile-anchored operator for a fixed admissible design triple
\\
$S_i=S_w(r_i)$
&
Canonical weighted-arithmetic self-anchor
\\
$G_{ij}=K_{w_j}(r_{ij},S_i)$
&
Canonical anchored term for criterion $C_j$
\\
$T_{ij}=\varphi_j(w)G_{ij}$
&
Complete contribution of criterion $C_j$
\\
$P_w^{\mathrm{PEJWAK}}(r_i)$
&
Canonical PEJWAK score of alternative $\mathcal A_i$
\\
\bottomrule
\end{tabularx}
\end{table}
Let
\begin{equation*}
\mathcal W_n
=
\left\{
v\in[0,1]^n:
\sum_{j=1}^{n}v_j=1
\right\}.
\end{equation*}

The elementary operation underlying the architecture is the criterion-level anchoring kernel.\footnote{Throughout, \emph{kernel} denotes a compositional two-argument aggregation map that generates each criterion-level value by geometrically coupling a score with its anchor. It is not a Mercer (positive semidefinite, reproducing) kernel, nor a linear-algebraic kernel (the null space of a map), and no inner-product, spectral, or positive-definiteness structure is claimed or used.} For a fixed retention exponent $\alpha\in[0,1]$ and $x,a\in[0,1]$, define:

\begin{equation}
K_{\alpha}(x,a)=
\begin{cases}
a, & \alpha=0,\\[1mm]
x^{\alpha}a^{1-\alpha}, & 0<\alpha<1,\\[1mm]
x, & \alpha=1.
\end{cases}
\label{eq:anchoring-kernel}
\end{equation}
For every fixed $\alpha$, this is the continuous boundary extension of the corresponding weighted geometric mean on the unit square. It is nondecreasing in each argument and satisfies
\begin{equation*}
\min\{x,a\}
\leq
K_{\alpha}(x,a)
\leq
\max\{x,a\}.
\end{equation*}
The endpoint identities have direct structural meanings:
\[
K_0(x,a)=a
\qquad\text{and}\qquad
K_1(x,a)=x.
\]
The first represents complete return to the anchor; the second represents complete retention of the criterion score.

For each fixed criterion-importance vector $w\in\mathcal W_n$, an anchor
\[
A_w:[0,1]^n\longrightarrow[0,1]
\]
is called admissible if it is continuous, coordinatewise nondecreasing, and diagonally calibrated:
\begin{equation*}
A_w(c,\ldots,c)=c
\qquad
\text{for every }c\in[0,1].
\end{equation*}
When the anchor is independent of $w$, or when $w$ is fixed and no ambiguity can arise, the subscript is suppressed. Every admissible anchor is internal. Indeed, if
\[
\underline r=\min_jr_j,
\qquad
\overline r=\max_jr_j,
\]
then
\[
A_w(\underline r,\ldots,\underline r)
\leq
A_w(r)
\leq
A_w(\overline r,\ldots,\overline r),
\]
and diagonal calibration gives
\begin{equation*}
\underline r
\leq
A_w(r)
\leq
\overline r.
\end{equation*}

Given an admissible anchor $A$, a contribution vector
\[
\varphi\in\mathcal W_n,
\]
and a retention vector
\[
\alpha=(\alpha_1,\ldots,\alpha_n)\in[0,1]^n,
\]
define the general profile-anchored operator by
\begin{equation}
P_{A,\varphi,\alpha}(r_i)
=
\sum_{j=1}^{n}
\varphi_j
K_{\alpha_j}
\left(
r_{ij},A(r_i)
\right).
\label{eq:general-pejwak-family}
\end{equation}
On the strictly positive interior, this may be written as
\begin{equation*}
P_{A,\varphi,\alpha}(r_i)
=
\sum_{j=1}^{n}
\varphi_j
r_{ij}^{\alpha_j}
A(r_i)^{1-\alpha_j}.
\end{equation*}
Equation~\eqref{eq:general-pejwak-family}, rather than the expanded power expression, remains the governing definition on the boundary.

The adjective jointly weighted refers to the linkage of the architectural roles through a common criterion-importance vector, not to arbitrary independent choices of $A$, $\alpha$, and $\varphi$. A linked member therefore assigns to each $w\in\mathcal W_n$ an admissible anchor $A_w$, a retention vector $\alpha(w)\in[0,1]^n$, and a contribution vector $\varphi(w)\in\mathcal W_n$, giving
\begin{equation*}
P_{A_w,\varphi(w),\alpha(w)}(r_i)
=
\sum_{j=1}^{n}
\varphi_j(w)
K_{\alpha_j(w)}
\left(
r_{ij},A_w(r_i)
\right).
\end{equation*}
The three objects have a common informational source but remain distinct mathematical mappings.

Equation~\eqref{eq:general-pejwak-family} establishes the general design
space of profile-anchored aggregation. It permits the anchor, retention
mapping, and contribution rule to vary while preserving the defining
criterion-level architecture. For every fixed admissible design triple,
continuity, coordinatewise monotonicity, diagonal calibration, and
internality follow at the profile level. Dependence on a changing
criterion-importance vector is a separate regularity question because it
changes the architectural mappings themselves. The canonical
realization now provides the first fully specified importance-linked
member of this design space: one importance vector generates the
weighted-arithmetic self-anchor, the criterion-specific retention
exponents, and the normalized contribution coefficients through three
explicit and coordinated mappings.

The canonical realization assigns the following three mappings to the common criterion-importance vector:
\[
A_w(r_i)=S_w(r_i),
\qquad
\alpha_j(w)=w_j,
\qquad
\varphi_j(w)=
\frac{\sqrt{w_j}}{\sum_{k=1}^{n}\sqrt{w_k}}.
\]

\begin{definition}[Canonical PEJWAK operator]
Let
\[
r_i=(r_{i1},\ldots,r_{in})\in[0,1]^n
\qquad\text{and}\qquad
w\in\mathcal W_n.
\]
The self-anchor, contribution coefficients, anchored terms, and complete
criterion contributions are defined jointly by
\[
\begin{aligned}
S_i=S_w(r_i)&=\sum_{k=1}^{n}w_kr_{ik},\\
\varphi_j(w)&=\frac{\sqrt{w_j}}{\sum_{k=1}^{n}\sqrt{w_k}},
\qquad j=1,\ldots,n,\\
G_{ij}&=K_{w_j}(r_{ij},S_i),\\
T_{ij}&=\varphi_j(w)G_{ij}.
\end{aligned}
\]
The canonical PEJWAK score is
\begin{equation}
P_w^{\mathrm{PEJWAK}}(r_i)
=
\sum_{j=1}^{n}T_{ij}
=
\sum_{j=1}^{n}
\varphi_j(w)
K_{w_j}(r_{ij},S_i).
\label{eq:canonical-pejwak}
\end{equation}
If $S_i>0$ and
\[
J_+=\{j:w_j>0\},
\]
the score admits the expanded representation
\begin{equation}
P_w^{\mathrm{PEJWAK}}(r_i)
=
\sum_{j\in J_+}
\varphi_j(w)
r_{ij}^{w_j}
S_i^{1-w_j}.
\label{eq:canonical-expanded}
\end{equation}
If $S_i=0$, all positively weighted criterion scores are zero and Equation~\eqref{eq:canonical-pejwak} gives
\[
P_w^{\mathrm{PEJWAK}}(r_i)=0.
\]
\end{definition}

The canonical member converts one criterion-importance vector into three
coordinated but noninterchangeable actions. Through \(S_w\), criterion
importance determines how strongly each criterion shapes the internal
reference of the alternative. Through
\(\alpha_j(w)=w_j\), it determines the exact fraction of criterion-specific
logarithmic deviation retained relative to that reference. Through the
square-root contribution map, it determines the direct participation of
the completed anchored term in the final score. This is a deliberate
multi-role encoding of criterion importance, not the repeated use of one
coefficient with one meaning. It allows reference formation, criterion-specific retention, and final
participation to be controlled by one coherent informational source while
remaining mathematically identifiable.

The same construction gives \textit{Profile-Echoing} its exact
operational content. The criterion values first form the performance
profile and its endogenous self-anchor. That anchor then returns to every
criterion through
\[
G_{ij}=K_{w_j}(r_{ij},S_i).
\]
Accordingly, each active criterion helps construct the reference that
subsequently conditions its own interpretation and the interpretation of
the remaining criteria:
\[
(r_{i1},\ldots,r_{in})
\longrightarrow
S_i
\longrightarrow
\bigl(
K_{w_1}(r_{i1},S_i),\ldots,
K_{w_n}(r_{in},S_i)
\bigr).
\]
The echo is therefore an endogenous information mechanism, not merely a
name attached to the completed operator. Each criterion preserves a
controlled local voice while receiving the reflected context of the
alternative to which it belongs.

This profile dependence implies that equal local scores and equal criterion importance need not produce equal complete contributions across alternatives. For a minimal numerical illustration, let
\[
w=(0.50,0.30,0.20),
\qquad
r_a=(1.00,0.60,0.60),
\qquad
r_b=(0.20,0.60,0.60).
\]
The corresponding self-anchors are $S_a=0.80$ and $S_b=0.40$. Although
\[
r_{a2}=r_{b2}=0.60
\qquad\text{and}\qquad
w_2=0.30,
\]
the anchored terms differ:
\[
G_{a2}
=
0.60^{0.30}0.80^{0.70}
\approx
0.733852,
\qquad
G_{b2}
=
0.60^{0.30}0.40^{0.70}
\approx
0.451739.
\]
Since $\varphi_2(w)\approx0.321803$, their complete contributions are
\[
T_{a2}\approx0.236156,
\qquad
T_{b2}\approx0.145371.
\]
Hence, equality of a criterion score does not imply equality of its contribution: the contribution is conditioned by the performance profile in which that score occurs. This is the criterion-level consequence of the endogenous self-anchor.

The boundary case remains consistent with this interpretation. If $w_j=0$, then
\[
\varphi_j(w)=0,
\qquad
G_{ij}=K_0(r_{ij},S_i)=S_i,
\qquad
T_{ij}=0.
\]
Thus, an inactive criterion makes no complete contribution to the final score, even though its kernel output remains well defined.

The square-root contribution map preserves the support and pairwise ordering of the original importance vector:
\begin{equation*}
\varphi_j(w)=0
\quad\Longleftrightarrow\quad
w_j=0,
\end{equation*}
and, for $w_j,w_k>0$,
\begin{equation}
\frac{\varphi_j(w)}{\varphi_k(w)}
=
\left(
\frac{w_j}{w_k}
\right)^{1/2}.
\label{eq:ratio-compression}
\end{equation}
Consequently,
\[
w_j>w_k
\quad\Longleftrightarrow\quad
\varphi_j(w)>\varphi_k(w),
\]
while every non-unit pairwise ratio is moved toward one.

Let
\[
u=\left(\frac1n,\ldots,\frac1n\right).
\]
Then
\begin{equation*}
\varphi_j(w)
=
\frac{\sqrt{w_ju_j}}
{\sum_{k=1}^{n}\sqrt{w_ku_k}},
\end{equation*}
so $\varphi(w)$ is the normalized componentwise geometric compromise between the original criterion-importance profile and an equal-contribution baseline. The term compromise refers to this exact identity; it does not introduce a second elicited importance vector.

The exact ratio law in
Equation~\eqref{eq:ratio-compression} supplies the design rationale for
the square-root contribution map. Criterion importance has already shaped
the common self-anchor and the criterion-specific retention exponents.
Repeating the original positive ratios unchanged in the outer layer would
carry their full concentration into a third architectural role. The
square-root map instead preserves support and strict importance ordering
while moderating outer concentration by the exact relation
\[
\frac{\varphi_j(w)}{\varphi_k(w)}
=
\sqrt{\frac{w_j}{w_k}}.
\]
For example, when
\[
w=(0.50,0.30,0.20),
\]
the original ratio \(w_1/w_3=2.5\) becomes
\[
\frac{\varphi_1(w)}{\varphi_3(w)}
=
\sqrt{2.5}
\approx1.581.
\]
Criterion \(C_1\) retains its declared priority over \(C_3\), while the
outer aggregation preserves meaningful participation by the remaining
active criteria. The canonical square-root rule is therefore a precise
architectural balance: importance order is retained, but its
concentration is moderated at the final participation layer. The escort
family below makes the neighbouring equal-participation, direct-transfer,
and concentration regimes analytically visible while leaving the
canonical PEJWAK specification fixed at \(q=1/2\).

\begin{remark}[Contribution-rule escort family]
Let
\[
J_+=\{j:w_j>0\}.
\]
The broader power-transformed contribution map is
\begin{equation*}
\varphi_j(q)
=
\begin{cases}
\displaystyle
\frac{w_j^q}
{\sum_{k\in J_+}w_k^q},
& j\in J_+,\\[3mm]
0,
& j\notin J_+,
\end{cases}
\qquad
q\in[0,\infty).
\end{equation*}
At $q=0$, active criteria contribute equally; $q=1/2$ gives the canonical square-root rule; and $q=1$ recovers the original importance vector. Values $0<q<1$ compress positive weight ratios, whereas $q>1$ amplifies them. As $q\to\infty$, contribution concentrates equally on the criteria attaining the largest positive weight. The exponent $q$ is neither elicited from the decision maker nor fitted as an operational parameter. It serves only as an ex post analytical coordinate for auditing the fixed canonical contribution rule. Its logarithmic geometry is developed in Subsection~\ref{subsec:escort-geometry}.
\end{remark}

\subsection{Mathematical Properties and Structural Geometry}
\label{subsec:mathematical-properties}

This subsection establishes the mathematical discipline of the PEJWAK
architecture. The first results show that every fixed admissible
profile-anchored design is continuous, monotone, diagonally calibrated,
and internal. The subsequent results identify what the canonical
importance-linked member adds: strict improvement on active coordinates,
exact logarithmic retention, endogenous profile coupling,
nonseparability, departure from the classical quasi-arithmetic class,
and state-dependent marginal substitution. These properties convert
self-anchoring from an interpretive proposal into a formally controlled
aggregation geometry.

\begin{theorem}[Basic properties of a fixed profile-anchored operator]
\label{thm:basic-profile-properties}
Fix an admissible anchor $A:[0,1]^n\to[0,1]$, a contribution vector
$\varphi\in\mathcal W_n$, and a retention vector
$\alpha\in[0,1]^n$. Then the associated operator
\[
P_{A,\varphi,\alpha}(r)
=
\sum_{j=1}^{n}
\varphi_j
K_{\alpha_j}\bigl(r_j,A(r)\bigr)
\]
has the following properties:
\begin{enumerate}[label=\textnormal{(\roman*)}]
\item $P_{A,\varphi,\alpha}$ is continuous on $[0,1]^n$;
\item $P_{A,\varphi,\alpha}$ is coordinatewise nondecreasing;
\item it is diagonally calibrated, or equivalently, idempotent in the
aggregation sense:
\[
P_{A,\varphi,\alpha}(c,\ldots,c)=c
\qquad
\text{for every }c\in[0,1];
\]
\item it is internal:
\[
\min_j r_j
\leq
P_{A,\varphi,\alpha}(r)
\leq
\max_j r_j.
\]
\end{enumerate}
\end{theorem}

\begin{proof}
For each $j$, the exponent $\alpha_j$ is fixed. The map
\[
r\longmapsto \bigl(r_j,A(r)\bigr)
\]
is continuous on $[0,1]^n$, and $K_{\alpha_j}$ is continuous on
$[0,1]^2$ under the boundary convention in
Equation~\eqref{eq:anchoring-kernel}. Hence each composite map
\[
r\longmapsto K_{\alpha_j}\bigl(r_j,A(r)\bigr)
\]
is continuous. Since $P_{A,\varphi,\alpha}$ is a finite weighted sum of
these maps, it is continuous on $[0,1]^n$.

Now suppose that $r\leq s$ componentwise. The coordinatewise
monotonicity of $A$ gives
\[
A(r)\leq A(s).
\]
For each fixed $\alpha_j\in[0,1]$, the kernel $K_{\alpha_j}$ is
nondecreasing in both arguments. Indeed, the cases $\alpha_j=0$ and
$\alpha_j=1$ select the second and first arguments, respectively,
whereas for $0<\alpha_j<1$ both power factors are nonnegative and
nondecreasing. Therefore,
\[
K_{\alpha_j}\bigl(r_j,A(r)\bigr)
\leq
K_{\alpha_j}\bigl(s_j,A(s)\bigr)
\]
for every $j$. Since $\varphi_j\geq0$, summation yields
\[
P_{A,\varphi,\alpha}(r)
\leq
P_{A,\varphi,\alpha}(s).
\]

For any $c\in[0,1]$, diagonal calibration of the anchor gives
\[
A(c\mathbf 1)=c.
\]
The kernel definition then yields
\[
K_{\alpha_j}(c,c)=c
\]
for every $j$. Since $\varphi\in\mathcal W_n$,
\[
P_{A,\varphi,\alpha}(c\mathbf 1)
=
c\sum_{j=1}^{n}\varphi_j
=
c.
\]

Finally, let
\[
\underline r=\min_j r_j,
\qquad
\overline r=\max_j r_j.
\]
Then
\[
\underline r\mathbf 1
\leq
r
\leq
\overline r\mathbf 1.
\]
Coordinatewise monotonicity and diagonal calibration therefore imply
\[
\underline r
=
P_{A,\varphi,\alpha}(\underline r\mathbf 1)
\leq
P_{A,\varphi,\alpha}(r)
\leq
P_{A,\varphi,\alpha}(\overline r\mathbf 1)
=
\overline r,
\]
which proves internality.
\end{proof}

\begin{corollary}[Canonical strict monotonicity]
Let $P_w^{\mathrm{PEJWAK}}$ be the canonical PEJWAK operator. If
$r\leq s$ componentwise and
\[
r_k<s_k
\]
for at least one criterion with $w_k>0$, then
\[
P_w^{\mathrm{PEJWAK}}(r)
<
P_w^{\mathrm{PEJWAK}}(s).
\]
\end{corollary}

\begin{proof}
Since $r\leq s$ componentwise and $r_k<s_k$ for an active criterion
$k$, we have
\[
S_w(s)-S_w(r)
=
\sum_{j=1}^{n}w_j(s_j-r_j)
\geq
w_k(s_k-r_k)
>
0.
\]
Hence $S_w(r)<S_w(s)$. Moreover, $w_k>0$ implies
$\varphi_k(w)>0$.

We first show that the $k$th anchored term increases strictly. If
$w_k=1$, the boundary definition of the kernel gives
\[
K_{1}\bigl(r_k,S_w(r)\bigr)
=
r_k
<
s_k
=
K_{1}\bigl(s_k,S_w(s)\bigr).
\]
If $0<w_k<1$, then $s_k>0$ and
\[
S_w(s)\geq w_ks_k>0.
\]
Therefore,
\[
\begin{aligned}
K_{w_k}\bigl(r_k,S_w(r)\bigr)
&=
r_k^{w_k}S_w(r)^{1-w_k}\\
&\leq
r_k^{w_k}S_w(s)^{1-w_k}\\
&<
s_k^{w_k}S_w(s)^{1-w_k}\\
&=
K_{w_k}\bigl(s_k,S_w(s)\bigr).
\end{aligned}
\]

For every $j$, coordinatewise monotonicity of the boundary-extended
kernel yields
\[
K_{w_j}\bigl(r_j,S_w(r)\bigr)
\leq
K_{w_j}\bigl(s_j,S_w(s)\bigr).
\]
All corresponding differences are therefore nonnegative, while the
difference for $j=k$ is strictly positive and its contribution
coefficient satisfies $\varphi_k(w)>0$. Consequently,
\[
\begin{aligned}
P_w^{\mathrm{PEJWAK}}(s)
-
P_w^{\mathrm{PEJWAK}}(r)
&=
\sum_{j=1}^{n}\varphi_j(w)
\left[
K_{w_j}\bigl(s_j,S_w(s)\bigr)
-
K_{w_j}\bigl(r_j,S_w(r)\bigr)
\right]\\
&>0.
\end{aligned}
\]
Thus,
\[
P_w^{\mathrm{PEJWAK}}(r)
<
P_w^{\mathrm{PEJWAK}}(s).
\]
\end{proof}

The preceding results establish the regularity of the aggregate. Before
examining the profile coupling created by the shared anchor, it is useful
to make the geometry of the anchoring kernel explicit. For
$x,a>0$ and $\alpha\in[0,1]$,
\[
\log K_\alpha(x,a)
=
(1-\alpha)\log a
+
\alpha\log x,
\]
or, equivalently,
\begin{equation}
\log
\frac{K_\alpha(x,a)}{a}
=
\alpha
\log
\frac{x}{a}.
\label{eq:exact-log-retention}
\end{equation}
Thus, in logarithmic coordinates, $K_\alpha(x,a)$ lies exactly the
fraction $\alpha$ of the way from the anchor $a$ to the criterion
score $x$.

\begin{theorem}[Exact logarithmic retention law]
For $x,a>0$ and $\alpha\in[0,1]$, the anchoring kernel retains exactly
an $\alpha$-fraction of the logarithmic distance between the criterion
score $x$ and the anchor $a$:
\begin{equation}
\left|
\log
\frac{K_\alpha(x,a)}{a}
\right|
=
\alpha
\left|
\log
\frac{x}{a}
\right|.
\label{eq:log-distance-retention}
\end{equation}
Moreover,
\[
K_0(x,a)=a,
\qquad
K_1(x,a)=x,
\]
and, for $\alpha\in(0,1)$,
\begin{equation}
\frac{\mathrm d}{\mathrm d\alpha}
K_\alpha(x,a)
=
K_\alpha(x,a)
\log\frac{x}{a}.
\label{eq:kernel-alpha-derivative}
\end{equation}
Consequently, the path $\alpha\mapsto K_\alpha(x,a)$ is strictly
decreasing from $a$ to $x$ when $x<a$, strictly increasing from $a$
to $x$ when $x>a$, and constant at $a=x$ when $x=a$.
\end{theorem}

\begin{proof}
Equation~\eqref{eq:log-distance-retention} follows immediately by
taking absolute values in
Equation~\eqref{eq:exact-log-retention}. The endpoint identities follow
by setting $\alpha=0$ and $\alpha=1$ in
$K_\alpha(x,a)=x^\alpha a^{1-\alpha}$.

For fixed $x,a>0$, the kernel can be written as
\[
K_\alpha(x,a)
=
a\exp\left(
\alpha\log\frac{x}{a}
\right).
\]
Differentiation with respect to $\alpha$ gives
Equation~\eqref{eq:kernel-alpha-derivative}. Since
$K_\alpha(x,a)>0$, the sign of the derivative is determined by
$\log(x/a)$. It is negative when $x<a$, positive when $x>a$, and
zero when $x=a$. Together with the endpoint identities, this proves
the stated monotonic behaviour.
\end{proof}

For fixed $x,a>0$, the retention exponent regulates departures from the
anchor through the same logarithmic geometry on either side of $a$. If
$x<a$, a smaller $\alpha$ draws $K_\alpha(x,a)$ closer to $a$ and
attenuates the shortfall, whereas a larger $\alpha$ preserves more of
that shortfall. If $x>a$, the corresponding mechanism applies to the
surplus relative to the anchor. The kernel therefore scales downward
and upward logarithmic deviations by the same retention factor. This
interpretation is an algebraic property of the kernel and does not
presuppose a particular behavioural model of the decision maker.

A small calculation makes this geometry concrete. Let $a=1/4$ and
$\alpha=1/2$. For a score below the anchor,
\[
K_{1/2}\left(\frac{1}{16},\frac{1}{4}\right)
=
\frac{1}{8},
\]
whereas, for a score above the same anchor,
\[
K_{1/2}\left(1,\frac{1}{4}\right)
=
\frac{1}{2}.
\]
The original score-to-anchor ratios are respectively $1/4$ and $4$,
while the corresponding kernel-to-anchor ratios are $1/2$ and $2$.
Hence the kernel halves both the logarithmic shortfall and the
logarithmic surplus. The anchoring effect is therefore geometric, not
an arithmetic displacement toward the anchor.

\begin{remark}[Conditional scope of the retention path]
The path associated with
Equation~\eqref{eq:exact-log-retention} is
$\alpha\mapsto K_\alpha(x,a)$ for fixed $x$ and $a$; it is therefore
a path at the level of the kernel. In canonical PEJWAK,
$\alpha_j=w_j$. However, an admissible perturbation of the criterion
importance vector $w$ changes the retention exponents and the normalized
contribution vector $\varphi(w)$, and generally also changes the weighted
self-anchor $S_w(r)$. Moreover, because $w\in\mathcal W_n$, changing one
component requires at least one other component to adjust. Hence the path
in Equation~\eqref{eq:exact-log-retention} isolates the pure retention
effect for a fixed score and anchor. It should not be interpreted as the
complete comparative statics of
$P_w^{\mathrm{PEJWAK}}(r)$ along a change in $w$.
\end{remark}

With the retention geometry established, attention now turns to the
criterion coupling created by the shared anchor. Restrict the analysis
to the open positive interior $r\in(0,1)^n$. Suppressing the alternative
index, write
\[
S=S_w(r)=\sum_{\ell=1}^{n}w_\ell r_\ell
\]
and define
\begin{equation}
\begin{aligned}
B_w(r)
&=
\sum_{\ell=1}^{n}
\varphi_\ell(1-w_\ell)
r_\ell^{w_\ell}S^{-w_\ell}\\
&=
\sum_{\ell=1}^{n}
\varphi_\ell(1-w_\ell)
\left(\frac{r_\ell}{S}\right)^{w_\ell}.
\end{aligned}
\label{eq:profile-coupling-B}
\end{equation}
Equivalently,
\[
B_w(r)
=
\left.
\frac{\partial}{\partial a}
\sum_{\ell=1}^{n}
\varphi_\ell r_\ell^{w_\ell}a^{1-w_\ell}
\right|_{a=S_w(r)}.
\]
Thus, $B_w(r)$ measures the aggregate marginal response to the shared
anchor and isolates the indirect channel through which a change in one
criterion propagates across the complete performance profile.

\begin{proposition}[Profile coupling through the endogenous anchor]
Fix $w\in\mathcal W_n$ and let $r\in(0,1)^n$. With $B_w(r)$ defined
in Equation~\eqref{eq:profile-coupling-B}, the marginal effect of every
active criterion $p$ is
\begin{equation}
\frac{\partial P_w^{\mathrm{PEJWAK}}}{\partial r_p}
=
\varphi_pw_pr_p^{w_p-1}S^{1-w_p}
+
w_pB_w(r).
\label{eq:pejwak-gradient}
\end{equation}
For any two distinct active criteria $p\neq q$,
\begin{align}
\frac{\partial^2P_w^{\mathrm{PEJWAK}}}
{\partial r_p\partial r_q}
=
w_pw_q\Bigg[
&
\varphi_p(1-w_p)
r_p^{w_p-1}S^{-w_p}
+
\varphi_q(1-w_q)
r_q^{w_q-1}S^{-w_q}
\notag\\
&
-
\sum_{\ell=1}^{n}
\varphi_\ell w_\ell(1-w_\ell)
r_\ell^{w_\ell}S^{-w_\ell-1}
\Bigg].
\label{eq:pejwak-mixed-partial}
\end{align}
For every such pair, the mixed derivative does not vanish identically
on $(0,1)^n$. Consequently, whenever at least two criteria are active,
canonical PEJWAK is not additively separable on $(0,1)^n$ as
\[
P(r)=\sum_{j=1}^{n}g_j(r_j).
\]
\end{proposition}

\begin{proof}
Since
\[
\frac{\partial S}{\partial r_p}=w_p,
\]
the chain rule gives
\[
\begin{aligned}
\frac{\partial P_w^{\mathrm{PEJWAK}}}{\partial r_p}
&=
\sum_{\ell=1}^{n}\varphi_\ell
\Bigl[
\delta_{\ell p}w_\ell
r_\ell^{w_\ell-1}S^{1-w_\ell}
+
w_p(1-w_\ell)
r_\ell^{w_\ell}S^{-w_\ell}
\Bigr]\\
&=
\varphi_pw_pr_p^{w_p-1}S^{1-w_p}
+
w_pB_w(r),
\end{aligned}
\]
where $\delta_{\ell p}$ denotes the Kronecker delta. This proves
Equation~\eqref{eq:pejwak-gradient}.

For a distinct active criterion $q$,
\[
\frac{\partial B_w(r)}{\partial r_q}
=
\varphi_qw_q(1-w_q)
r_q^{w_q-1}S^{-w_q}
-
w_q
\sum_{\ell=1}^{n}
\varphi_\ell w_\ell(1-w_\ell)
r_\ell^{w_\ell}S^{-w_\ell-1}.
\]
Differentiating the first term in
Equation~\eqref{eq:pejwak-gradient} with respect to $r_q$ and combining
the result with $w_p\partial B_w(r)/\partial r_q$ yields
Equation~\eqref{eq:pejwak-mixed-partial}.

It remains to prove that this mixed derivative is not identically zero.
Fix all coordinates other than $r_p$ at positive values and write
$r_p=t$. Since $q$ is active,
\[
S(t)
=
w_pt+\sum_{\ell\neq p}w_\ell r_\ell
\longrightarrow
S_0
=
\sum_{\ell\neq p}w_\ell r_\ell
\geq
w_qr_q
>
0
\]
as $t\to0^+$. Because $p$ and $q$ are distinct active criteria,
\[
0<w_p<1,
\qquad
\varphi_p>0.
\]
Hence
\[
\varphi_p(1-w_p)
t^{w_p-1}S(t)^{-w_p}
\longrightarrow
+\infty,
\]
whereas all remaining terms inside the brackets in
Equation~\eqref{eq:pejwak-mixed-partial} remain bounded. Therefore,
\[
\frac{\partial^2P_w^{\mathrm{PEJWAK}}}
{\partial r_p\partial r_q}
\]
cannot vanish identically on $(0,1)^n$.

Any twice differentiable additively separable representation has
identically zero mixed derivatives in distinct coordinates.
Accordingly, canonical PEJWAK cannot admit such a representation
whenever at least two criteria are active.
\end{proof}

The nonzero mixed partials are the differential signature of endogenous
profile coupling. A change in one criterion moves the shared self-anchor
and thereby modifies the anchored terms and local marginal effects of the
remaining criteria. PEJWAK therefore creates genuine cross-criterion
dependence without discarding fixed criterion identity or introducing a
separate interaction matrix. The dependence has a precise semantics: it
is state-dependent and mediated by the alternative's own profile.
Accordingly, its sign should not be relabelled as a fixed pairwise
synergy, redundancy, or veto parameter; those concepts require an
explicit interaction model with different primitives.

\begin{theorem}[Failure of bisymmetry]
The equal-importance binary restriction
$F:[0,1]^2\to[0,1]$ of canonical PEJWAK,
\begin{equation}
F(x,y)
=
\frac{\sqrt{(x+y)/2}}{2}
\bigl(\sqrt{x}+\sqrt{y}\bigr),
\label{eq:binary-equal-pejwak}
\end{equation}
violates bisymmetry and, consequently, is not a classical
quasi-arithmetic mean.
\end{theorem}

\begin{proof}
A classical binary quasi-arithmetic mean generated by a continuous
strictly monotone function $f$ has the form
\[
M_f(x,y)
=
f^{-1}
\left(
\frac{f(x)+f(y)}{2}
\right).
\]
It necessarily satisfies the bisymmetry identity
\begin{equation}
M_f\bigl(M_f(a,b),M_f(c,d)\bigr)
=
M_f\bigl(M_f(a,c),M_f(b,d)\bigr)
\label{eq:bisymmetry-identity}
\end{equation}
for all $a,b,c,d\in[0,1]$
\cite{Aczel1948,Beliakov2007}. Indeed, applying $f$ to either side of
Equation~\eqref{eq:bisymmetry-identity} gives
\[
\frac{f(a)+f(b)+f(c)+f(d)}{4}.
\]

Now take
\[
(a,b,c,d)
=
\left(
\frac{1}{16},
\frac{1}{4},
\frac{3}{4},
\frac{1}{4}
\right).
\]
Direct evaluation of Equation~\eqref{eq:binary-equal-pejwak} gives
\[
F\bigl(F(a,b),F(c,d)\bigr)
\approx
0.3033513954,
\]
whereas
\[
F\bigl(F(a,c),F(b,d)\bigr)
\approx
0.3016690700.
\]
The two values differ by approximately $0.0016823254$. Hence $F$
violates bisymmetry and cannot be represented as
\[
f^{-1}
\left(
\frac{f(x)+f(y)}{2}
\right)
\]
for any continuous strictly monotone generator $f$.
\end{proof}

The local compensation geometry of canonical PEJWAK can be examined
through the marginal rate of substitution (MRS). Fix
$w\in\mathcal W_n$ and $r\in(0,1)^n$, and let $p\neq q$ be two active
criteria. Consider a local variation in $r_p$ and $r_q$ that holds all
remaining criterion scores fixed and leaves the aggregate score
unchanged. Along such a local equal-score path,
\[
\mathrm dP_w^{\mathrm{PEJWAK}}
=
\frac{\partial P_w^{\mathrm{PEJWAK}}}{\partial r_p}
\,\mathrm dr_p
+
\frac{\partial P_w^{\mathrm{PEJWAK}}}{\partial r_q}
\,\mathrm dr_q
=
0.
\]
Because $q$ is active, Equation~\eqref{eq:pejwak-gradient} gives
$\partial P_w^{\mathrm{PEJWAK}}/\partial r_q>0$ on this domain.
Therefore, suppressing the dependence on $r$ and $w$ in the notation,
the marginal rate of substitution is well defined by
\begin{equation}
\operatorname{MRS}_{pq}^{\mathrm{PEJWAK}}
=
-
\left.
\frac{\mathrm dr_q}{\mathrm dr_p}
\right|_{\mathrm dP_w^{\mathrm{PEJWAK}}=0}
=
\frac{
\partial P_w^{\mathrm{PEJWAK}}/\partial r_p
}{
\partial P_w^{\mathrm{PEJWAK}}/\partial r_q
}.
\label{eq:mrs-definition}
\end{equation}
It measures the first-order reduction in $r_q$ that locally compensates
for a marginal increase in $r_p$; it is not a finite-change exchange
rate.

\begin{proposition}[State-dependent marginal substitution]
Let $w\in\mathcal W_n$ and $r\in(0,1)^n$, and let $p\neq q$ be active
criteria. Write $S=S_w(r)$ and $\varphi_j=\varphi_j(w)$. Then
\begin{equation}
\operatorname{MRS}_{pq}^{\mathrm{PEJWAK}}
=
\frac{w_p}{w_q}
\frac{
\varphi_p\left(S/r_p\right)^{1-w_p}+B_w(r)
}{
\varphi_q\left(S/r_q\right)^{1-w_q}+B_w(r)
}.
\label{eq:pejwak-mrs}
\end{equation}
This rate is strictly positive and, in general, depends on the current
active-score profile through $r_p$, $r_q$, $S_w(r)$, and $B_w(r)$.
\end{proposition}

\begin{proof}
For every active criterion $j$, Equation~\eqref{eq:pejwak-gradient}
can be factored as
\[
\frac{\partial P_w^{\mathrm{PEJWAK}}}{\partial r_j}
=
w_j
\left[
\varphi_j
\left(\frac{S}{r_j}\right)^{1-w_j}
+
B_w(r)
\right].
\]
Because $j$ is active, $w_j>0$ and $\varphi_j>0$. Moreover,
$r_j>0$, $S>0$, and $B_w(r)\geq0$. Hence every displayed marginal
effect is strictly positive. In particular, the derivative with respect
to $r_q$ is nonzero, so the marginal rate of substitution is well
defined. Substitution of the expressions for $j=p$ and $j=q$ into
Equation~\eqref{eq:mrs-definition} gives
Equation~\eqref{eq:pejwak-mrs}.

To verify that the dependence on the score profile is genuine, consider
the equal-importance binary member
\[
w_1=w_2=\varphi_1=\varphi_2=\frac12.
\]
At
\[
r=\left(\frac1{100},\frac{49}{100}\right),
\]
one has $S=1/4$ and $B_w(r)=2/5$, and therefore
\[
\operatorname{MRS}_{12}^{\mathrm{PEJWAK}}
=
\frac{
\frac12\sqrt{(1/4)/(1/100)}+\frac25
}{
\frac12\sqrt{(1/4)/(49/100)}+\frac25
}
=
\frac{203}{53}.
\]
At the coordinate-swapped profile, the corresponding rate is
$53/203$. Thus the marginal substitution rate is not constant over
the positive interior.
\end{proof}

For the same criterion-importance vector
$w\in\mathcal W_n$ and score profile $r\in(0,1)^n$, consider two
distinct active criteria $p$ and $q$. The standard SAW and WP
aggregates satisfy
\[
\frac{\partial P_w^{\mathrm{SAW}}}{\partial r_j}
=
w_j,
\qquad
\frac{\partial P_w^{\mathrm{WP}}}{\partial r_j}
=
\frac{w_j}{r_j}P_w^{\mathrm{WP}}.
\]
Consequently,
\begin{equation}
\operatorname{MRS}_{pq}^{\mathrm{SAW}}
=
\frac{w_p}{w_q},
\qquad
\operatorname{MRS}_{pq}^{\mathrm{WP}}
=
\frac{w_p}{w_q}
\frac{r_q}{r_p}.
\label{eq:benchmark-mrs}
\end{equation}
The three aggregation rules therefore exhibit distinct local
compensation geometries. SAW has a state-independent substitution rate.
The WP rate is state-dependent, but only through the pairwise score
ratio $r_q/r_p$ and is independent of the remaining criterion scores.
By contrast, Equation~\eqref{eq:pejwak-mrs} shows that the canonical
PEJWAK rate is generally profile-coupled: besides $r_p$ and $r_q$, it
depends on the endogenous self-anchor $S_w(r)$ and the aggregate
profile-response term $B_w(r)$.

The profile-wide nature of this dependence can be seen while holding the
focal pair fixed. For the three-criterion canonical member with
\[
w=\left(\frac13,\frac13,\frac13\right),
\]
compare
\[
r=(0.04,0.64,0.04)
\qquad\text{and}\qquad
r'=(0.04,0.64,0.96).
\]
Only the third score changes, yet
$\operatorname{MRS}_{12}^{\mathrm{PEJWAK}}$ changes from approximately
$2.277$ to $2.793$. By contrast,
$\operatorname{MRS}_{12}^{\mathrm{SAW}}=1$ and
$\operatorname{MRS}_{12}^{\mathrm{WP}}=16$ at both profiles. Thus, in
canonical PEJWAK, a criterion outside the focal pair can alter their
local substitution rate by moving the endogenous self-anchor.

The comparison becomes especially sharp on uniform performance
profiles. Write $\varphi_j=\varphi_j(w)$ and define
\[
\overline w_\varphi
=
\sum_{\ell=1}^{n}\varphi_\ell w_\ell.
\]

\begin{corollary}[Diagonal compensation signature]
Let $w\in\mathcal W_n$, and let $p\neq q$ be active criteria. For every
$c\in(0,1)$ and every $j=1,\ldots,n$,
\begin{equation}
\left.
\frac{\partial P_w^{\mathrm{PEJWAK}}}
{\partial r_j}
\right|_{r=c\mathbf 1}
=
w_j
\left(
\varphi_j+1-\overline w_\varphi
\right)
=:
\eta_j^\Delta(w).
\label{eq:diagonal-gradient}
\end{equation}
The diagonal marginal-effect vector
$\eta^\Delta(w)=(\eta_1^\Delta(w),\ldots,\eta_n^\Delta(w))$
belongs to $\mathcal W_n$. Consequently, using the subscript
$\Delta$ to denote restriction to $r=c\mathbf1$,
\begin{equation}
\operatorname{MRS}_{pq,\Delta}^{\mathrm{PEJWAK}}(w)
:=
\left.
\operatorname{MRS}_{pq}^{\mathrm{PEJWAK}}
\right|_{r=c\mathbf 1}
=
\frac{\eta_p^\Delta(w)}{\eta_q^\Delta(w)}
=
\frac{
w_p(\varphi_p+1-\overline w_\varphi)
}{
w_q(\varphi_q+1-\overline w_\varphi)
}.
\label{eq:diagonal-mrs}
\end{equation}
This rate depends on the criterion-importance architecture but not on
the common performance level $c$.

On the same diagonal,
\begin{equation}
\operatorname{MRS}_{pq,\Delta}^{\mathrm{SAW}}(w)
=
\operatorname{MRS}_{pq,\Delta}^{\mathrm{WP}}(w)
=
\frac{w_p}{w_q}.
\label{eq:diagonal-benchmark-mrs}
\end{equation}
Under the canonical square-root contribution rule, the following exact
comparison therefore holds:
\[
\begin{cases}
\operatorname{MRS}_{pq,\Delta}^{\mathrm{PEJWAK}}(w)
>
\operatorname{MRS}_{pq,\Delta}^{\mathrm{SAW}}(w)
=
\operatorname{MRS}_{pq,\Delta}^{\mathrm{WP}}(w)
>1,
& w_p>w_q,\\[1mm]
\operatorname{MRS}_{pq,\Delta}^{\mathrm{PEJWAK}}(w)
=
\operatorname{MRS}_{pq,\Delta}^{\mathrm{SAW}}(w)
=
\operatorname{MRS}_{pq,\Delta}^{\mathrm{WP}}(w)
=1,
& w_p=w_q,\\[1mm]
0<
\operatorname{MRS}_{pq,\Delta}^{\mathrm{PEJWAK}}(w)
<
\operatorname{MRS}_{pq,\Delta}^{\mathrm{SAW}}(w)
=
\operatorname{MRS}_{pq,\Delta}^{\mathrm{WP}}(w)
<1,
& w_p<w_q.
\end{cases}
\]
Thus, on every uniform positive profile, canonical PEJWAK amplifies the
pairwise marginal asymmetry already declared by unequal criterion
importance. This exact diagonal signature distinguishes its local
compensation geometry from both SAW and WP. The result concerns the
uniform-profile diagonal and does not assert a global ordering of
compensability over the complete score domain.
\end{corollary}

\begin{proof}
At $r=c\mathbf1$, the endogenous self-anchor satisfies $S_w(r)=c$, and
\[
B_w(c\mathbf1)
=
\sum_{\ell=1}^{n}
\varphi_\ell(1-w_\ell)
=
1-\overline w_\varphi.
\]
Substitution into Equation~\eqref{eq:pejwak-gradient} gives
Equation~\eqref{eq:diagonal-gradient}. Moreover,
\[
\sum_{j=1}^{n}\eta_j^\Delta(w)
=
\sum_{j=1}^{n}w_j\varphi_j
+
\left(1-\overline w_\varphi\right)
\sum_{j=1}^{n}w_j
=
\overline w_\varphi+1-\overline w_\varphi
=
1.
\]
Since $0\leq\overline w_\varphi\leq1$, every component is
nonnegative, while the components corresponding to active criteria are
strictly positive. Hence $\eta^\Delta(w)\in\mathcal W_n$, and taking
the ratio of its $p$th and $q$th components proves
Equation~\eqref{eq:diagonal-mrs}.

Equation~\eqref{eq:benchmark-mrs} gives
$\operatorname{MRS}_{pq}^{\mathrm{SAW}}=w_p/w_q$, while the factor
$r_q/r_p$ in the WP rate equals one on $r=c\mathbf1$. Finally, the
canonical map
\[
w_j\longmapsto
\varphi_j(w)
=
\frac{\sqrt{w_j}}{\sum_k\sqrt{w_k}}
\]
is strictly increasing on the active support. Therefore,
\[
\frac{
\operatorname{MRS}_{pq,\Delta}^{\mathrm{PEJWAK}}(w)
}{
w_p/w_q
}
=
\frac{
\varphi_p+1-\overline w_\varphi
}{
\varphi_q+1-\overline w_\varphi
}
\]
is greater than, equal to, or less than one exactly when $w_p$ is
greater than, equal to, or less than $w_q$, respectively. This proves
the stated comparison.
\end{proof}

The exact retention law in
Equation~\eqref{eq:log-distance-retention} is an interior logarithmic
identity and is therefore undefined at $x=0$. Under the continuous
endpoint convention for the anchoring kernel,
\begin{equation}
K_\alpha(0,a)=0,
\qquad
a\in[0,1],
\qquad
\alpha\in(0,1].
\label{eq:zero-score-kernel}
\end{equation}
The excluded endpoint $\alpha=0$ is structurally different, since
$K_0(0,a)=a$. In the canonical member, however, $\alpha=w_j=0$
corresponds to zero criterion importance and simultaneously implies
$\varphi_j(w)=0$. Its boundary behaviour must therefore be examined at
the level of the complete PEJWAK summand rather than through the kernel
alone.
\subsection{Full-Domain Coherence: Well-Definedness and Joint Continuity}
\label{subsec:boundary-joint-continuity}

The canonical PEJWAK is not confined to the strictly positive interior.
Because normalized decision matrices naturally contain exact zeros and
criterion-importance vectors may lie on the boundary of the simplex, a
usable operator must remain coherent when scores or importance values
vanish. The decisive object is the complete criterion contribution, not
the anchoring kernel considered in isolation. The following theorem
establishes that the linked anchor, retention, and contribution mappings
jointly extend canonical PEJWAK to the complete score-weight domain
without score flooring, positive replacement, or
\(\varepsilon\)-regularization.

Write
\[
J_{+}(w)
=
\{j\in\{1,\ldots,n\}:w_j>0\}
\]
for the active criterion set and define
\[
D(w)
=
\sum_{k=1}^{n}\sqrt{w_k}.
\]
Recalling that
\[
S_i
=
S_w(r_i)
=
\sum_{k=1}^{n}w_kr_{ik},
\]
the complete contribution of criterion \(C_j\) is
\begin{equation*}
T_{ij}(r_i,w)
=
\varphi_j(w)G_{ij}
=
\varphi_j(w)
K_{w_j}\bigl(r_{ij},S_w(r_i)\bigr).
\end{equation*}

\begin{theorem}[Full-domain well-definedness and joint continuity]
\label{thm:boundary-joint-continuity}
For each alternative \(\mathcal A_i\), the canonical PEJWAK evaluation map
\[
(r_i,w)
\longmapsto
P_w^{\mathrm{PEJWAK}}(r_i)
=
\sum_{j=1}^{n}T_{ij}(r_i,w)
\]
is well-defined and jointly continuous on
\[
[0,1]^n\times\mathcal W_n.
\]
In particular, exact zero normalized scores belong to its domain and
require neither a score floor, a positive replacement, nor
\(\varepsilon\)-regularization.
\end{theorem}

\begin{proof}
Since \(w\in\mathcal W_n\) and
\(\sqrt{w_k}\geq w_k\) for every \(w_k\in[0,1]\),
\[
D(w)
=
\sum_{k=1}^{n}\sqrt{w_k}
\geq
\sum_{k=1}^{n}w_k
=
1.
\]
Hence
\[
\varphi_j(w)
=
\frac{\sqrt{w_j}}{D(w)}
\]
is well-defined and continuous on the entire weight simplex.
The weighted-arithmetic self-anchor \(S_w(r_i)\) is likewise jointly
continuous in \((r_i,w)\).

Fix
\[
(r_i^\ast,w^\ast)
\in
[0,1]^n\times\mathcal W_n
\]
and a criterion \(j\). Three cases exhaust the possible values of
\(w_j^\ast\).

If \(w_j^\ast=0\), kernel internality gives
\[
0
\leq
T_{ij}(r_i,w)
\leq
\varphi_j(w)
=
\frac{\sqrt{w_j}}{D(w)}
\leq
\sqrt{w_j}.
\]
Therefore,
\[
T_{ij}(r_i,w)
\longrightarrow
0
=
T_{ij}(r_i^\ast,w^\ast)
\]
as
\[
(r_i,w)\longrightarrow(r_i^\ast,w^\ast).
\]
This convergence is uniform in the criterion score and the anchor; in
particular, it does not require \(r_{ij}\to0\).

If \(0<w_j^\ast<1\), there exist \(\delta>0\) and a neighbourhood of
\((r_i^\ast,w^\ast)\) on which
\[
\delta
\leq
w_j
\leq
1-\delta.
\]
The map
\[
(x,a,\beta)
\longmapsto
x^\beta a^{1-\beta}
\]
is continuous on
\[
[0,1]^2\times[\delta,1-\delta],
\]
including points at which \(x=0\), \(a=0\), or both. Since
\(\varphi_j(w)\) and \(S_w(r_i)\) are jointly continuous, it follows that
\(T_{ij}(r_i,w)\) is continuous at \((r_i^\ast,w^\ast)\).

Finally, suppose that \(w_j^\ast=1\). Then every other weight converges
to zero and
\[
\begin{aligned}
\left|S_w(r_i)-r_{ij}\right|
&=
\left|
\sum_{k\neq j}w_k(r_{ik}-r_{ij})
\right| \\
&\leq
\sum_{k\neq j}w_k
=
1-w_j
\longrightarrow0.
\end{aligned}
\]
Because \(r_{ij}\to r_{ij}^\ast\), this implies
\[
S_w(r_i)\longrightarrow r_{ij}^\ast.
\]
Kernel internality gives
\[
\min\{r_{ij},S_w(r_i)\}
\leq
K_{w_j}\bigl(r_{ij},S_w(r_i)\bigr)
\leq
\max\{r_{ij},S_w(r_i)\},
\]
and hence
\[
K_{w_j}\bigl(r_{ij},S_w(r_i)\bigr)
\longrightarrow
r_{ij}^\ast.
\]
Moreover,
\[
D(w)\longrightarrow1
\qquad\text{and}\qquad
\varphi_j(w)\longrightarrow1.
\]
Consequently,
\[
T_{ij}(r_i,w)
\longrightarrow
r_{ij}^\ast
=
T_{ij}(r_i^\ast,w^\ast).
\]
Every remaining contribution converges to zero by the first case.

Thus every complete criterion contribution \(T_{ij}\) is jointly
continuous on the full score-weight domain. Since the canonical PEJWAK
score is their finite sum, it is jointly continuous on
\([0,1]^n\times\mathcal W_n\).
\end{proof}

\begin{corollary}[Zero-weight neutrality]
\label{cor:zero-weight-neutrality}
If \(w_j=0\), the canonical PEJWAK score is invariant under every change
in \(r_{ij}\). In particular, a zero-weight criterion affects neither
the self-anchor nor the final aggregate.
\end{corollary}

\begin{proof}
When \(w_j=0\), the score \(r_{ij}\) does not enter
\[
S_w(r_i)
=
\sum_{k=1}^{n}w_kr_{ik},
\]
and
\[
\varphi_j(w)=0.
\]
Hence \(T_{ij}=0\), while every remaining contribution is independent
of \(r_{ij}\).
\end{proof}

\begin{corollary}[Characterization of the zero set]
\label{cor:zero-set}
For every
\[
r_i\in[0,1]^n
\qquad\text{and}\qquad
w\in\mathcal W_n,
\]
\begin{equation}
P_w^{\mathrm{PEJWAK}}(r_i)=0
\quad\Longleftrightarrow\quad
r_{ij}=0
\quad
\text{for every }j\in J_{+}(w).
\label{eq:zero-set}
\end{equation}
\end{corollary}

\begin{proof}
Suppose first that \(r_{ij}=0\) for every \(j\in J_{+}(w)\).
Then
\[
S_w(r_i)
=
\sum_{j\in J_{+}(w)}w_jr_{ij}
=
0.
\]
Every active anchored term is therefore zero. For each inactive
criterion, \(\varphi_j(w)=0\), so its complete contribution is also
zero. Hence
\[
P_w^{\mathrm{PEJWAK}}(r_i)=0.
\]

Conversely, suppose that \(r_{ij}>0\) for at least one active criterion
\(j\). Since \(w_j>0\),
\[
S_w(r_i)
\geq
w_jr_{ij}
>
0.
\]
Moreover, \(\varphi_j(w)>0\), and therefore
\[
T_{ij}
=
\varphi_j(w)
K_{w_j}\bigl(r_{ij},S_w(r_i)\bigr)
>
0.
\]
All remaining contributions are nonnegative, so
\[
P_w^{\mathrm{PEJWAK}}(r_i)>0.
\]
This proves the equivalence.
\end{proof}

The full-domain coherence is a property of the complete jointly weighted
architecture rather than of the kernel in isolation. For a fixed
\(a_0>0\), the map \(K_\alpha(x,a)\) is not jointly continuous at
\((0,a_0,0)\). In canonical PEJWAK, however,
\[
w_j\longrightarrow0
\quad\Longrightarrow\quad
\varphi_j(w)\longrightarrow0,
\qquad
0\leq T_{ij}\leq\varphi_j(w),
\]
so the complete contribution suppresses the kernel-level discontinuity.
This result reveals a substantive advantage of the linked construction:
zero importance removes a criterion coherently from both the self-anchor
and the final score, while a zero value on an active criterion removes
only its own anchored contribution. The complete PEJWAK score vanishes
if and only if every active criterion value is zero. Thus exact zeros are
part of the mathematical domain of the operator, not exceptional inputs
requiring a numerical patch.
\subsection{Contribution-Rule Geometry and Rank Phases}
\label{subsec:escort-geometry}
The fixed square-root contribution rule possesses an exact and globally
auditable rank geometry. To expose that geometry without altering the
canonical operator, the normalized matrix, self-anchors, retention
exponents, and matrix of kernel values \(G=[G_{ij}]\) are held fixed, and only the
contribution transformation is continued along the escort path. This
construction places the canonical point \(q=1/2\) within a complete
analytical continuum from equal participation to concentrated
participation. The path is not an additional decision variable; it is a
mathematical instrument for determining how contribution concentration
changes pairwise score gaps, ranking phases, and the limiting order.

Let
\[
J_{+}(w)=\{j:w_j>0\},
\qquad
Z(q)=\sum_{k\in J_{+}(w)}w_k^q.
\]
Then
\begin{equation}
\varphi_j(q)
=
\begin{cases}
\displaystyle
\frac{w_j^q}{Z(q)}, & j\in J_{+}(w),\\[3mm]
0, & j\notin J_{+}(w),
\end{cases}
\qquad
q\in[0,\infty).
\label{eq:escort-path}
\end{equation}
For any active criteria \(j\) and \(k\),
\begin{equation*}
\frac{\varphi_j(q)}{\varphi_k(q)}
=
\left(\frac{w_j}{w_k}\right)^q.
\end{equation*}
Thus \(q=0\) assigns equal contribution to the active criteria,
\(q=1/2\) recovers the canonical square-root rule, \(q=1\) transfers the
criterion-importance vector directly to the contribution channel, and
\(q>1\) progressively magnifies disparities among its positive components.
Throughout this section, \(q\) is only an ex post audit coordinate; it is
neither elicited from the decision maker nor fitted to obtain a preferred
ranking.

Writing
\[
\ell_j=\log w_j,
\qquad j\in J_{+}(w),
\]
Equation~\eqref{eq:escort-path} becomes
\begin{equation*}
\varphi(q)
=
\operatorname{softmax}(q\ell).
\end{equation*}
The contribution vector therefore follows a one-parameter exponential
path on the active-criterion simplex. Because the kernel values
\[
G_{ij}=K_{w_j}(r_{ij},S_i)
\]
are held fixed, the audited score of alternative \(\mathcal A_i\) is
\begin{equation}
P_i(q)
=
\sum_{j\in J_{+}(w)}\varphi_j(q)G_{ij}.
\label{eq:escort-score}
\end{equation}

\begin{proposition}[Endpoint and concentration geometry]
\label{prop:escort-endpoints}
Let \(w\in\mathcal W_n\). For every \(j\in J_{+}(w)\),
\[
\varphi_j(0)=\frac{1}{|J_{+}(w)|},
\qquad
\varphi_j(1)=w_j.
\]
Define
\[
w_{\max}=\max_{1\leq j\leq n}w_j,
\qquad
J_{\max}(w)=\{j:w_j=w_{\max}\}.
\]
Then, as \(q\to\infty\),
\[
\varphi_j(q)
\longrightarrow
\begin{cases}
|J_{\max}(w)|^{-1}, & j\in J_{\max}(w),\\
0, & j\notin J_{\max}(w),
\end{cases}
\]
and consequently
\begin{equation}
P_i(q)
\longrightarrow
\frac{1}{|J_{\max}(w)|}
\sum_{j\in J_{\max}(w)}G_{ij}.
\label{eq:escort-limit-score}
\end{equation}
If the maximum component of the criterion-importance vector is unique,
the limiting score is the anchoring-kernel output associated with that
criterion.
\end{proposition}

\begin{proof}
The endpoint identities follow from Equation~\eqref{eq:escort-path} and
\(\sum_{j=1}^{n}w_j=1\). For the limit, divide the numerator and denominator
in Equation~\eqref{eq:escort-path} by \(w_{\max}^{q}\). Ratios associated
with smaller weights converge to zero, whereas those associated with
\(J_{\max}(w)\) remain equal to one. Substitution into
Equation~\eqref{eq:escort-score} gives
Equation~\eqref{eq:escort-limit-score}.
\end{proof}

Differentiating the escort coefficients gives the centered logarithmic
identity
\begin{equation}
\frac{\mathrm d\varphi_j(q)}{\mathrm dq}
=
\varphi_j(q)
\left(
\log w_j
-
\sum_{k\in J_{+}(w)}\varphi_k(q)\log w_k
\right).
\label{eq:escort-phi-derivative}
\end{equation}

\begin{proposition}[Covariance sensitivity identity]
\label{prop:escort-covariance}
For every alternative \(\mathcal A_i\),
\begin{equation}
\frac{\mathrm dP_i(q)}{\mathrm dq}
=
\operatorname{Cov}_{\varphi(q)}
\bigl(G_i,\log w\bigr),
\label{eq:escort-score-derivative}
\end{equation}
where \(G_i=(G_{ij})_{j\in J_{+}(w)}\). For two alternatives
\(\mathcal A_a\) and \(\mathcal A_b\), define
\[
d_{ab}=G_a-G_b,
\qquad
\Delta_{ab}(q)=P_a(q)-P_b(q).
\]
Then
\begin{equation}
\frac{\mathrm d\Delta_{ab}(q)}{\mathrm dq}
=
\operatorname{Cov}_{\varphi(q)}
\bigl(d_{ab},\log w\bigr).
\label{eq:escort-gap-derivative}
\end{equation}
\end{proposition}

\begin{proof}
Differentiate Equation~\eqref{eq:escort-score}, substitute
Equation~\eqref{eq:escort-phi-derivative}, and collect the centered terms.
The pairwise identity follows by linearity of covariance in its first
argument.
\end{proof}

Equation~\eqref{eq:escort-score-derivative} has a precise local meaning:
increasing \(q\) raises \(P_i(q)\) when the alternative's kernel-output
profile has positive covariance with the logarithmic importance profile,
and lowers it when that covariance is negative. Pairwise movement is
governed by Equation~\eqref{eq:escort-gap-derivative}. On the
active-criterion simplex
\[
\mathcal S_{J_{+}(w)}
=
\left\{
\psi\in[0,1]^{J_{+}(w)}:
\sum_{j\in J_{+}(w)}\psi_j=1
\right\},
\]
the pairwise indifference set is the simplex section
\begin{equation}
\mathcal H_{ab}
=
\left\{
\psi\in\mathcal S_{J_{+}(w)}:
\langle\psi,d_{ab}\rangle=0
\right\}.
\label{eq:escort-indifference-hyperplane}
\end{equation}
When \(d_{ab}\neq0\), the underlying equality defines a hyperplane; its
intersection with the simplex may nevertheless be empty. Hence
\(P_a(q^*)=P_b(q^*)\) exactly when the escort path meets
\(\mathcal H_{ab}\) at \(q^*\). Such an equality is a transverse crossing,
and therefore a genuine local order reversal, whenever
\[
\Delta_{ab}'(q^*)\neq0.
\]
A zero derivative does not by itself imply a reversal, because the path
may touch the equality set tangentially.

To determine how many pairwise equalities can occur, group the active
weights into their \(s\) distinct levels
\[
\omega_1>\omega_2>\cdots>\omega_s>0,
\qquad
J_h=\{j\in J_{+}(w):w_j=\omega_h\}.
\]
For a fixed pair \((a,b)\), define the grouped gap coefficients
\[
c_h
=
\sum_{j\in J_h}d_{ab,j},
\qquad h=1,\ldots,s.
\]
Then
\begin{equation}
\Delta_{ab}(q)
=
\frac{F_{ab}(q)}{Z(q)},
\qquad
F_{ab}(q)
=
\sum_{h=1}^{s}c_h\omega_h^q.
\label{eq:escort-exponential-polynomial}
\end{equation}
Because \(Z(q)>0\), the gap and the exponential polynomial have the same
sign.

\begin{theorem}[Finite contribution-rule phase structure]
\label{thm:finite-escort-phases}
Consider a finite collection of \(m\) alternatives and a fixed
\(w\in\mathcal W_n\) with \(s\) distinct positive weight levels. For each
pair \((a,b)\), exactly one of two cases holds. If
\(c_1=\cdots=c_s=0\), then \(\Delta_{ab}(q)=0\) for every
\(q\in[0,\infty)\), so the pair is persistently tied along the entire
escort path. Otherwise, \(F_{ab}\) has at most \(s-1\) real zeros,
counted with multiplicity, and therefore at most \(s-1\) equality points
on \([0,\infty)\).

Let \(\mathcal Z\) denote the union of the equality points of all pairs
that are not persistently tied. Then \(\mathcal Z\) is finite and satisfies
\begin{equation}
|\mathcal Z|
\leq
\binom{m}{2}(s-1).
\label{eq:escort-transition-bound}
\end{equation}
On every connected component of
\([0,\infty)\setminus\mathcal Z\), the complete weak ranking induced by
\(P_i(q)\) is constant. Every genuine order reversal occurs at a point of
\(\mathcal Z\), although an isolated equality point need not produce a
reversal.
\end{theorem}

\begin{proof}
If all grouped coefficients vanish,
Equation~\eqref{eq:escort-exponential-polynomial} gives
\(F_{ab}\equiv0\) and hence \(\Delta_{ab}\equiv0\).

Suppose that the grouped coefficient vector is not identically zero and
discard its zero terms. It remains to show that an exponential polynomial
with \(t\) nonzero terms and distinct positive bases has at most \(t-1\)
real zeros, counted with multiplicity. The claim is immediate for \(t=1\).
For \(t\geq2\), order the remaining bases as
\(\widetilde\omega_1>\cdots>\widetilde\omega_t\), factor out the positive
term \(\widetilde\omega_1^q\), and write
\[
f(q)
=
a_1
+
\sum_{h=2}^{t}a_h
\exp\left(
q\log\frac{\widetilde\omega_h}{\widetilde\omega_1}
\right),
\qquad a_h\neq0.
\]
Its derivative is an exponential polynomial with \(t-1\) nonzero terms.
By induction, \(f'\) has at most \(t-2\) real zeros, counted with
multiplicity. The multiplicity form of Rolle's theorem then implies that
\(f\) has at most \(t-1\) real zeros. Multiplication by
\(\widetilde\omega_1^q>0\) does not change the zero set. Since
\(t\leq s\), the required bound is \(s-1\).

There are only \(\binom{m}{2}\) alternative pairs, so the union of all
nonpersistent pairwise equality sets is finite and obeys
Equation~\eqref{eq:escort-transition-bound}. Between consecutive points of
this union, every nonpersistent pairwise gap has constant sign, while every
persistent pair remains tied. The induced complete weak ranking is
therefore constant on each connected component.
\end{proof}

The finite-phase theorem converts the escort continuation from an
exploratory sensitivity curve into an exact rank-geometric audit. It
bounds the number of pairwise equality points, partitions the complete
\(q\)-axis into finitely many constant-ranking phases, distinguishes
transverse reversals from tangential contacts, and determines the
limiting order under concentrated participation. The operational PEJWAK
definition remains fixed at \(q=1/2\); the escort path functions as an
analytical microscope through which the rank consequences of that fixed
design choice become fully visible.
\subsection{Computational Realization and Worked Illustration}
\label{subsec:computational-realization}

The formal architecture admits a deterministic and fully transparent
computational realization. PEJWAK receives a normalized performance
matrix as its numerical input; transforming a raw matrix \(X\) into
\[
R\in[0,1]^{m\times n}
\]
is an application-level preprocessing decision rather than part of the
aggregation operator. Criterion directions, normalization formulas,
reference bounds, and the treatment of constant criteria must therefore
be declared separately. This separation is operationally consequential:
current-set normalization can transmit changes in the candidate set to
the normalized profiles, whereas canonical PEJWAK evaluates every fixed
normalized row through its own endogenous self-anchor.

For a normalized profile
\(r_i=(r_{i1},\ldots,r_{in})\), the complete scalar chain is
\begin{equation}
S_i=\sum_{j=1}^{n}w_jr_{ij},
\qquad
G_{ij}=K_{w_j}(r_{ij},S_i),
\qquad
T_{ij}=\varphi_jG_{ij},
\qquad
P_i^{\mathrm{PEJWAK}}=\sum_{j=1}^{n}T_{ij},
\label{eq:scalar-computational-chain}
\end{equation}
where
\[
w_j=
\frac{\widetilde w_j}{\sum_{k=1}^{n}\widetilde w_k},
\qquad
\varphi_j=
\frac{\sqrt{w_j}}{\sum_{k=1}^{n}\sqrt{w_k}}.
\]
This chain makes all three roles of criterion importance computationally
visible. The vector \(w\) forms the self-anchor, the exponent \(w_j\)
controls the retention of criterion \(j\)'s logarithmic deviation, and
\(\varphi_j\) determines the final participation of the completed kernel
term. Every PEJWAK score is therefore accompanied by an exact
criterion-level decomposition rather than produced as an opaque terminal
index.

Algorithm~\ref{alg:canonical-pejwak-rowwise} implements this construction in
the same componentwise, row-by-row order used in a hand calculation. It first
forms the self-anchor of one alternative and then evaluates, weights, and sums
the criterion-level kernel outputs for that row before proceeding to the next
alternative.

\begin{algorithm}[htbp]
\caption{Componentwise row-by-row evaluation of canonical PEJWAK}
\label{alg:canonical-pejwak-rowwise}
\begin{algorithmic}[1]
\Require normalized performance matrix
\(R=[r_{ij}]\in[0,1]^{m\times n}\); nonnegative criterion-importance input
\(\widetilde w=(\widetilde w_1,\ldots,\widetilde w_n)^\top\) satisfying
\(\sum_{j=1}^{n}\widetilde w_j>0\)
\Ensure score vector \(\mathbf P\), criterion-contribution matrix \(T\), and
descending weak ranking \(\mathcal R\)
\State \(w_j\gets\widetilde w_j/\sum_{k=1}^{n}\widetilde w_k\),
\(j=1,\ldots,n\)
\State \(\varphi_j\gets\sqrt{w_j}/\sum_{k=1}^{n}\sqrt{w_k}\),
\(j=1,\ldots,n\)
\For{\(i=1,\ldots,m\)}
    \State \(S_i\gets\sum_{j=1}^{n}w_jr_{ij}\)
    \State \(P_i^{\mathrm{PEJWAK}}\gets0\)
    \For{\(j=1,\ldots,n\)}
        \State \(G_{ij}\gets K_{w_j}(r_{ij},S_i)\)
        \State \(T_{ij}\gets\varphi_jG_{ij}\)
        \State \(P_i^{\mathrm{PEJWAK}}\gets
        P_i^{\mathrm{PEJWAK}}+T_{ij}\)
    \EndFor
\EndFor
\State \(\mathbf P\gets
(P_1^{\mathrm{PEJWAK}},\ldots,P_m^{\mathrm{PEJWAK}})^\top\)
\State \(\mathcal R\gets\operatorname{WeakRank}_{\downarrow}(\mathbf P)\)
\State \Return \(\mathbf P,T,\mathcal R\)
\end{algorithmic}
\end{algorithm}

The same scalar construction admits an exact matrix--vector realization, which
compresses the loops in Algorithm~\ref{alg:canonical-pejwak-rowwise} into
broadcasted array operations (Algorithm~\ref{alg:canonical-pejwak-matrix}). This is an algebraically equivalent computational
form of canonical PEJWAK, not a different member of the operator family. 
Let
\[
R=[r_{ij}]\in[0,1]^{m\times n},
\qquad
w=(w_1,\ldots,w_n)^\top,
\qquad
\varphi=(\varphi_1,\ldots,\varphi_n)^\top,
\]
let \(\mathbf 1_q\) denote the \(q\)-dimensional vector of ones, and let
\(\mathbf 1_{p\times q}\) denote the \(p\times q\) matrix of ones. All
self-anchors are obtained simultaneously as
\begin{equation*}
\mathbf s=Rw=(S_1,\ldots,S_m)^\top.
\end{equation*}
Replicate the criterion-importance vector and the self-anchor vector across
the full decision matrix by defining
\begin{equation*}
W=\mathbf 1_mw^\top,
\qquad
A=\mathbf s\mathbf 1_n^\top.
\end{equation*}
Thus, \(W_{ij}=w_j\) and \(A_{ij}=S_i\). Define the entrywise kernel map by
\[
\left[\mathcal K_W(R,A)\right]_{ij}
=K_{W_{ij}}(R_{ij},A_{ij}).
\]
The complete matrix of criterion-level kernel outputs is therefore
\begin{equation}
G=\mathcal K_W(R,A).
\label{eq:matrix-kernel-map}
\end{equation}
On the strictly positive interior, this expression becomes
\begin{equation*}
G
=
R^{\circ W}
\circ
A^{\circ(\mathbf 1_{m\times n}-W)},
\end{equation*}
where \(\circ\) denotes the Hadamard product and
\(X^{\circ Y}\) denotes entrywise exponentiation, so that
\([X^{\circ Y}]_{ij}=x_{ij}^{y_{ij}}\); the same notation is used
componentwise for vectors. Equation~\eqref{eq:matrix-kernel-map},
rather than the positive-interior power expression, remains the governing form
on the boundary. The conventions established in
Subsection~\ref{subsec:boundary-joint-continuity} are applied entrywise, without
introducing an arbitrary \(\varepsilon\)-perturbation.

The matrix of complete criterion contributions and the full score vector are
\begin{equation}
T=G\operatorname{Diag}(\varphi),
\qquad
\boxed{
\mathbf P
=T\mathbf 1_n
=G\varphi
},
\qquad
\mathbf P
=
\left(
P_1^{\mathrm{PEJWAK}},\ldots,P_m^{\mathrm{PEJWAK}}
\right)^\top.
\label{eq:vectorized-pejwak}
\end{equation}
Although the last operation is a matrix--vector product, the mapping from the
original performance matrix \(R\) to \(\mathbf P\) is generally nonlinear.
The matrix of kernel values \(G\) depends on \(R\) both directly through each criterion
score and indirectly through the endogenous anchor vector \(\mathbf s=Rw\).

This dependence yields a precise computational separation. Once \(R\), \(w\),
and \(\varphi(w)\) have been fixed, the score of alternative
\(\mathcal A_i\) depends only on its own normalized profile:
\[
P_i^{\mathrm{PEJWAK}}=P_w^{\mathrm{PEJWAK}}(r_i),
\qquad
S_i=S_w(r_i).
\]
Evaluating one alternative therefore requires neither the anchor nor the score
of any other alternative. This row-level independence is not additive
separability across criteria. All criteria within a row remain coupled through
their shared endogenous self-anchor, even though distinct rows can be evaluated
independently after preprocessing.

\begin{algorithm}[htbp]
\caption{Equivalent matrix--vector evaluation of canonical PEJWAK}
\label{alg:canonical-pejwak-matrix}
\begin{algorithmic}[1]
\Require normalized performance matrix
\(R=[r_{ij}]\in[0,1]^{m\times n}\); nonnegative criterion-importance input
\(\widetilde w=(\widetilde w_1,\ldots,\widetilde w_n)^\top\) satisfying
\(\mathbf 1_n^\top\widetilde w>0\)
\Ensure score vector \(\mathbf P\), criterion-contribution matrix \(T\), and
descending weak ranking \(\mathcal R\)
\State \(w\gets\widetilde w/(\mathbf 1_n^\top\widetilde w)\)
\State \(\varphi\gets w^{\circ 1/2}/(\mathbf 1_n^\top w^{\circ 1/2})\)
\State \(\mathbf s\gets Rw\)
\State \(W\gets\mathbf 1_mw^\top\)
\State \(A\gets\mathbf s\mathbf 1_n^\top\)
\State \(G\gets\mathcal K_W(R,A)\)
\State \(T\gets G\operatorname{Diag}(\varphi)\)
\State \(\mathbf P\gets T\mathbf 1_n=G\varphi\)
\State \(\mathcal R\gets\operatorname{WeakRank}_{\downarrow}(\mathbf P)\)
\State \Return \(\mathbf P,T,\mathcal R\)
\end{algorithmic}
\end{algorithm}

Algorithms~\ref{alg:canonical-pejwak-rowwise} and
\ref{alg:canonical-pejwak-matrix} are exactly equivalent in exact arithmetic.
The first exposes the alternative-wise and criterion-wise construction, whereas
the second evaluates the same self-anchors, kernel outputs, contributions, and
scores through array operations. Their distinction concerns evaluation order
and memory layout, not the mathematical method or the resulting ranking. In
both algorithms, every kernel call applies the full-domain conventions
directly; hence no numerical patch is required for zero scores or zero
criterion importance. Ties are defined mathematically by equality of final
scores. If floating-point software replaces exact equality with a numerical
tolerance, that tolerance is an implementation and reporting convention and
must be declared separately. When numerical ranks are required, alternatives
with equal scores receive their average rank. The worked illustration below
follows Algorithm~\ref{alg:canonical-pejwak-rowwise} to make every intermediate
quantity visible, whereas the supplementary implementation and the bulk
PEJWAK calculations reported in this study use
Algorithm~\ref{alg:canonical-pejwak-matrix}. The criterion-contribution matrix
\(T\) is retained when a criterion-level decomposition is required and may be
omitted when only the final scores are needed.

\begin{proposition}[Computational complexity]
Under the standard unit-cost real-arithmetic model, with each kernel evaluation
treated as a constant-time operation, canonical PEJWAK score evaluation for
\(m\) alternatives and \(n\) criteria requires \(O(mn)\) time. Constructing a
descending weak order of the \(m\) scores by comparison-based sorting requires
\(O(m\log m)\) time. The complete scoring and ranking procedure therefore has
time complexity
\begin{equation}
O(mn+m\log m).
\label{eq:total-complexity}
\end{equation}
A fully vectorized implementation that materializes \(G\), and optionally
\(T\), requires \(O(mn)\) working memory. If the stored normalized matrix is
processed row by row, the additional requirement falls to \(O(n)\) working
memory plus \(O(m)\) storage for the scores and ranks, giving \(O(m+n)\)
additional memory.
\end{proposition}

\begin{proof}
Normalizing criterion importance and constructing \(\varphi\) cost \(O(n)\).
For each alternative, forming its weighted-arithmetic self-anchor and
evaluating the \(n\) kernel terms require \(O(n)\) operations. Processing all
\(m\) rows therefore costs \(O(mn)\), after which comparison-based sorting
costs \(O(m\log m)\). Explicit storage of a constant number of
\(m\times n\) arrays is \(O(mn)\). In row-wise execution, a contribution vector
and one row of kernel outputs require \(O(n)\) storage, while the score and rank
vectors require \(O(m)\).
\end{proof}

The row-wise architecture also determines how the canonical operator behaves
in large, dynamic, and distributed environments. Array-oriented
implementations may execute the two matrix--vector products and the entrywise
kernel map through vector instructions, multicore numerical libraries,
graphics processing units, or distributed matrix frameworks without altering
the mathematical operator.

\begin{proposition}[Parallel evaluability]
Suppose that \(R\) and \(w\) are fixed and that
\(1\leq N_{\mathrm{proc}}\leq m\) processing units receive a balanced row
partition. The idealized scoring time is
\begin{equation}
O\!\left(
n\left\lceil\frac{m}{N_{\mathrm{proc}}}\right\rceil
\right),
\label{eq:parallel-complexity}
\end{equation}
in addition to the \(O(n)\) initialization required for the normalized
criterion-importance and contribution vectors. The total arithmetic work
remains \(O(mn)\).
\end{proposition}

\begin{proof}
For fixed \(w\), the vector \(\varphi(w)\) is common to all alternatives and is
computed once. Each processing unit receives a disjoint subset of rows. For
every assigned row \(r_i\), it computes \(S_i\), the \(n\) anchored terms, and
\(P_i\) without accessing any other row. A balanced partition assigns at most
\(\lceil m/N_{\mathrm{proc}}\rceil\) rows to each unit, and every row requires
\(O(n)\) work. Summation over all units preserves the sequential work bound
\(O(mn)\).
\end{proof}

Equation~\eqref{eq:parallel-complexity} is an idealized scoring bound, not a
claim of perfect linear speedup. Communication, memory transfer, scheduling,
hardware occupancy, and load imbalance may reduce realized acceleration. The
absence of cross-row dependencies nevertheless makes the aggregation stage
naturally data parallel. Ranking remains a separate global operation performed
after the scores have been collected and may itself use a parallel or
distributed sorting procedure when \(m\) is very large.

The same locality permits memory-efficient execution when the complete matrix
\(G\) cannot be stored. Processing blocks of \(b\) rows requires
\(O(bn+n)\) working memory in addition to any retained score vector, while the
total scoring work remains \(O(mn)\). In an out-of-core or streaming
implementation, a block can be read, scored, and released before the next block
is loaded. If only the best \(k\) alternatives are required, a size-\(k\)
priority queue reduces retained output storage to \(O(k)\) and maintains the
current leading set in \(O(m\log k)\) total queue-update time.

The architecture also supports exact local updates when the normalization
bounds and \(w\) remain fixed. If an entire profile changes from \(r_i\) to
\(r_i'\), only that row's self-anchor, anchored terms, and score must be
recomputed, at a cost of \(O(n)\). If only entry \(r_{ij}\) changes, the
self-anchor first admits the constant-time update
\begin{equation*}
S_i'
=
S_i+w_j\left(r_{ij}'-r_{ij}\right).
\end{equation*}
Because the revised anchor enters every criterion-level kernel in the same row,
all \(n\) anchored terms must generally be refreshed; the complete row update
therefore remains \(O(n)\).

A newly arriving alternative can similarly be scored in \(O(n)\) time under
fixed external normalization bounds and inserted into a balanced ordered score
structure in \(O(\log m)\) time. Emitting a new explicit rank vector for all
alternatives still requires \(O(m)\) output time. If normalization instead uses
sample minima and maxima from the current candidate set, adding, deleting, or
updating an alternative may alter an entire criterion column and require global
renormalization and rescoring. This dependence originates in preprocessing,
not in the self-anchored aggregation rule, and is examined directly in
Subsection~\ref{subsec:rank-reversal}.

A compact calculation now follows the componentwise schedule of
Algorithm~\ref{alg:canonical-pejwak-rowwise} to make the complete mechanism
explicit. Consider
\begin{equation*}
R=
\begin{bmatrix}
0.80&0.60&0.40\\
0.70&0.90&0.50\\
0.60&0.50&0.90
\end{bmatrix},
\qquad
w=(0.50,0.30,0.20).
\end{equation*}
The canonical square-root transformation gives
\begin{equation*}
\varphi
=
(0.415446,\;0.321803,\;0.262751),
\end{equation*}
where the displayed entries are rounded to six decimal places. The
weighted-arithmetic self-anchors are
\begin{equation*}
(S_1,S_2,S_3)
=
(0.660000,\;0.720000,\;0.630000).
\end{equation*}

For alternative \(\mathcal A_1\),
\[
S_1
=
0.50(0.80)+0.30(0.60)+0.20(0.40)
=
0.66.
\]
Its criterion-level anchored terms are
\begin{align*}
G_{11}
&=
K_{0.50}(0.80,0.66)
=
0.80^{0.50}0.66^{0.50}
=
0.726636,
\\
G_{12}
&=
K_{0.30}(0.60,0.66)
=
0.60^{0.30}0.66^{0.70}
=
0.641396,
\\
G_{13}
&=
K_{0.20}(0.40,0.66)
=
0.40^{0.20}0.66^{0.80}
=
0.597100.
\end{align*}
The complete criterion contributions are
\[
(T_{11},T_{12},T_{13})
=
(0.301878,\;0.206403,\;0.156889),
\]
and hence
\[
P_1^{\mathrm{PEJWAK}}
=
T_{11}+T_{12}+T_{13}
=
0.665170.
\]
Repeating the same row-by-row computation for \(\mathcal A_2\) and
\(\mathcal A_3\) gives the complete results in
Table~\ref{tab:worked-pejwak-example}.

\begin{table}[htbp]
\caption{Criterion-level anchoring and final PEJWAK scores in the worked
illustration.}
\label{tab:worked-pejwak-example}
\centering
\small
\begin{tabular}{@{}lrrrrrr@{}}
\toprule
Alternative
&
Self-anchor
&
\(G_{i1}\)
&
\(G_{i2}\)
&
\(G_{i3}\)
&
PEJWAK score
&
Rank
\\
\midrule
\(\mathcal A_1\)
&
0.660000
&
0.726636
&
0.641396
&
0.597100
&
0.665170
&
2
\\
\(\mathcal A_2\)
&
0.720000
&
0.709930
&
0.769849
&
0.669360
&
0.718552
&
1
\\
\(\mathcal A_3\)
&
0.630000
&
0.614817
&
0.587800
&
0.676583
&
0.622352
&
3
\\
\bottomrule
\end{tabular}
\end{table}

Equivalently, the matrix--vector realization in
Algorithm~\ref{alg:canonical-pejwak-matrix} collects the three componentwise
results simultaneously as
\begin{equation*}
\mathbf P
=G\varphi
=
(0.665170,\;0.718552,\;0.622352)^\top,
\end{equation*}
where the displayed values are rounded to six decimal places.

The example makes the exact retention law visible. For every criterion,
\[
\log\frac{G_{ij}}{S_i}
=
w_j\log\frac{r_{ij}}{S_i},
\]
so the three criteria retain respectively \(50\%\), \(30\%\), and \(20\%\)
of their logarithmic deviations from the alternative-specific self-anchor.

Alternative \(\mathcal A_2\) obtains the highest score. Its high performance
on the second criterion is moderated from \(0.90\) to \(0.769849\), while its
third-criterion shortfall is drawn from \(0.50\) toward \(S_2=0.72\), yielding
\(G_{23}=0.669360\). Its most important criterion remains close to its original
score because
\[
G_{21}=K_{0.50}(0.70,0.72)=0.709930.
\]

Alternative \(\mathcal A_3\) shows the complementary effect. Its largest raw
performance, \(0.90\), occurs on the least important criterion. Because
\(w_3=0.20\), only one fifth of its logarithmic deviation from
\(S_3=0.63\) is retained, giving
\[
G_{33}=0.676583.
\]
The same low importance also gives this criterion the smallest direct
contribution coefficient, \(\varphi_3=0.262751\). By contrast, the first
criterion has the largest importance but the lower performance \(0.60\), and
its anchored term remains comparatively close to that value:
\[
G_{31}=0.614817.
\]
The example therefore displays all three roles of criterion importance:
formation of the self-anchor, retention within the anchoring kernel, and
participation in the final aggregation.

The resulting weak order is
\begin{equation*}
\boxed{
\mathcal A_2
\succ
\mathcal A_1
\succ
\mathcal A_3
}.
\end{equation*}

The worked illustration demonstrates the operational transparency of
PEJWAK. Every final score can be reconstructed from the
alternative-specific self-anchor, the criterion-level kernel outputs,
the retention exponents, and the contribution coefficients. It also
shows that criterion importance acts through three visible mechanisms
rather than through one terminal multiplier. The resulting order is
therefore accompanied by an explanation of how each criterion was
contextualized and how much it contributed. Section~\ref{sec:numerical-study}
extends this transparent construction to comparative protocols,
continuous sensitivity paths, exact rank transitions, set-dependence
audits, and reproducibility checks. All displayed quantities are
generated by the supplementary implementation.

\section{Numerical Audit of PEJWAK: Comparative Behaviour, Sensitivity,
and Reproducibility}
\label{sec:numerical-study}

The numerical study converts the canonical PEJWAK architecture into a
controlled and fully auditable computational experiment. A deliberately
constructed instance is held fixed so that boundary behaviour,
criterion-level decomposition, ranking transitions, set dependence, and
rank affinity can be attributed to declared interventions rather than to
simultaneous changes in the data. The experiment therefore serves as a
mechanism stress test: it examines not only what ranking PEJWAK produces,
but how that ranking is formed, when it remains stable, where it changes,
and which effects originate inside or outside the aggregation mapping.

Four connected questions organize the analysis. First, how does PEJWAK
process heterogeneous profiles and isolated boundary zeros in comparison
with additive, multiplicative, positional, power-based, hybrid, and
external-reference architectures? Second, what exact stability intervals
and transition thresholds arise when criterion importance or the
contribution geometry changes continuously? Third, through what
mathematical pathway can a change in the alternative set alter the scores
and strict pairwise orders of surviving alternatives? Fourth, how do
directional top-weighted affinity, symmetric full-ranking association,
and deterministic reproducibility clarify the relationship between
PEJWAK and the declared benchmarks?

The numerical design is intentionally diagnostic rather than decorative.
It tests the principal capabilities established theoretically:
localization of boundary zeros, preservation of criterion identity,
complete score decomposition, exact phase identification, separation of
aggregation from preprocessing, and full computational traceability.
The resulting conclusions are exact for the declared instance,
normalization rules, importance paths, and benchmark protocols. Their
purpose is to establish how the architecture operates and to provide a
reproducible foundation for subsequent empirical applications.

\subsection{Experimental Design and Benchmark Architecture}
\label{subsec:data-architecture}

The fixed eight-alternative, five-criterion matrix is constructed as a
controlled stress-test environment rather than as an arbitrary toy
example. It contains heterogeneous profiles, exact best and worst
boundaries, isolated zeros on criteria with different importance levels,
and score configurations capable of exhibiting both stable ranking
phases and exact transitions. The supplier-selection framing supplies a
plausible decision context, while the generic criterion labels preserve
the methodological focus of the experiment. Criteria \(C_1\) and \(C_3\)
are costs, whereas \(C_2\), \(C_4\), and \(C_5\) are benefits. The fixed
criterion-importance vector is
\begin{equation*}
w=(0.30,0.25,0.20,0.15,0.10).
\end{equation*}
Table~\ref{tab:raw-matrix} is the sole raw numerical input to the main
audit.

\begin{table}[htbp]
\centering
\caption{Fixed synthetic decision matrix.}
\label{tab:raw-matrix}
\begin{tabular}{lrrrrr}
\toprule
Alternative
& $C_1$ (cost)
& $C_2$ (benefit)
& $C_3$ (cost)
& $C_4$ (benefit)
& $C_5$ (benefit)\\
\midrule
$A_1$&42&88&94&72&56\\
$A_2$&55&90&50&88&82\\
$A_3$&82&83&32&76&93\\
$A_4$&38&62&91&42&58\\
$A_5$&57&87&54&84&22\\
$A_6$&61&58&45&79&78\\
$A_7$&88&89&48&96&80\\
$A_8$&63&76&70&45&96\\
\bottomrule
\end{tabular}
\end{table}

\phantomsection
\label{subsec:numerical-normalization}

Normalization is declared preprocessing rather than part of the PEJWAK operator. The main experiment uses exact current-set min--max normalization. For each criterion, let
\[
x_j^{\min}=\min_{1\leq i\leq m}x_{ij},
\qquad
x_j^{\max}=\max_{1\leq i\leq m}x_{ij},
\]
and define the index set of nonconstant criteria by
\[
J^\ast
=
\left\{
j\in\{1,\ldots,n\}:
x_j^{\max}>x_j^{\min}
\right\}.
\]
Every criterion outside $J^\ast$ is removed before normalization because
it contains no information for discriminating among the alternatives.
For $j\in J^\ast$, benefit and cost criteria are normalized,
respectively, by
\begin{equation}
r_{ij}=
\begin{cases}
\dfrac{x_{ij}-x_j^{\min}}
{x_j^{\max}-x_j^{\min}},
& C_j \text{ is a benefit criterion},\\[9pt]
\dfrac{x_j^{\max}-x_{ij}}
{x_j^{\max}-x_j^{\min}},
& C_j \text{ is a cost criterion}.
\end{cases}
\label{eq:minmax-normalization}
\end{equation}
Thus, the best and worst observations within the declared alternative
set map exactly to one and zero.

After removal of the constant criteria, the retained criterion-importance values are renormalized as
\begin{equation*}
w_j^\ast
=
\frac{w_j}{\sum_{k\in J^\ast}w_k},
\qquad j\in J^\ast,
\end{equation*}
and every PEJWAK quantity derived from criterion importance is
recomputed from $w^\ast$. Assigning a common value, such as one, to a
constant column is not used: although that column provides no
discrimination, it could alter the endogenous row anchors and thereby
affect the subsequent nonlinear aggregation. If $J^\ast$ is empty, or
if $\sum_{k\in J^\ast}w_k=0$, the instance contains no admissibly
weighted discriminatory criterion and no ranking is produced.

No epsilon replacement is applied to the normalized matrix. In
particular, an exact boundary zero is neither deleted nor replaced by a
small positive number. The same unmodified matrix $R=(r_{ij})$ is
supplied to PEJWAK and every normalized-data benchmark. Consequently, any
zero propagation or tie formation produced by a multiplicative
benchmark is retained as part of that method's mathematical behaviour,
rather than suppressed through method-specific preprocessing. Because
the extrema are estimated from the declared alternative set, this
normalization and the resulting numerical comparisons are
current-set dependent.

This convention is compatible with the epsilon-free boundary extension established in Section~\ref{subsec:boundary-joint-continuity}. The resulting matrix is shown in Table~\ref{tab:normalized-matrix}.

\begin{table}[htbp]
\centering
\caption{Current-set min--max normalized decision matrix.}
\label{tab:normalized-matrix}
\begin{tabular}{lrrrrr}
\toprule
Alternative & $C_1$ & $C_2$ & $C_3$ & $C_4$ & $C_5$\\
\midrule
$A_1$&0.920000&0.937500&0.000000&0.555556&0.459459\\
$A_2$&0.660000&1.000000&0.709677&0.851852&0.810811\\
$A_3$&0.120000&0.781250&1.000000&0.629630&0.959459\\
$A_4$&1.000000&0.125000&0.048387&0.000000&0.486486\\
$A_5$&0.620000&0.906250&0.645161&0.777778&0.000000\\
$A_6$&0.540000&0.000000&0.790323&0.685185&0.756757\\
$A_7$&0.000000&0.968750&0.741935&1.000000&0.783784\\
$A_8$&0.500000&0.562500&0.387097&0.055556&1.000000\\
\bottomrule
\end{tabular}
\end{table}

All subsequent calculations use full-precision normalized values generated
directly from the raw matrix; the six-decimal entries in
Table~\ref{tab:normalized-matrix} are reported only for presentation.

The canonical self-anchor
\[
S_i=\sum_{j=1}^{5}w_jr_{ij}
\]
produces
\[
(S_1,\ldots,S_8)
=
(0.639654,0.798794,0.621703,0.389576,0.658261,0.498518,0.618953,0.476378).
\]
Although $S_i$ is numerically identical to the SAW score under the same weights, its role in PEJWAK is not terminal. It is the endogenous row-profile level against which each criterion score is coupled before final aggregation.

Figure~\ref{fig:normalized-profiles-anchors} juxtaposes the normalized
profiles, including five genuine boundary zeros, with their row anchors. Both
panels are generated from the reproducibility outputs.

\begin{figure}[htbp]
\centering
\begin{minipage}[t]{0.59\textwidth}
\centering
\includegraphics[width=\linewidth]{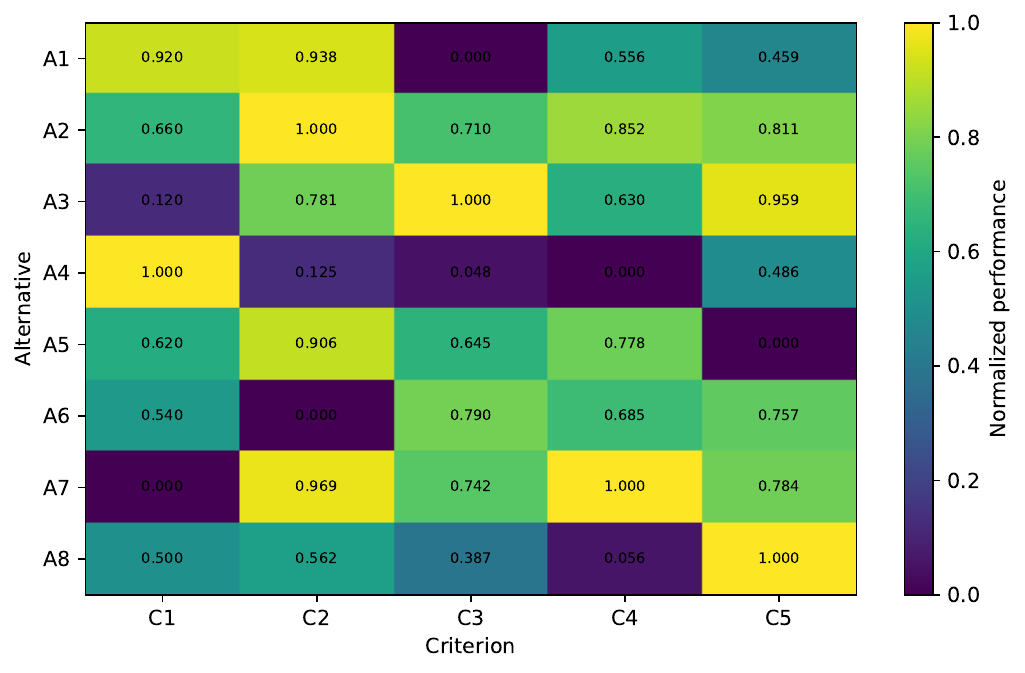}\\[
-1mm]
\small (a) Normalized performance profiles
\end{minipage}\hfill
\begin{minipage}[t]{0.38\textwidth}
\centering
\includegraphics[width=\linewidth]{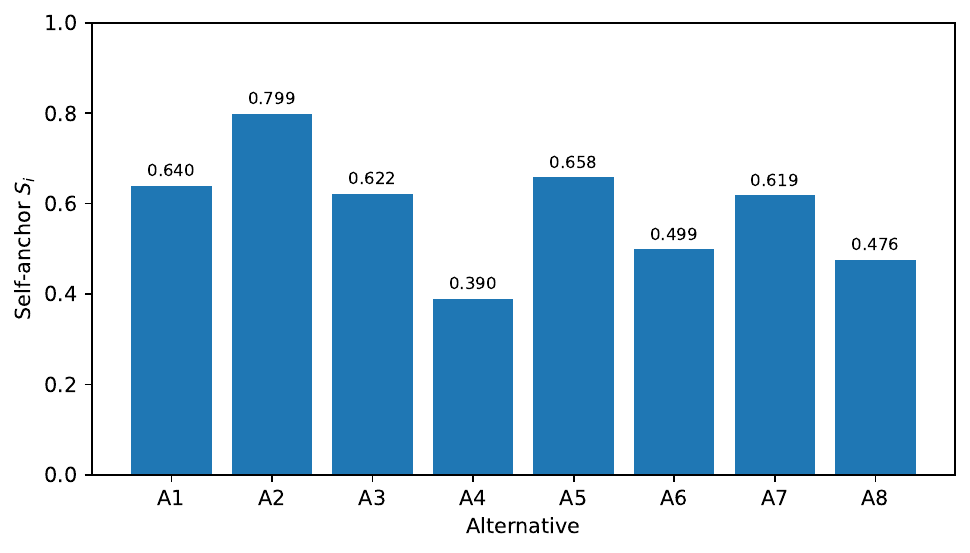}\\[-1mm]
\small (b) Corresponding self-anchors
\end{minipage}
\caption{Normalized profiles and endogenous row anchors.}
\label{fig:normalized-profiles-anchors}
\end{figure}
Because the extrema in Equation~\eqref{eq:minmax-normalization} depend on the current alternative set, the normalized coordinate system is sample-dependent. Section~\ref{subsec:rank-reversal} audits that dependence explicitly.

\phantomsection
\label{subsec:benchmark-protocols}
The benchmark set is selected for diagnostic architectural coverage
rather than numerical abundance. Each comparator isolates a distinct
location or geometry of compensation: SAW provides direct
fixed-criterion addition; WP supplies global multiplicative coupling and
zero propagation; $M_2$ introduces one global power curvature; OWA moves
influence to ordered positions; WASPAS combines completed additive and
multiplicative scores; and MACONT represents an external-reference,
multiple-representation procedure. PEJWAK is evaluated against these
methods because each contrast exposes a different part of its
profile-echoing construction.

Where their definitions permit, the methods receive the same normalized
matrix and criterion-importance vector. MACONT retains its native
raw-data procedure because replacing it with the common normalized input
would alter the method being compared. All method settings are fixed
before the numerical analysis. The comparison is therefore conducted at
the level of explicitly declared protocols: differences among SAW, WP,
WASPAS, $M_2$, OWA, and PEJWAK can be read against a common normalized
input, whereas MACONT is identified explicitly as a complete
procedure-level contrast.

For reproducibility, the MACONT implementation follows Wen et al.~\cite{Wen2020}: its sum, ratio, and min--max normalized matrices are combined with
\[
\lambda_M=\mu_M=\frac13,
\]
and their column mean defines the virtual reference. With
\[
d_{ij}=\widehat{x}_{ij}-\bar{x}_j,
\qquad
L_i=\{j:d_{ij}<0\},
\qquad
H_i=\{j:d_{ij}>0\},
\]
the implementation uses
\[
\rho_i=\sum_jw_jd_{ij},
\qquad
Q_i=
\frac{\prod_{j\in L_i}(-d_{ij})^{w_j}}
{\prod_{j\in H_i}d_{ij}^{w_j}},
\]
\[
S_{1i}
=
\delta\frac{\rho_i}{\|\rho\|_2}
+(1-\delta)\frac{Q_i}{\|Q\|_2},
\]
\[
S_{2i}
=
\vartheta\max_j(w_jd_{ij})
+(1-\vartheta)\min_j(w_jd_{ij}),
\]
and
\begin{equation*}
S_i^{\mathrm{MACONT}}
=
\frac12\left(S_{1i}+\frac{S_{2i}}{\|S_2\|_2}\right),
\qquad
\delta=\vartheta=0.5.
\end{equation*}
Empty products are assigned one and exact zero deviations are omitted from both products, in agreement with the implementation. All intermediate matrices and vectors are included in the supplementary outputs.

Table~\ref{tab:benchmark-architecture} consolidates the input scale, aggregation mechanism, fixed setting, and experimental role assigned to each comparator. Reading the protocols side by side is essential: SAW, WP, WASPAS, $M_2$, OWA, and PEJWAK share the normalized matrix, whereas MACONT retains its native raw-data pipeline, and OWA applies its coefficients to ordered positions rather than fixed criterion identities.

\begin{table}[htbp]
\centering
\footnotesize
\caption{Declared benchmark architectures and their diagnostic roles.}
\label{tab:benchmark-architecture}
\begin{tabularx}{\textwidth}{lXXXX}
\toprule
Method & Numerical input & Aggregation architecture & Fixed setting & Experimental role\\
\midrule
SAW & $R,w$ & Fixed-criterion arithmetic sum & None & Linear-compensation baseline and numerical anchor reference\\
WP & $R,w$ & Fixed-criterion weighted product & None & Multiplicative, zero-propagating contrast\\
WASPAS & $R,w$ & Post-aggregation blend & $\lambda=0.5$ & Hybrid whole-score contrast\\
$M_2$ & $R,w$ & Global power curvature & $p=2$ & Global-exponent contrast\\
OWA & Ordered rows of $R$ and $v$ & Positional arithmetic aggregation & $v=(0.30,0.25,0.20,0.15,0.10)$ & Criterion-identity contrast\\
MACONT & Raw matrix, directions, $w$ & Three normalization channels and a virtual column reference & $\lambda_M=\mu_M=1/3$, $\delta=\vartheta=0.5$ & External-reference multiple-channel contrast\\
PEJWAK
&
\(R,w\)
&
Criterion-level geometric coupling to an endogenous self-anchor
&
Canonical square-root contribution rule
&
Profile-echoing aggregation, localized zero handling, and complete
criterion-level decomposition
\\
\bottomrule
\end{tabularx}
\end{table}

\subsection{Comparative Behaviour and Rank Structure}
\label{subsec:scores-rankings}

Tables~\ref{tab:benchmark-scores} and~\ref{tab:benchmark-rankings} report method-specific scores and descending ranks. Larger values are preferred within each method, and tied scores receive average ranks. Because the score scales are method-specific, numerical values are interpreted within columns; the rankings provide the principal common basis for comparison.

\begin{table}[htbp]
\centering
\scriptsize
\caption{Method-specific scores under declared protocols.}
\label{tab:benchmark-scores}
\resizebox{\textwidth}{!}{%
\begin{tabular}{lrrrrrrr}
\toprule
Alternative & SAW & WP & WASPAS & $M_2$ & OWA & MACONT & PEJWAK\\
\midrule
$A_1$&0.639654&0.000000&0.319827&0.735563&0.691280&0.098591&0.537131\\
$A_2$&0.798794&0.787992&0.793393&0.809937&0.847577&0.580461&0.795951\\
$A_3$&0.621703&0.462399&0.542051&0.713042&0.802559&0.075450&0.586467\\
$A_4$&0.389576&0.000000&0.194788&0.572749&0.453880&-0.010719&0.304980\\
$A_5$&0.658261&0.000000&0.329131&0.703299&0.688352&0.143098&0.575516\\
$A_6$&0.498518&0.000000&0.249259&0.583174&0.644323&-0.333323&0.405291\\
$A_7$&0.618953&0.000000&0.309476&0.745751&0.810235&-0.191963&0.496340\\
$A_8$&0.476378&0.377133&0.426755&0.533417&0.604245&-0.212864&0.460920\\
\bottomrule
\end{tabular}}
\end{table}

\begin{table}[htbp]
\centering
\small
\caption{Comparative rankings under declared protocols.}
\label{tab:benchmark-rankings}
\begin{tabular}{lrrrrrrr}
\toprule
Alternative & SAW & WP & WASPAS & $M_2$ & OWA & MACONT & PEJWAK\\
\midrule
$A_1$&3&6&5&3&4&3&4\\
$A_2$&1&1&1&1&1&1&1\\
$A_3$&4&2&2&4&3&4&2\\
$A_4$&8&6&8&7&8&5&8\\
$A_5$&2&6&4&5&5&2&3\\
$A_6$&6&6&7&6&6&8&7\\
$A_7$&5&6&6&2&2&6&5\\
$A_8$&7&3&3&8&7&7&6\\
\bottomrule
\end{tabular}
\end{table}

The comparative ranking table immediately reveals both agreement and
architectural separation. Alternative \(A_2\) is ranked first by every
declared protocol, reflecting a profile that remains strong without an
exact boundary zero. The more diagnostic result concerns the remaining
alternatives, especially those containing zeros.

WP maps five alternatives to the identical score zero because each
contains a zero on a positively important criterion. Its global product
therefore loses all discrimination among \(A_1\), \(A_4\), \(A_5\),
\(A_6\), and \(A_7\). PEJWAK processes the same observations locally.
For \(r_{ij}=0\) and \(w_j>0\),
\[
G_{ij}=0,
\qquad
T_{ij}=0,
\]
while the remaining criterion contributions continue to operate.
Consequently, the five zero-bearing alternatives remain fully
distinguishable:
\[
A_5\succ A_1\succ A_7\succ A_6\succ A_4.
\]

The contrast between \(A_7\) and \(A_5\) demonstrates why criterion
identity matters. Alternative \(A_7\) has four large values but its zero
occurs on \(C_1\), the most important criterion, so the entire
criterion-\(C_1\) contribution is removed. Alternative \(A_5\) has its
zero on \(C_5\), the least important criterion, and therefore retains a
substantially stronger complete evaluation. This is a structural
advantage over the global weighted product: PEJWAK preserves geometric
sensitivity to an exact weakness while preventing that weakness from
destroying all information supplied by the remaining criteria.

Figure~\ref{fig:pejwak-contributions} converts each aggregate score into
an explanation. Every stacked bar reconstructs
\(P_i=\sum_jT_{ij}\), while the self-anchor marks the profile level
against which the criterion terms were formed. The figure makes visible
which criteria retain, lose, or moderate their contributions and shows
that the final PEJWAK order is completely traceable to criterion-level
operations. The difference between \(P_i\) and \(S_i\) is not treated as
a penalty; its value lies in revealing how anchoring transforms the
profile before final participation.

\begin{figure}[htbp]
\centering
\includegraphics[width=0.90\textwidth]{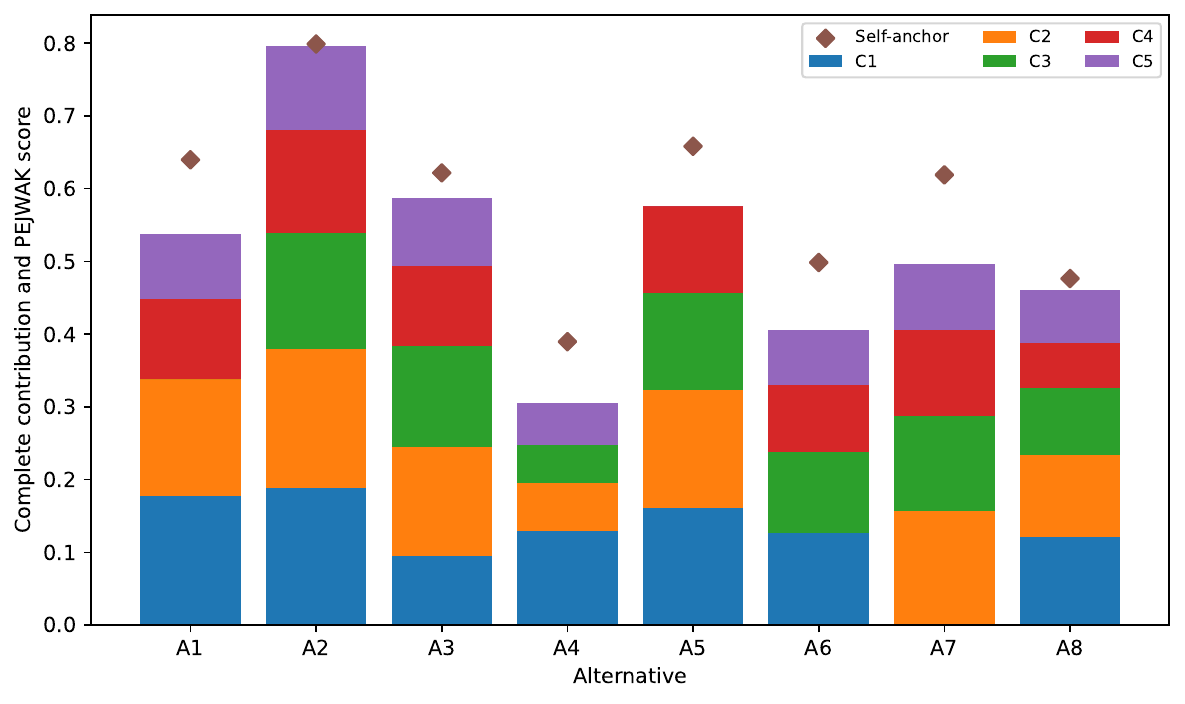}
\caption{Criterion-level decomposition of canonical PEJWAK scores.}
\label{fig:pejwak-contributions}
\end{figure}

\phantomsection
\label{subsec:comparative-rank-structure}

The benchmark methods are categorical alternatives, not points on a continuous scale. A connected-line display would therefore suggest false continuity and obscure the five-way WP tie. Figure~\ref{fig:rank-structure} instead combines a rank matrix with each alternative's observed rank span. The first panel makes every rank explicit; the second summarizes disagreement without assigning metric meaning to the spacing between methods.

\begin{figure}[htbp]
\centering
\begin{minipage}[t]{0.56\textwidth}
\centering
\includegraphics[width=\linewidth]{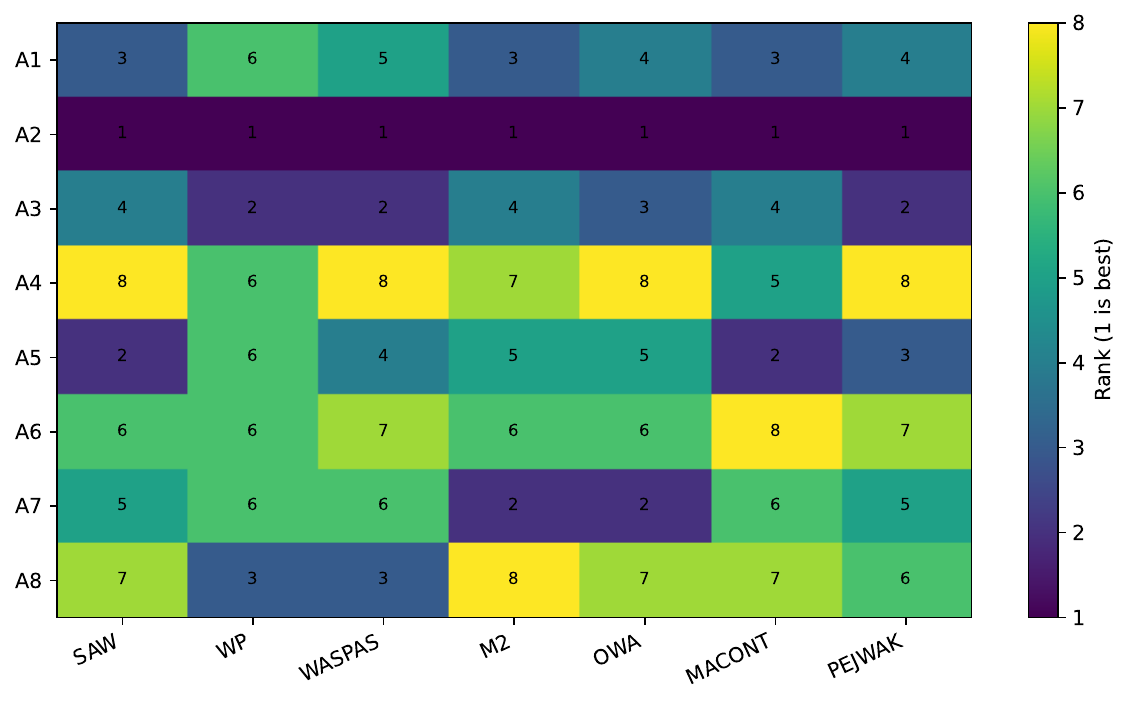}\\[
-1mm]
\small (a) Complete comparative rank matrix
\end{minipage}\hfill
\begin{minipage}[t]{0.41\textwidth}
\centering
\includegraphics[width=\linewidth]{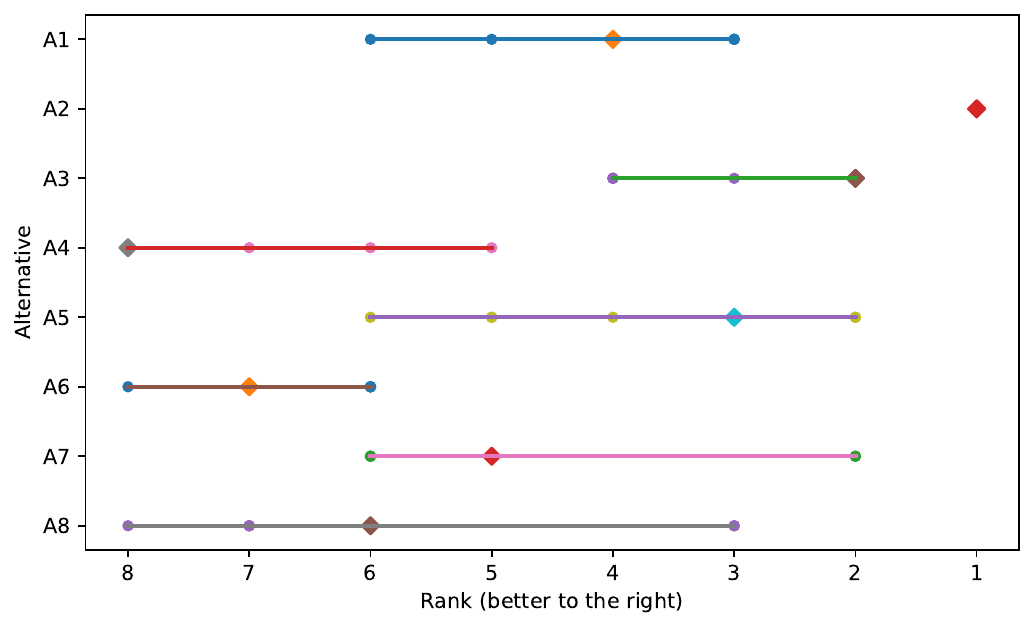}\\[-1mm]
\small (b) Best-to-worst rank span
\end{minipage}
\caption{Comparative rank structure across declared protocols.}
\label{fig:rank-structure}
\end{figure}
The unanimous rank of $A_2$ gives a zero-width span. Alternative $A_8$ has the widest span, from rank three to rank eight, and $A_7$ ranges from rank two to rank six. These ranges show sensitivity to aggregation architecture, not uncertainty about an externally correct position.

\subsection{Exact Sensitivity and Rank-Phase Geometry}
\label{subsec:continuous-weight-sensitivity}
Criterion-importance sensitivity is formulated as an exact
phase-identification problem rather than as a collection of isolated
percentage perturbations. Each audited path varies one importance value
continuously and rescales the remaining values proportionally, preserving
their relative structure and the unit-sum requirement. The analysis then
locates every detected zero of each pairwise score difference and
partitions the path into maximal open intervals with constant complete
rankings. It therefore produces explicit stability certificates and
exact transition boundaries rather than a small set of scenario
snapshots. Let
\[
w^{(0)}=(0.30,0.25,0.20,0.15,0.10).
\]
When criterion $C_k$ is varied, its weight is set to $t$ and the remaining weights are proportionally rescaled:
\begin{equation}
w_k^{(k)}(t)=t,
\qquad
w_j^{(k)}(t)
=
\frac{1-t}{1-w_k^{(0)}}w_j^{(0)},
\quad j\neq k.
\label{eq:proportional-weight-path}
\end{equation}
This construction preserves unit sum and the relative proportions among nonvaried criteria. The audited paths are
\[
w^{(1)}(t)
=
\left(t,\frac{5}{14}(1-t),\frac{2}{7}(1-t),\frac{3}{14}(1-t),\frac{1}{7}(1-t)\right),
\]
\[
w^{(3)}(t)
=
\left(\frac{3}{8}(1-t),\frac{5}{16}(1-t),t,\frac{3}{16}(1-t),\frac{1}{8}(1-t)\right),
\]
with
\[
t\in[0.05,0.60].
\]
The interval is a declared interior design-audit range rather than a probabilistic model of weights.

For each path, define
\[
P_i^{(k)}(t)=P_{w^{(k)}(t)}^{\mathrm{PEJWAK}}(r_i)
\]
and
\[
\Delta_{ab}^{(k)}(t)=P_a^{(k)}(t)-P_b^{(k)}(t).
\]
Joint continuity implies that a strict pairwise order can change only at a root of $\Delta_{ab}^{(k)}(t)=0$. A dense deterministic grid is used to bracket every detected sign-changing intersection, and each bracketed root is refined to absolute tolerance $10^{-12}$. The complete ranking is then evaluated on every open interval between successive identified transitions.

Tables~\ref{tab:w1-transitions} and~\ref{tab:w3-transitions} list every detected crossing on the two audited paths, including the pair whose order changes on each side of the critical value. The $w_1$ path contains fifteen transitions, whereas the $w_3$ path contains ten; these counts describe the declared interval and data rather than a general difference in criterion sensitivity.

\begin{table}[htbp]
\centering
\scriptsize
\caption{PEJWAK transitions along the $w_1$ path.}
\label{tab:w1-transitions}
\begin{tabular}{rll}
\toprule
$w_1^*$ & Order before & Order after\\
\midrule
0.068748 & $A_8\succ A_1$ & $A_1\succ A_8$\\
0.214087 & $A_7\succ A_5$ & $A_5\succ A_7$\\
0.269595 & $A_7\succ A_1$ & $A_1\succ A_7$\\
0.313259 & $A_3\succ A_5$ & $A_5\succ A_3$\\
0.336390 & $A_7\succ A_8$ & $A_8\succ A_7$\\
0.338879 & $A_3\succ A_1$ & $A_1\succ A_3$\\
0.385343 & $A_5\succ A_1$ & $A_1\succ A_5$\\
0.392594 & $A_7\succ A_6$ & $A_6\succ A_7$\\
0.409119 & $A_7\succ A_4$ & $A_4\succ A_7$\\
0.428813 & $A_6\succ A_4$ & $A_4\succ A_6$\\
0.439839 & $A_3\succ A_8$ & $A_8\succ A_3$\\
0.465102 & $A_3\succ A_4$ & $A_4\succ A_3$\\
0.490777 & $A_8\succ A_4$ & $A_4\succ A_8$\\
0.498549 & $A_3\succ A_6$ & $A_6\succ A_3$\\
0.592997 & $A_5\succ A_4$ & $A_4\succ A_5$\\
\bottomrule
\end{tabular}
\end{table}

\begin{table}[htbp]
\centering
\scriptsize
\caption{PEJWAK transitions along the $w_3$ path.}
\label{tab:w3-transitions}
\begin{tabular}{rll}
\toprule
$w_3^*$ & Order before & Order after\\
\midrule
0.066548 & $A_4\succ A_6$ & $A_6\succ A_4$\\
0.070686 & $A_8\succ A_7$ & $A_7\succ A_8$\\
0.159396 & $A_1\succ A_5$ & $A_5\succ A_1$\\
0.165694 & $A_1\succ A_3$ & $A_3\succ A_1$\\
0.178084 & $A_5\succ A_3$ & $A_3\succ A_5$\\
0.238244 & $A_1\succ A_7$ & $A_7\succ A_1$\\
0.299713 & $A_1\succ A_8$ & $A_8\succ A_1$\\
0.306594 & $A_1\succ A_6$ & $A_6\succ A_1$\\
0.316912 & $A_8\succ A_6$ & $A_6\succ A_8$\\
0.506789 & $A_2\succ A_3$ & $A_3\succ A_2$\\
\bottomrule
\end{tabular}
\end{table}

The two paths establish complementary properties of the canonical
mechanism. Around the baseline vector, the complete PEJWAK order possesses
the exact local stability certificates
\begin{equation}
0.269595<w_1<0.313259,
\label{eq:w1-local-interval}
\end{equation}
and
\begin{equation}
0.178084<w_3<0.238244.
\label{eq:w3-local-interval}
\end{equation}
Throughout either interval, all eight alternatives retain the complete
order
\[
A_2\succ A_3\succ A_5\succ A_1\succ A_7\succ A_8\succ A_6\succ A_4.
\]
Beyond these phases, Tables~\ref{tab:w1-transitions}
and~\ref{tab:w3-transitions} identify precisely which pair changes order
and at what criterion-importance value. At $w_3=0.506789$, $A_3$
overtakes $A_2$ and becomes the leading alternative, consistently with
$r_{33}=1$ and $r_{23}=0.709677$. This transition is not driven by one
terminal coefficient alone: changing $w_3$ jointly modifies the
self-anchor, the retention exponents, and the contribution coefficients.
PEJWAK therefore reveals both local ranking stability and the exact
architectural boundaries at which that stability ends.

Figure~\ref{fig:weight-paths} places these tabulated crossings on the continuous score and rank trajectories. Panels (a) and (c) show that the scores vary continuously, whereas panels (b) and (d) show the piecewise-constant rankings and locate the baseline values within their corresponding invariant phases.

\begin{figure}[htbp]
\centering
\begin{minipage}[t]{0.48\textwidth}
\centering
\includegraphics[width=\linewidth]{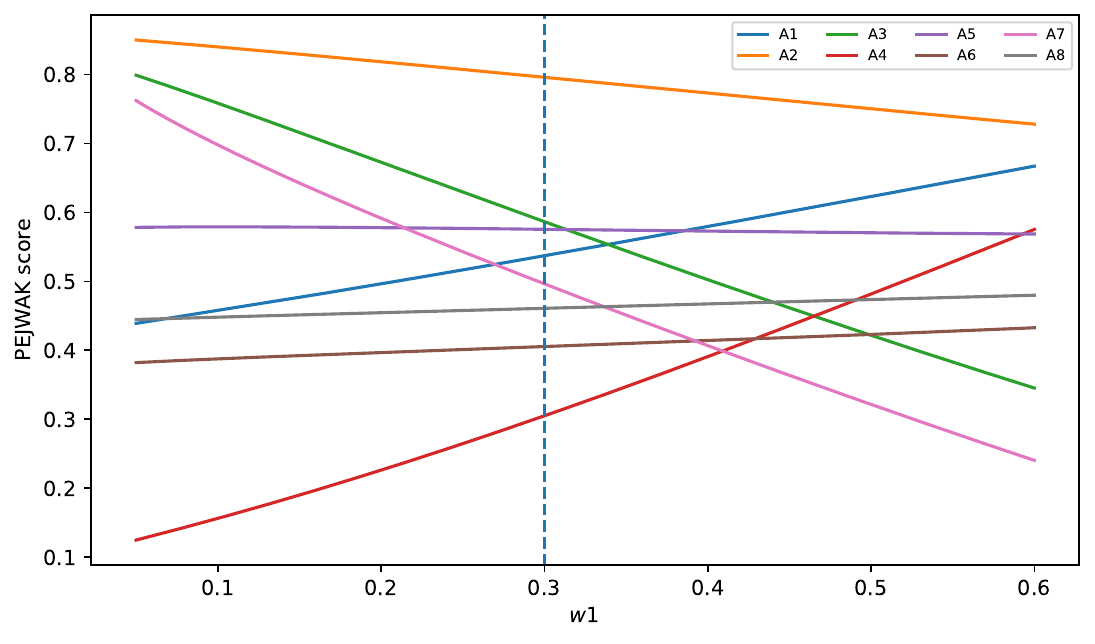}\\[
-1mm]
\small (a) Score paths under $w_1=t$
\end{minipage}\hfill
\begin{minipage}[t]{0.48\textwidth}
\centering
\includegraphics[width=\linewidth]{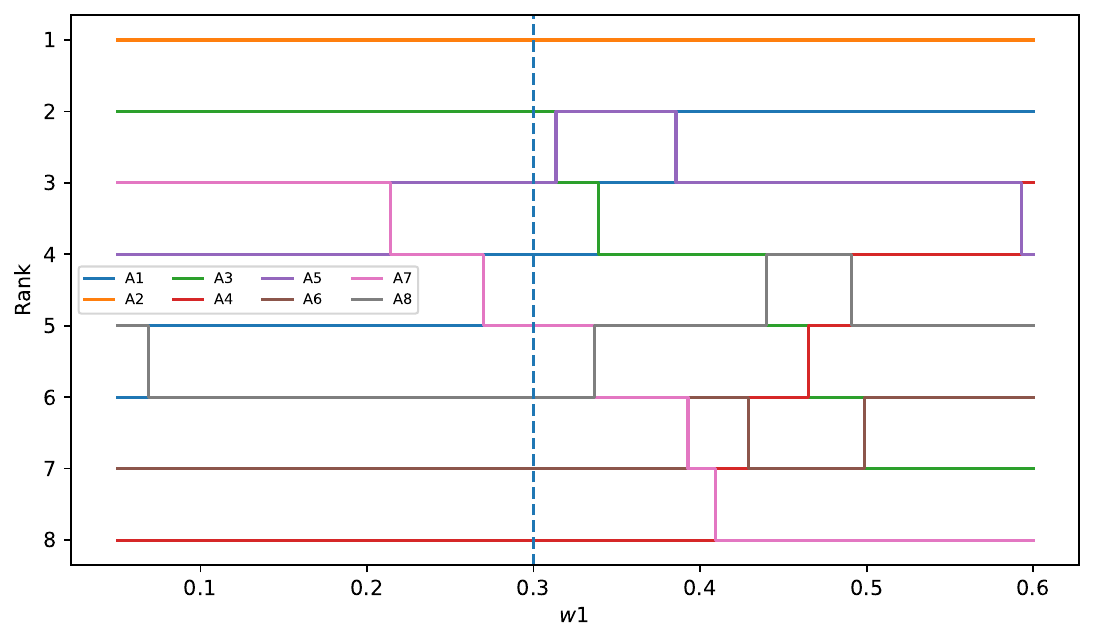}\\[-1mm]
\small (b) Piecewise-constant rank paths under $w_1=t$
\end{minipage}
\vspace{2mm}
\begin{minipage}[t]{0.48\textwidth}
\centering
\includegraphics[width=\linewidth]{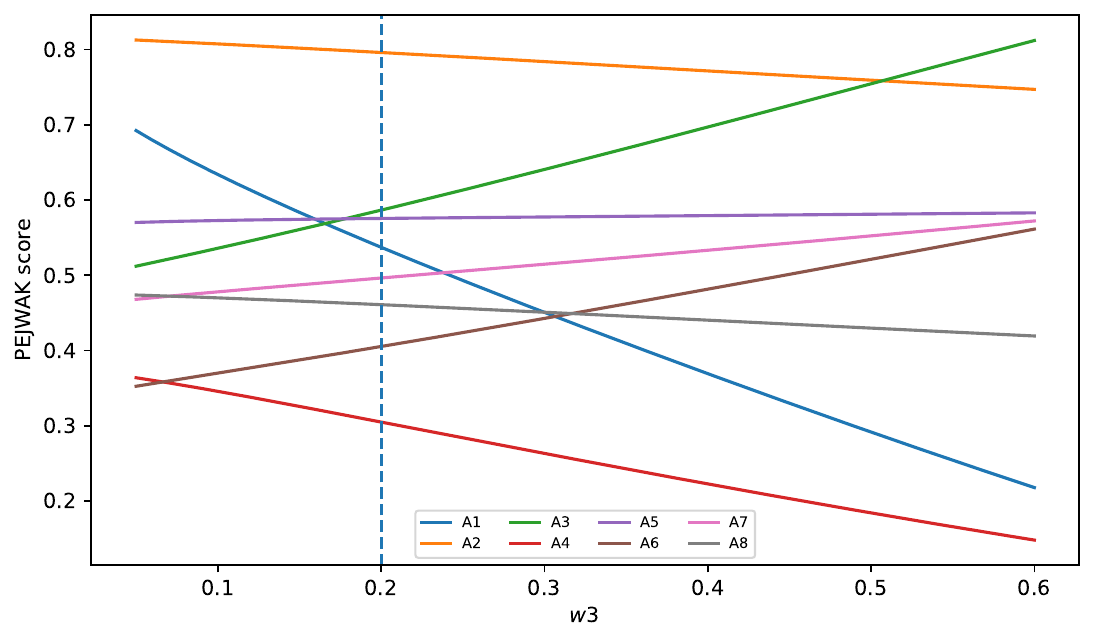}\\[-1mm]
\small (c) Score paths under $w_3=t$
\end{minipage}\hfill
\begin{minipage}[t]{0.48\textwidth}
\centering
\includegraphics[width=\linewidth]{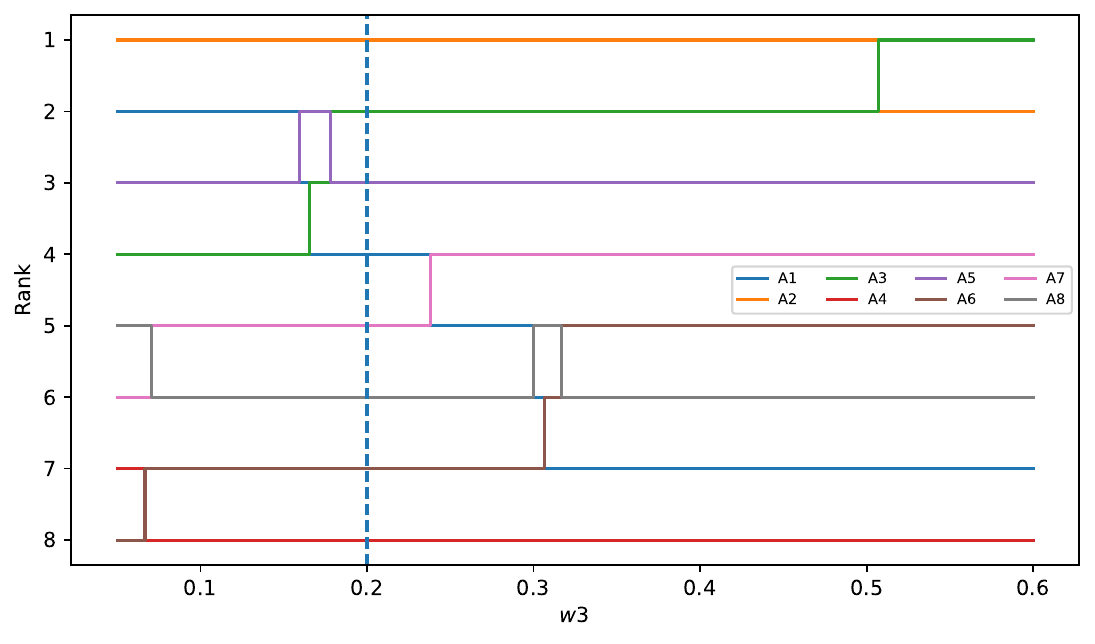}\\[-1mm]
\small (d) Piecewise-constant rank paths under $w_3=t$
\end{minipage}
\caption{Continuous PEJWAK score and rank paths.}
\label{fig:weight-paths}
\end{figure}
The dashed vertical lines mark the baseline weights. The score trajectories are continuous and the rank trajectories are
piecewise constant, changing only at the certified thresholds. The
baseline phases therefore demonstrate that local stability and a richer
global transition geometry coexist within the same architecture.

\phantomsection
\label{subsec:contribution-rule-audit}

A distinct question concerns the contribution transformation rather than
the criterion-importance vector itself. The second sensitivity audit
varies the contribution rule while holding the normalized matrix, the
self-anchors, the retention exponents, and the matrix of kernel values
$G=[G_{ij}]$ fixed. It applies the escort geometry established in
Section~\ref{subsec:escort-geometry},
\[
\varphi(q)=\operatorname{softmax}(q\log w),
\qquad q\in[0,\infty),
\]
to the numerical instance. The canonical point is $q=1/2$;
$0\leq q<1$ compresses differences among positive importance values,
$q=1$ transfers them directly, and $q>1$ provides a concentration
stress test.

For each pair,
\[
\Delta_{ab}(q)
=
\frac{\sum_j w_j^q(G_{aj}-G_{bj})}{\sum_kw_k^q}.
\]
The numerator is the exponential polynomial in Equation~\eqref{eq:escort-exponential-polynomial}. Root isolation uses the derivative structure from Theorem~\ref{thm:finite-escort-phases}, followed by deterministic refinement and a dominance certificate that excludes further roots beyond the final search interval.

Table~\ref{tab:escort-transitions} records the complete sequence of twelve transverse crossings over $q\in[0,\infty)$ and identifies the pairwise order on both sides of each boundary. It therefore supplies the numerical partition underlying the thirteen constant-ranking phases discussed below.

\begin{table}[htbp]
\centering
\scriptsize
\caption{Contribution-rule transitions for $q\in[0,\infty)$.}
\label{tab:escort-transitions}
\begin{tabular}{rlll}
\toprule
$q^*$ & Pair & Order before & Order after\\
\midrule
0.621750 & $A_3/A_5$ & $A_3\succ A_5$ & $A_5\succ A_3$\\
1.013628 & $A_7/A_8$ & $A_7\succ A_8$ & $A_8\succ A_7$\\
1.556367 & $A_1/A_3$ & $A_3\succ A_1$ & $A_1\succ A_3$\\
2.460517 & $A_6/A_7$ & $A_7\succ A_6$ & $A_6\succ A_7$\\
2.542918 & $A_4/A_7$ & $A_7\succ A_4$ & $A_4\succ A_7$\\
2.674756 & $A_4/A_6$ & $A_6\succ A_4$ & $A_4\succ A_6$\\
5.008459 & $A_3/A_8$ & $A_3\succ A_8$ & $A_8\succ A_3$\\
5.746905 & $A_1/A_5$ & $A_5\succ A_1$ & $A_1\succ A_5$\\
6.833078 & $A_3/A_4$ & $A_3\succ A_4$ & $A_4\succ A_3$\\
9.010732 & $A_3/A_6$ & $A_3\succ A_6$ & $A_6\succ A_3$\\
10.426339 & $A_4/A_8$ & $A_8\succ A_4$ & $A_4\succ A_8$\\
15.887199 & $A_6/A_8$ & $A_8\succ A_6$ & $A_6\succ A_8$\\
\bottomrule
\end{tabular}
\end{table}

The first transition is analytically characterized by
\begin{equation*}
q_{35}^*
=
0.6217503447\ldots,
\qquad
P_3(q_{35}^*)=P_5(q_{35}^*)=0.583359816\ldots.
\end{equation*}
With $d_j=G_{3j}-G_{5j}$ and $w_5=0.10$ as the reference weight, the exact crossing equation is
\begin{equation}
d_1 3^q
+d_2\left(\frac52\right)^q
+d_3 2^q
+d_4\left(\frac32\right)^q
+d_5
=0.
\label{eq:q35-crossing}
\end{equation}
The sign pattern is
\[
d_1<0,\quad d_2<0,\quad d_3>0,\quad d_4<0,\quad d_5>0.
\]
Thus $A_3$ has its largest kernel advantage on the least-weighted criterion $C_5$, whereas $A_5$ is stronger on higher-weight criteria, especially $C_1$ and $C_2$. Increasing $q$ progressively shifts relative contribution toward those higher-weight criteria. The logarithmic balance function obtained by separating the positive and negative sides of Equation~\eqref{eq:q35-crossing} is strictly decreasing on $[0,1]$, proving that this crossing is unique in the compression interval. Moreover,
\[
\Delta_{35}'(q_{35}^*)=-0.088039\ldots\neq0,
\]
so the transition is transverse rather than tangential.

All twelve observed crossings are transverse. They partition the half-line into thirteen constant-ranking phases. The canonical value $q=1/2$ belongs to the first phase and preserves the baseline order. Beyond $q=1$, additional transitions reveal the effect of progressively concentrating the contribution coefficients. Because $w_1=0.30$ is the unique maximum weight,
\[
\lim_{q\to\infty}P_i(q)=G_{i1},
\]
and the limiting order is
\begin{equation*}
A_2\succ A_1\succ A_5\succ A_4\succ A_6\succ A_8\succ A_3\succ A_7.
\end{equation*}
This limit remains profile-echoing: $G_{i1}$ is not the raw first-criterion score but the output of the anchoring kernel after criterion $C_1$ has been coupled to the alternative-specific self-anchor.

For display, the half-line is compactified by
$\tau=q/(1+q)\in[0,1)$. The panels trace the contribution coefficients, the
thirteen rank phases, and the first transverse crossing. This path is an
analytical design audit, not a fitted or decision-maker-controlled parameter.

Figure~\ref{fig:escort-phase-geometry} provides the visual counterpart to Table~\ref{tab:escort-transitions}: panel (a) shows the redistribution of contribution coefficients, panel (b) displays all thirteen phases on the compactified half-line, and panel (c) verifies that the first $A_3/A_5$ crossing changes sign rather than merely touching zero.

\begin{figure}[htbp]
\centering
\begin{minipage}[t]{0.48\textwidth}
\centering
\includegraphics[width=\linewidth]{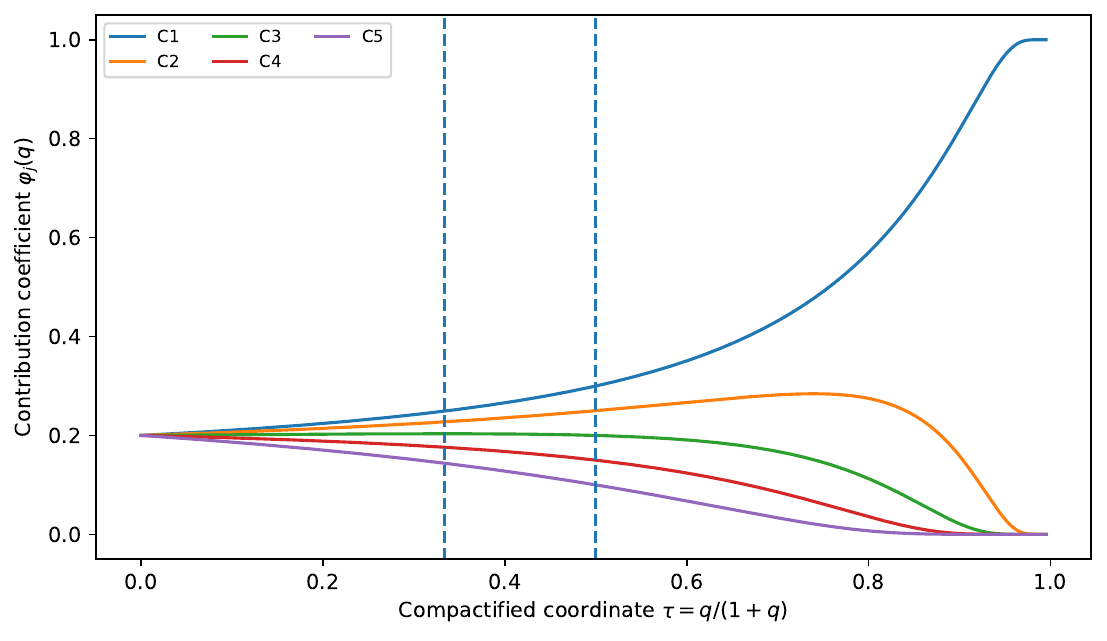}\\[
-1mm]
\small (a) Escort contribution coefficients
\end{minipage}\hfill
\begin{minipage}[t]{0.48\textwidth}
\centering
\includegraphics[width=\linewidth]{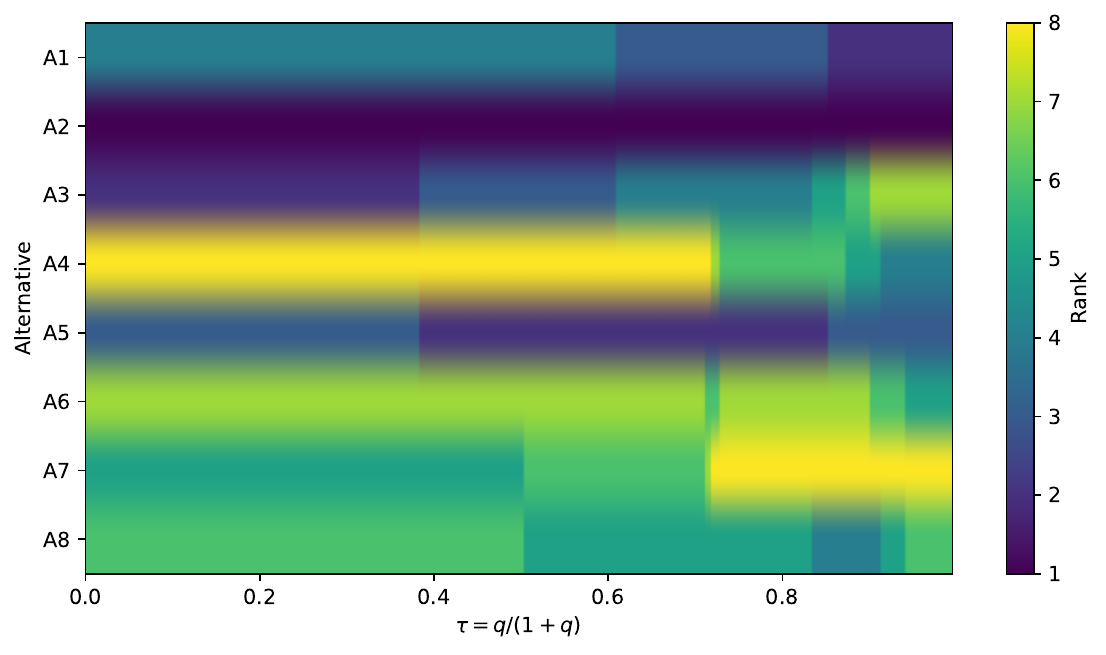}\\[-1mm]
\small (b) Rank phases on the compactified half-line
\end{minipage}
\vspace{2mm}
\begin{minipage}[t]{0.62\textwidth}
\centering
\includegraphics[width=\linewidth]{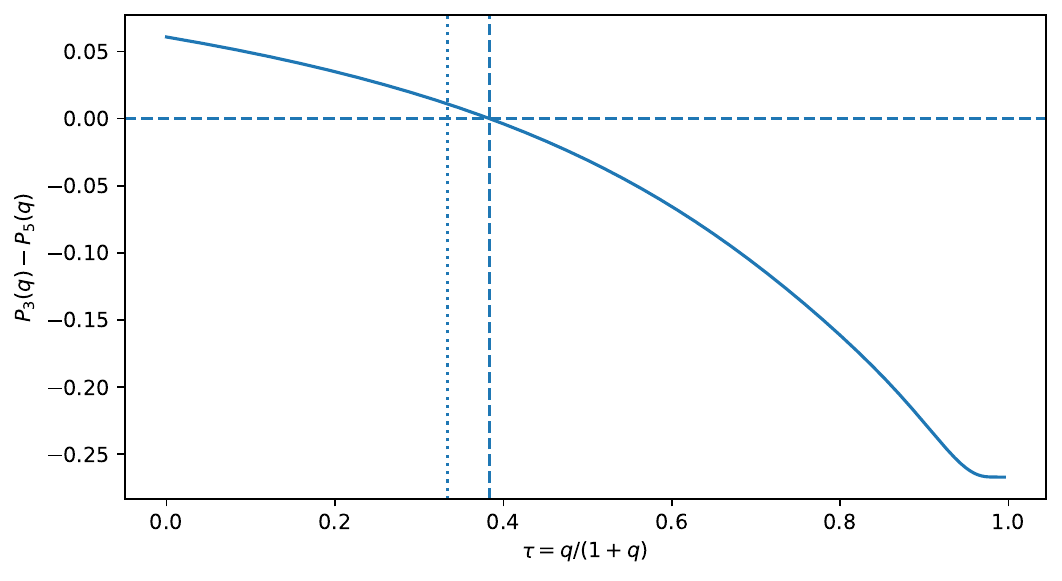}\\[-1mm]
\small (c) Transverse $A_3/A_5$ score-gap crossing
\end{minipage}
\caption{Contribution-rule phase geometry.}
\label{fig:escort-phase-geometry}
\end{figure}
\subsection{Set Dependence, Rank Affinity, and Reproducibility}
\label{subsec:rank-reversal}
The set-dependence audit asks a sharper question than ordinary weight
sensitivity. Criterion importance, criterion directions, and the PEJWAK
scoring rule remain fixed; only the alternative set changes. The audit
then determines whether the strict order of any surviving pair reverses
and identifies the exact computational pathway responsible for the
change. This design separates genuine parameter sensitivity from
dependence transmitted by current-set normalization.

\begin{lemma}[Criterion-wise affine transport under current-set normalization]
\label{lem:normalization-transport}
Let $\mathcal A$ and $\mathcal A'$ denote the baseline and modified
alternative sets, respectively. For criterion $C_j$, let
$(m_j,M_j)$ and $(m'_j,M'_j)$ be the corresponding current-set bounds.
Suppose that $C_j$ remains nonconstant in both sets:
\[
M_j>m_j,
\qquad
M'_j>m'_j.
\]
Then, for every surviving alternative
$A_i\in\mathcal A\cap\mathcal A'$, the min--max normalized values satisfy
\begin{equation}
r'_{ij}=a_jr_{ij}+b_j,
\qquad
a_j=\frac{M_j-m_j}{M'_j-m'_j}>0,
\label{eq:normalization-transport}
\end{equation}
where
\[
b_j=
\begin{cases}
\dfrac{m_j-m'_j}{M'_j-m'_j},
& C_j \text{ is a benefit criterion},\\[9pt]
\dfrac{M'_j-M_j}{M'_j-m'_j},
& C_j \text{ is a cost criterion}.
\end{cases}
\]
Consequently, a change in the alternative set induces a
criterion-specific positive affine transport of every surviving
normalized profile.
\end{lemma}

\begin{proof}
For a benefit criterion, Equation~\eqref{eq:minmax-normalization} gives
\[
r_{ij}=\frac{x_{ij}-m_j}{M_j-m_j},
\qquad
x_{ij}=m_j+(M_j-m_j)r_{ij}.
\]
Substitution into the normalization based on $\mathcal A'$ yields
\[
\begin{aligned}
r'_{ij}
&=\frac{x_{ij}-m'_j}{M'_j-m'_j}\\
&=\frac{M_j-m_j}{M'_j-m'_j}r_{ij}
+\frac{m_j-m'_j}{M'_j-m'_j}.
\end{aligned}
\]
For a cost criterion,
\[
r_{ij}=\frac{M_j-x_{ij}}{M_j-m_j},
\qquad
x_{ij}=M_j-(M_j-m_j)r_{ij},
\]
and hence
\[
\begin{aligned}
r'_{ij}
&=\frac{M'_j-x_{ij}}{M'_j-m'_j}\\
&=\frac{M_j-m_j}{M'_j-m'_j}r_{ij}
+\frac{M'_j-M_j}{M'_j-m'_j}.
\end{aligned}
\]
Since both criterion ranges are strictly positive, $a_j>0$, which
completes the proof.
\end{proof}

In particular, for any two distinct surviving alternatives
$A_a,A_b\in\mathcal A\cap\mathcal A'$,
\[
r'_{aj}-r'_{bj}
=
a_j\bigl(r_{aj}-r_{bj}\bigr).
\]
Because $a_j>0$, current-set normalization preserves the strict
within-criterion order of every surviving pair. Any final rank reversal
must therefore arise through the interaction of the distinct
criterion-specific transports with the subsequent aggregation
architecture, rather than through an order reversal within an
individual criterion. Also, define their directed score
difference under the alternative set $\mathcal A$ by
\[
\Delta_{ab}^{\mathcal A}
=
P_a^{\mathcal A}-P_b^{\mathcal A}.
\]
Thus, $\Delta_{ab}^{\mathcal A}>0$ means that $A_a$ strictly outranks
$A_b$, whereas $\Delta_{ab}^{\mathcal A}<0$ means that $A_b$ strictly
outranks $A_a$. A strict pairwise rank reversal occurs if and only if
\begin{equation}
\Delta_{ab}^{\mathcal A}
\Delta_{ab}^{\mathcal A'}<0.
\label{eq:strict-rank-reversal}
\end{equation}
Accordingly, a transition to or from a tie, for which either difference
is zero, is recorded separately and is not classified as a strict
pairwise reversal.

To operationalize the preceding criterion, the audit considers two independent 
and complementary classes of set perturbation. First, each of the eight baseline 
alternatives is removed individually from the complete baseline set, and the scores 
and strict pairwise orders of all surviving alternatives are recomputed. Second, 
in a separate experiment and without removing any baseline alternative, the 
artificial alternative
\[
A_9=(95,50,100,35,15)
\]
is added to the set. Given the declared criterion directions, \(A_9\) is strictly 
dominated by every baseline alternative on all five criteria. This construction 
tests whether the introduction of an alternative that cannot outperform any 
original option can nevertheless alter their ordering solely by expanding the 
bounds used in current-set normalization.

Table~\ref{tab:rank-reversal-audit} reports the results of these nine independent 
experiments, comprising eight deletions and one dominated addition. For each 
experiment, it identifies the bounds altered by current-set normalization, the 
number of strict pairwise reversals, and the affected pairs. This organization 
distinguishes a mere change in normalization bounds from the stronger event of an 
actual sign reversal in a pairwise score difference.

\begin{table}[htbp]
\centering
\footnotesize
\renewcommand{\arraystretch}{1.18}
\caption{Set-induced rank-reversal audit.}
\label{tab:rank-reversal-audit}

\begin{tabularx}{\textwidth}
{
>{\raggedright\arraybackslash}p{2.35cm}
>{\raggedright\arraybackslash}X
>{\centering\arraybackslash}p{1.55cm}
>{\raggedright\arraybackslash}p{3.15cm}
}
\toprule
Experiment
&
Changed current-set bounds
&
Strict reversals
&
Affected pairs
\\
\midrule

Delete $A_1$
&
$C_3:[32,94]\to[32,91]$
&
0
&
\textemdash
\\

\addlinespace[2pt]
Delete $A_2$
&
$C_2:[58,90]\to[58,89]$
&
0
&
\textemdash
\\

\addlinespace[2pt]
Delete $A_3$
&
$C_3:[32,94]\to[45,94]$
&
0
&
\textemdash
\\

\addlinespace[2pt]
Delete $A_4$
&
$C_1:[38,88]\to[42,88]$;
$C_4:[42,96]\to[45,96]$
&
2
&
$A_3\leftrightarrow A_5$;
$A_6\leftrightarrow A_8$
\\

\addlinespace[2pt]
Delete $A_5$
&
$C_5:[22,96]\to[56,96]$
&
2
&
$A_1\leftrightarrow A_7$;
$A_1\leftrightarrow A_8$
\\

\addlinespace[2pt]
Delete $A_6$
&
$C_2:[58,90]\to[62,90]$
&
0
&
\textemdash
\\

\addlinespace[2pt]
Delete $A_7$
&
$C_1:[38,88]\to[38,82]$;
$C_4:[42,96]\to[42,88]$
&
2
&
$A_1\leftrightarrow A_3$;
$A_3\leftrightarrow A_5$
\\

\addlinespace[2pt]
Delete $A_8$
&
$C_5:[22,96]\to[22,93]$
&
0
&
\textemdash
\\

\midrule
Add strictly dominated $A_9$
&
All five unfavorable bounds
&
3
&
$A_1\leftrightarrow A_3$;
$A_3\leftrightarrow A_5$;
$A_6\leftrightarrow A_8$
\\

\bottomrule
\end{tabularx}

\vspace{3pt}

\begin{minipage}{0.98\textwidth}
\footnotesize
\textit{Note:}
Upon adding $A_9$, the unfavorable bounds expand as follows:
$C_1:[38,88]\to[38,95]$,
$C_2:[58,90]\to[50,90]$,
$C_3:[32,94]\to[32,100]$,
$C_4:[42,96]\to[35,96]$, and
$C_5:[22,96]\to[15,96]$.
\end{minipage}
\end{table}

Nine strict pairwise reversals occur across four experiments. In this data set, every deletion changes at least one normalization bound, yet five deletions produce no reversal. A change in current-set bounds therefore opens a channel through which reversal may occur, but is not sufficient to produce one by itself.

Deleting $A_5$ provides a direct illustration of this mechanism. Because 
$A_5$ supplies the baseline-set minimum of benefit criterion $C_5$, its 
deletion raises that minimum from $22$ to $56$. Consequently, the normalized 
$C_5$ value of $A_1$ changes from
\[
\frac{56-22}{96-22}=0.459459
\]
to zero. Besides eliminating the direct contribution of $C_5$, this change 
also lowers the self-anchor of $A_1$ and thereby affects its remaining 
criterion contributions. Ultimately, both orders $A_1\succ A_7$ and 
$A_1\succ A_8$ reverse.

The strictly dominated-addition experiment reveals a complementary mechanism.
For every baseline alternative and criterion,
\begin{equation*}
r'_{ij}
=
\beta_j+(1-\beta_j)r_{ij},
\qquad
\boldsymbol{\beta}
=
\left(
\frac{7}{57},
\frac{1}{5},
\frac{3}{34},
\frac{7}{61},
\frac{7}{81}
\right).
\end{equation*}
Every pre-existing normalized value below one increases and moves closer to
one, but the magnitude of this transport differs across criteria. This
criterion-specific contraction toward one changes the normalized profile
shapes and produces three strict pairwise reversals, even though the added
alternative is strictly dominated by every baseline option.

\begin{proposition}[Alternative-set invariance under fixed normalization bounds]
\label{prop:fixed-reference-invariance}
Fix the criterion directions, criterion weights, and normalization bounds independently of the current alternative set. Then the normalized row, self-anchor, kernel values, complete criterion contributions, and final PEJWAK score of every surviving alternative remain unchanged under deletion or addition of other alternatives. Consequently, every strict pairwise order among the surviving alternatives is preserved.
\end{proposition}

\begin{proof}
Under fixed bounds, every normalized coordinate $r_{ij}$ depends only on the raw observation $x_{ij}$ and the fixed criterion reference. Hence $r_i$ is invariant to changes in the alternative set. The quantities $S_i$, $G_{ij}$, $T_{ij}$, and $P_i$ are functions only of $r_i$ and the fixed weight vector, and therefore remain unchanged. All surviving pairwise score differences are consequently preserved.
\end{proof}

Figure~\ref{fig:rank-reversal-transport} summarizes this distinction
visually. Panel~(a) traces the sole dependence pathway in the specified
computational procedure, beginning with a change in the alternative set
and potentially leading to a sign change in $\Delta_{ab}$. Panel~(b)
displays the baseline and post-intervention pairwise gaps for the pairs
that actually undergo reversal.

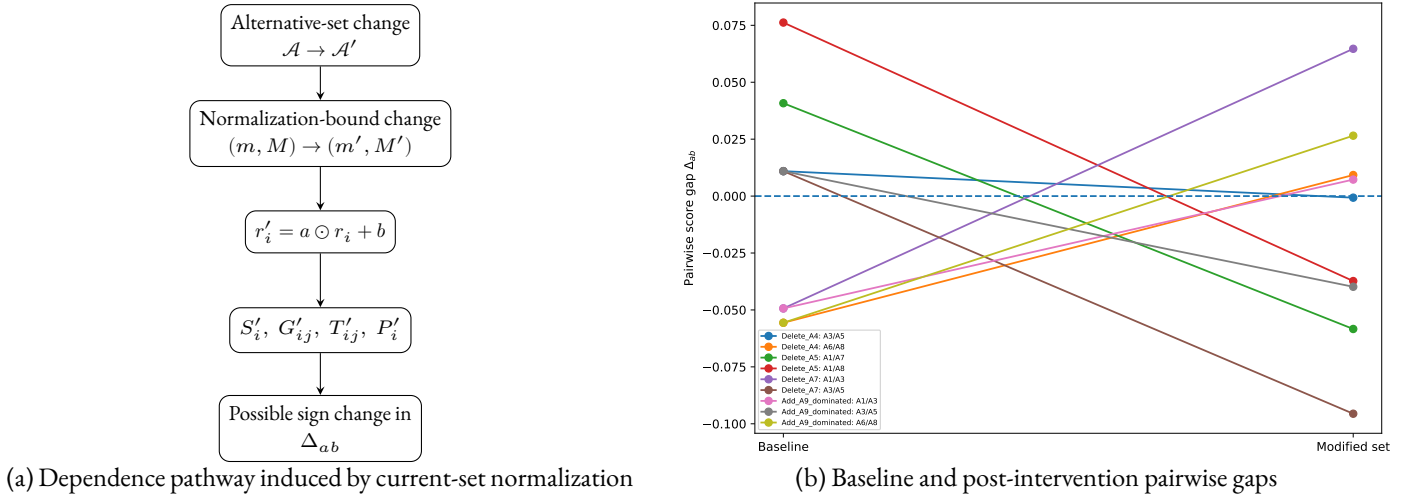
\begin{figure}[htbp]
\centering

\begin{minipage}[t]{0.45\textwidth}
\centering
\begin{tikzpicture}[font=\scriptsize,>=stealth]

\node[
draw,
rounded corners,
align=center
] (set) at (0,0)
{Alternative-set change\\
$\mathcal A\to\mathcal A'$};

\node[
draw,
rounded corners,
align=center
] (bounds) at (0,-1.3)
{Normalization-bound change\\
$(m,M)\to(m',M')$};

\node[
draw,
rounded corners,
align=center
] (trans) at (0,-2.6)
{$\boldsymbol{r}'_i
=
\boldsymbol{a}\odot\boldsymbol{r}_i+\boldsymbol{b}$};

\node[
draw,
rounded corners,
align=center
] (score) at (0,-3.9)
{$S'_i,\;G'_{ij},\;T'_{ij},\;P'_i$};

\node[
draw,
rounded corners,
align=center
] (rev) at (0,-5.2)
{Possible sign change in\\
$\Delta_{ab}$};

\draw[->] (set) -- (bounds);
\draw[->] (bounds) -- (trans);
\draw[->] (trans) -- (score);
\draw[->] (score) -- (rev);

\end{tikzpicture}\\[-1mm]
\small (a) Dependence pathway induced by current-set normalization
\end{minipage}
\hfill
\begin{minipage}[t]{0.52\textwidth}
\centering
\includegraphics[
width=\linewidth
]{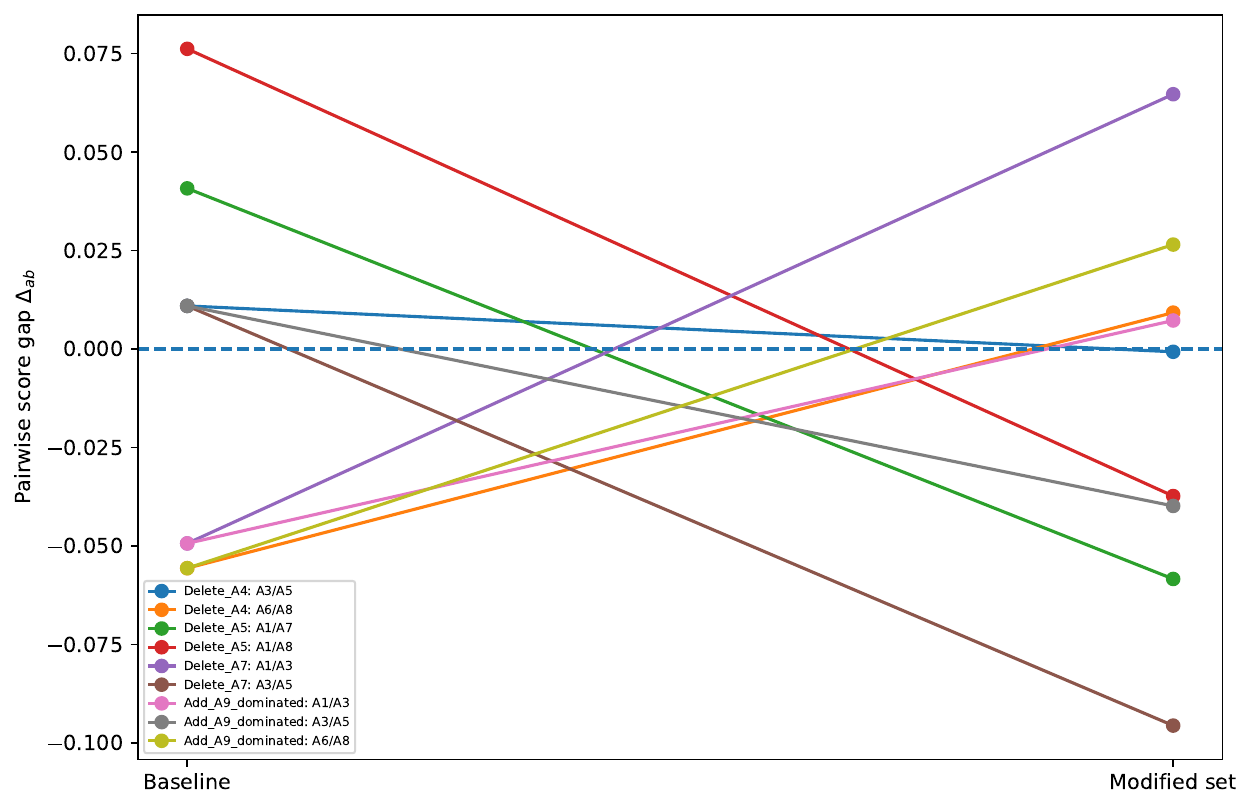}\\[-1mm]
\small (b) Baseline and post-intervention pairwise gaps
\end{minipage}

\caption{Normalization-induced transport and strict pairwise rank reversals.}
\label{fig:rank-reversal-transport}
\end{figure}

The audit does more than report the occurrence of rank reversal; it
localizes its source and provides an exact remedy for the identified
mechanism. Under current-set normalization, criterion-specific affine
transports reshape the surviving profiles and can alter PEJWAK scores.
Under fixed exogenous bounds, every surviving normalized row, self-anchor,
kernel term, criterion contribution, and final score remains unchanged.
Proposition~\ref{prop:fixed-reference-invariance} therefore converts the
diagnosis into a design guarantee: fixed references eliminate this
specific source of set dependence and preserve every strict order among
surviving alternatives.

Within the declared computational procedure, the observed reversals are
not produced by direct cross-alternative dependence inside canonical
PEJWAK. Once normalized rows and criterion importance are fixed, each
alternative is evaluated from its own profile. The dependence enters
through current-set normalization before aggregation. Identifying,
formalizing, and closing this pathway is a diagnostic strength of the
analysis, because it separates the behaviour of the operator from the
behaviour of its preprocessing environment.

\phantomsection
\label{subsec:rank-affinity}

Rank affinity is examined from two genuinely different perspectives.
Directional Weighted Similarity (WS) asks how closely a benchmark
preserves the leading priorities of the declared PEJWAK reference,
assigning exponentially greater influence to its highest positions.
Spearman's coefficient instead measures symmetric association over the
complete ranking. Reporting both coefficients prevents top-order
preservation from being confused with global concordance and reveals
which aspect of the PEJWAK order each benchmark most closely echoes.

In the present comparison, PEJWAK serves as the analytical reference $x$, while
benchmark method $M$ supplies the ranking $y^{(M)}$:
\begin{equation}
WS(x,y^{(M)})
=
1-
\sum_{i=1}^{n}
2^{-x_i}
\frac{|x_i-y_i^{(M)}|}
{\max\{|1-x_i|,|n-x_i|\}}.
\label{eq:ws}
\end{equation}
The coefficient is directional in general:
\[
WS(x,y)\neq WS(y,x).
\]
Accordingly, PEJWAK serves only as the fixed analytical reference in this
comparison, not as a known ground-truth ordering~\cite{Salabun2020}.

For an observed value $s_M$, the exact upper-tail probability under the
strict-ranking permutation null is defined as
\begin{equation}
p_M^{\mathrm{strict}}
=
\frac{1}{n!}
\sum_{\pi\in\mathfrak S_n}
\mathbf{1}
\left\{
WS(x,\pi)\geq s_M
\right\}.
\label{eq:ssd-exact-tail}
\end{equation}

For $n=8$, all
\[
8!=40{,}320
\]
possible permutations are exhaustively enumerated; hence, no Monte Carlo
approximation is used~\cite{Salabun2026}.

Table~\ref{tab:rank-affinity} places the directional WS coefficient, its exact strict-permutation upper tail, and the symmetric Spearman coefficient in a common view. The side-by-side reporting is needed because a method can preserve the top of the PEJWAK reference closely without maximizing concordance over the complete ranking.

\begin{table}[htbp]
\centering
\small
\caption{Rank affinity with PEJWAK as analytical reference.}
\label{tab:rank-affinity}
\begin{tabular}{lrrrr}
\toprule
Method & $WS(\mathrm{PEJWAK},M)$ & Upper-tail count & $p_{\mathrm{strict}}$ & Spearman $\rho_S$\\
\midrule
SAW&0.871615&650&0.016121&0.904762\\
WP&0.874144&611&0.015154&0.627376\\
WASPAS&0.942188&87&0.002158&0.857143\\
$M_2$&0.819494&1621&0.040203&0.714286\\
OWA&0.880469&532&0.013194&0.809524\\
MACONT&0.862128&792&0.019643&0.785714\\
\bottomrule
\end{tabular}
\end{table}

WASPAS achieves the greatest directional top-weighted affinity with
PEJWAK, \(WS=0.942188\), whereas SAW achieves the greatest symmetric
full-ranking association, \(\rho_S=0.904762\). The distinction is
substantively informative. The leading priorities of PEJWAK are most
closely echoed by the declared additive--multiplicative hybrid, while its
complete ordering remains closest to the additive anchor architecture.
Thus, agreement near the top and agreement over the full ranking capture
different structural relationships rather than interchangeable notions
of similarity.

The WP result requires an additional qualification. Its observed ranking
contains a five-way tie represented by midranks, whereas the null
distribution in Equation~\eqref{eq:ssd-exact-tail} contains only strict
rankings without ties. The reported value
\[
p_{\mathrm{WP}}^{\mathrm{strict}}=0.015154
\]
is therefore exact relative to this strict-ranking null distribution.

To assess the sensitivity of the result to the observed tie pattern, a
supplementary audit assigns the rank multiset
\[
(1,2,3,6,6,6,6,6)
\]
to the eight alternatives. Among the
\[
\frac{8!}{5!}=336
\]
distinct assignments, seven produce a WS value at least as large as the
observed WP value. Hence,
\[
p_{\mathrm{WP}}^{\mathrm{tie}}
=
\frac{7}{336}
=
0.020833.
\]
This value is reported solely as a sensitivity check and should not be
interpreted as a formal tie-handling rule of the published nonparametric
significance procedure for the WS coefficient.

Figure~\ref{fig:ws-permutation-tail} illustrates why the WS coefficients
and exact-tail results reported in Table~\ref{tab:rank-affinity} are
top-sensitive. Panel~(a) displays the exponential decay of the
reference-rank weights, while panel~(b) locates the observed coefficient
for each benchmark method on the exact survival function generated by all
$8!$ possible strict rankings.

\begin{figure}[htbp]
\centering

\begin{minipage}[t]{0.38\textwidth}
\centering
\includegraphics[
width=\linewidth
]{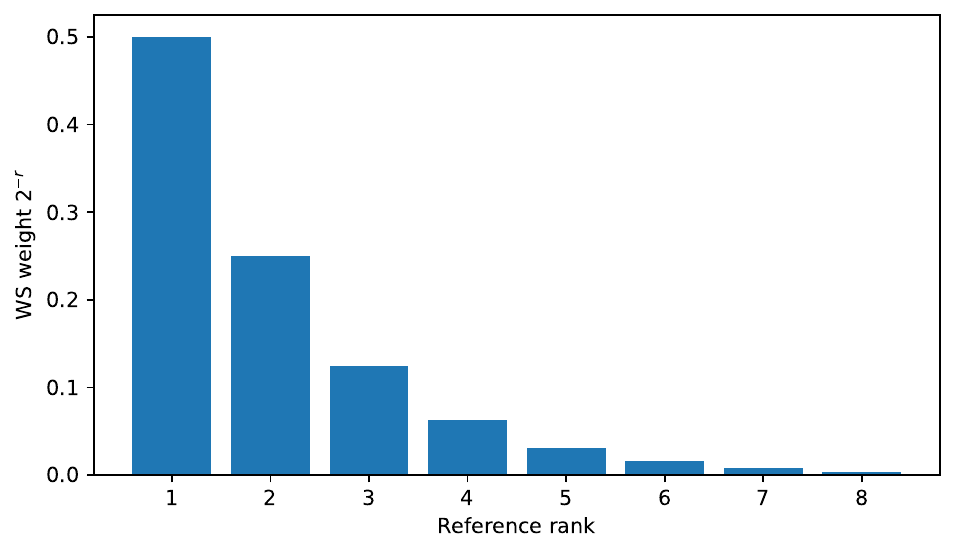}\\[-1mm]
\small (a) Exponential reference-rank weights
\end{minipage}
\hfill
\begin{minipage}[t]{0.59\textwidth}
\centering
\includegraphics[
width=\linewidth
]{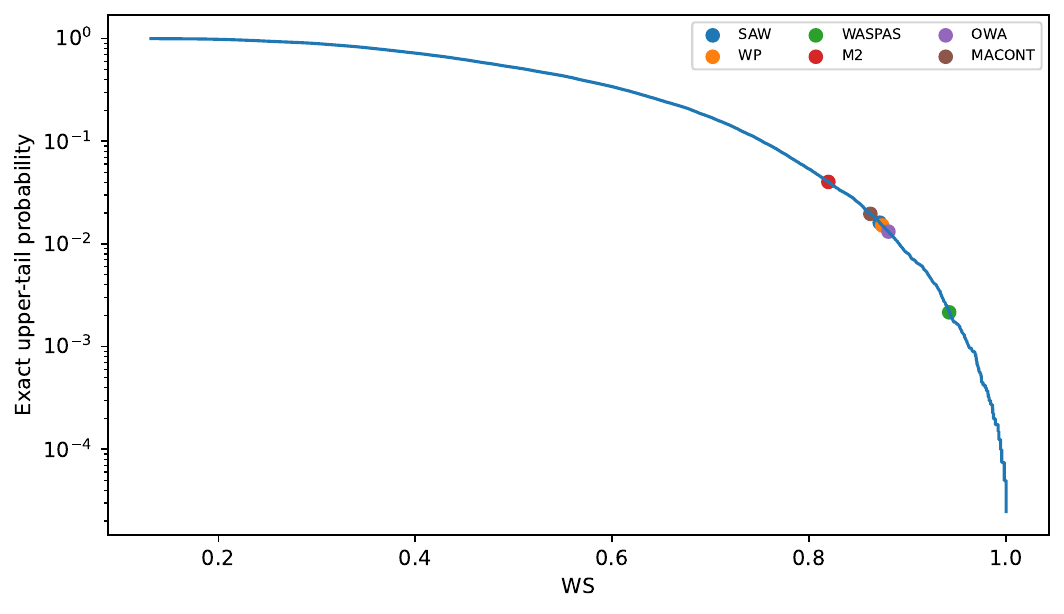}\\[-1mm]
\small (b) Exact survival function under strict-ranking permutations
\end{minipage}

\caption{Directional rank affinity and exact permutation upper tails.}
\label{fig:ws-permutation-tail}
\end{figure}

Exhaustive enumeration of all \(8!\) strict rankings removes Monte Carlo
uncertainty and makes every upper-tail probability exact under the
declared strict-ranking permutation null. These probabilities quantify
how unusual the observed directional affinities are relative to that
reference distribution. Their role is to strengthen the rank-affinity
audit, while external validation of any ranking remains a separate
empirical question.

\phantomsection
\label{subsec:reproducibility-synthesis}

Reproducibility is part of the numerical contribution rather than an
auxiliary reporting feature. One deterministic computational chain
generates the normalized matrix, PEJWAK and benchmark scores, rankings,
criterion contributions, weight-path roots, escort transitions,
set-dependence results, rank-affinity coefficients, and all data-driven
figures. The raw inputs, criterion directions, importance values,
protocol configurations, ranking conventions, and sensitivity paths are
held fixed, while numerical root certificates, manifests, checksums, and
automated verification preserve the trace from every reported quantity
back to its source. The complete numerical argument can therefore be
rerun without manual alteration of intermediate or final outputs.

Table~\ref{tab:evidence-ledger} closes the numerical section by tracing
each research question to its computational evidence, instance-specific
finding, and admissible evidential scope. The final column is part of the
substantive interpretation rather than a generic disclaimer: it prevents
local invariance, rank affinity, and computational reproducibility from
being restated as global robustness, methodological validation, or
external empirical adequacy.

\begin{table}[htbp]
\centering
\scriptsize
\renewcommand{\arraystretch}{1.16}
\setlength{\tabcolsep}{4pt}
\caption{Computational certificates and scientific meaning of the
numerical audit.}
\label{tab:evidence-ledger}

\begin{tabularx}{\textwidth}{p{0.18\textwidth}XXX}
\toprule
Research question
&
Computational certificate
&
Established result
&
Methodological meaning
\\
\midrule

How does PEJWAK handle an isolated boundary zero?
&
Exact decomposition \(P_i=\sum_jT_{ij}\) and the complete zero-set
characterization
&
The zero removes only its own active criterion term whenever another
active criterion sustains a positive anchor.
&
PEJWAK localizes boundary weakness while preserving discrimination among
zero-bearing alternatives.
\\

\addlinespace[2pt]
How stable is the baseline order under importance changes?
&
Continuous \(w_1\) and \(w_3\) paths, certified pairwise roots, and
constant-ranking phase intervals
&
The complete baseline order is preserved throughout the exact intervals
in Equations~\eqref{eq:w1-local-interval} and
\eqref{eq:w3-local-interval}.
&
The architecture supplies local stability certificates and identifies
the precise boundaries at which the order changes.
\\

\addlinespace[2pt]
What does the square-root contribution rule imply?
&
Logarithmic escort path, covariance identity, finite-phase theorem, and
complete root audit over \(q\geq0\)
&
The canonical point lies in the compression regime; twelve transverse
crossings, thirteen rank phases, and a kernel-defined limiting order are
identified.
&
The contribution rule is rank-relevant, structurally consequential, and
fully auditable rather than cosmetic.
\\

\addlinespace[2pt]
How does the alternative set affect surviving orders?
&
Eight deletions, one strictly dominated addition, the affine transport
lemma, and fixed-reference invariance
&
Nine strict pairwise reversals occur in four experiments, while fixed
exogenous bounds preserve every surviving score and order.
&
Set dependence is localized to current-set normalization and can be
removed through fixed reference bounds.
\\

\addlinespace[2pt]
How do benchmark rankings relate to PEJWAK?
&
Directional WS, exhaustive permutation tails, the WP tie-pattern audit,
and Spearman's coefficient
&
WASPAS has the greatest top-weighted affinity, whereas SAW has the
greatest symmetric full-ranking association.
&
Top-order preservation and whole-ranking concordance reveal distinct
architectural relationships with PEJWAK.
\\

\addlinespace[2pt]
Are the numerical results independently reproducible?
&
One reference implementation, scientific and environment manifests,
root certificates, checksums, and automated verification
&
Every numerical table and data-driven figure is regenerated from the raw
inputs and declared settings.
&
The complete numerical argument is traceable and independently
auditable.
\\

\bottomrule
\end{tabularx}
\end{table}

Across these complementary analyses, PEJWAK exhibits identifiable and
reproducible computational signatures: localized boundary effects, exact
stability intervals and transition thresholds, finite contribution-rule
phases, and a clearly isolated normalization pathway for set dependence.
Directional and symmetric affinity measures further separate preservation
of leading priorities from agreement over the complete order. The audit
therefore explains how the scores are formed, certifies where the ranking
is stable, and traces every numerical conclusion to declared inputs and
design rules.

\section{Discussion}
\label{sec:discussion}

PEJWAK is best understood as a distinct aggregation regime rather than a
variation of a weighted mean. Its canonical realization joins three
operations that remain mathematically distinguishable:
\[
S_i=\sum_{j=1}^{n}w_jr_{ij},
\qquad
G_{ij}=K_{w_j}(r_{ij},S_i),
\qquad
P_i=\sum_{j=1}^{n}\varphi_jG_{ij}.
\]
The weighted arithmetic profile is an endogenous reference, not the
terminal score; the kernel interprets each criterion relative to that
reference; and the outer sum converts the anchored terms into complete,
traceable contributions. Context therefore enters before final
aggregation while fixed criterion identity is preserved throughout.
Diagonal calibration and internality keep the result on the normalized
performance scale, and the separation of anchor, kernel, contribution,
and score makes the architecture auditable at every stage.

The exact logarithmic retention law explains the geometry of this
contextualization. For positive score and anchor,
\[
\log\!\left(\frac{K_{\alpha}(x,a)}{a}\right)
=\alpha\log\!\left(\frac{x}{a}\right).
\]
Thus the retention exponent preserves an exact fraction of both upward
and downward logarithmic deviation from the anchor. Because every active
criterion also helps form that anchor, changing one coordinate modifies
the contextual reference and, through it, the anchored terms and marginal
effects of the remaining criteria. The resulting dependence is
profile-mediated and state-dependent. It should not be relabelled as a
fixed synergy, redundancy, or veto effect, because those interpretations
require explicit interaction primitives that PEJWAK does not assume.

Criterion importance consequently operates through three coordinated but
noninterchangeable channels: reference formation, retention of
criterion-specific deviation, and final participation. The square-root
contribution map preserves support and strict importance ordering while
moderating positive ratios according to
\[
\frac{\varphi_j}{\varphi_k}=\sqrt{\frac{w_j}{w_k}}.
\]
This moderation is especially relevant because the same importance
information has already shaped the anchor and the retention exponents.
The escort continuation confirms that the outer rule is structurally
consequential: the canonical point lies in a compression regime, while
progressive concentration produces finitely many certified ranking phases
and a kernel-determined limiting order. The escort exponent remains an
ex post audit coordinate, not an elicited preference parameter.

The theoretical advantages appear directly in the numerical audit. An
active zero eliminates only its own criterion contribution whenever the
remaining active scores sustain a positive self-anchor; unlike a complete
weighted product, it does not erase all information elsewhere in the
profile. Every PEJWAK score is reconstructed exactly as
\(P_i=\sum_jT_{ij}\), so the path from normalized performance through
anchoring and retention to final participation remains visible. The
continuous importance paths provide exact local stability intervals and
transition thresholds, and the contribution-rule audit identifies all
ranking phases over the declared continuation. The set-dependence study
adds an equally important separation: canonical PEJWAK has no direct
cross-alternative dependence once normalized rows are fixed. The observed
deletion and dominated-addition reversals arise through current-set
normalization, which changes the surviving profiles before aggregation;
fixed exogenous bounds close this specific pathway and preserve every
surviving score and strict order.

These results locate PEJWAK in a distinct methodological space. Direct
additive rules provide transparent exchange but no alternative-specific
context before summation. Global multiplicative rules create profile-wide
geometric coupling and propagate an active zero. Ordered operators move
influence to rank positions, post hoc hybrids combine completed indices,
and Choquet-type integrals represent explicit coalitional interaction
through capacities. PEJWAK addresses a different design problem: a fixed
criterion should retain its identity while being interpreted within the
profile of the same alternative. When that semantics is substantively
appropriate, explicit coalition effects are not the target, and complete
decomposition matters, the profile-echoing architecture provides a
capability not supplied in equivalent form by those families. Its
contribution is therefore an aggregation relation, not a promise to
produce a different ranking in every instance.

\section{Scope and Cross-Domain Applicability}
\label{sec:limitations}

The present study establishes the real-valued mathematical core of
PEJWAK and audits its canonical realization in a controlled
multi-criteria decision analysis experiment. The proved results concern
the architecture itself, including the profile-derived anchor,
criterion-level geometric coupling, logarithmic retention, complete score
decomposition, and coherence on the full score-weight domain. The
numerical conclusions are exact for the declared data, normalization
rules, importance paths, and benchmark protocols. They establish how the
mechanism operates; predictive performance and domain-specific
effectiveness require independent empirical evidence.

Multi-criteria decision analysis is the first application semantics of
PEJWAK, not its mathematical boundary. Formally, the architecture acts on
a bounded labelled profile whose components retain fixed identities,
construct an endogenous reference, and are coupled with that reference
before final aggregation. Transfer to another domain is coherent when the
components admit a defensible common scale, their identities remain
meaningful, the within-instance reference has a substantive
interpretation, and the three roles of importance can still be
separated. Under those conditions, an alternative may become an observed
object, event, time window, spatial unit, or system state, and a criterion
may become an identifiable sensor, model, channel, feature, or indicator.

Information and model fusion provide the most immediate extension.
Aggregation operators have long combined sensors, detectors,
classifiers, and neural networks
\cite{Auephanwiriyakul2002,ChoKim1995}. In a PEJWAK formulation, each
source would remain identifiable, the self-anchor would summarize the
outputs for the same observed instance, and the final fused result would
be decomposable into source-level contributions. This architecture is
particularly relevant when within-instance context and source
traceability are both important, although calibration, correlated error,
missing observations, and source reliability must be evaluated within
the application.

Signal and image processing offer a second natural setting. Weighted
aggregation has been used for biomedical signals with unequal noise
levels, and fuzzy-integral fusion has been applied in computer vision
\cite{Bataillou1995,TahaniKeller1990}. A PEJWAK profile could be formed
from signal channels, repeated measurements, spectral components,
feature detectors, or outputs within a spatial or temporal window. The
self-anchor would provide a local reference and the kernel would regulate
how much channel-specific deviation remains in the fused representation.
Any such use would require task-specific assessment of noise robustness,
spatial or temporal consistency, and computational scalability.

Economics, composite measurement, security, and reliability are also
structurally compatible domains. Nonadditive aggregation has a
well-established role in expected utility under nonadditive probability,
and Choquet-based classifier fusion has been used in financial-distress
early warning \cite{Schmeidler1989,Cao2012}. PEJWAK may therefore be
examined for economic, social, sustainability, resilience, or
organizational indices when indicator meaning is partly profile-dependent.
Likewise, aggregation operators have been incorporated into attack-tree
and reliability models \cite{Yager2006}; a profile-echoing formulation
could retain the identities of vulnerability indicators, alarms, or
component-health measures while producing an auditable system-level
assessment. In each case, the aggregate must satisfy the interpretive,
axiomatic, and empirical requirements of the target field.

These pathways establish the structural portability of profile echoing;
they are not applications validated by the present numerical study. The
transfer claim is deliberately precise: PEJWAK is a
profile-conditioned aggregation architecture rather than an MCDA-only
scoring formula. Its relevance extends to settings in which bounded
labelled components must be combined while component identity,
within-instance context, and contribution-level traceability remain
substantively meaningful. The normalization scale, self-anchor,
importance semantics, and final output must nevertheless be justified and
tested within each domain.

\section{Conclusion}
\label{sec:conclusion}

This study introduced PEJWAK as a profile-echoing self-anchored
aggregation architecture. Its central design choice precedes the final
reduction of a multidimensional profile to a scalar score: each fixed
criterion is first interpreted relative to an endogenous reference
extracted from the same alternative. The criterion therefore enters the
aggregate neither as an isolated term in a completed sum nor as an
indistinguishable factor in one global product. It preserves its identity
while its expressed contribution is conditioned by the profile to which
it belongs.

The formal development establishes a general family that separates the
anchor function, anchoring kernel, retention exponents, contribution
coefficients, and outer aggregation rule. The canonical member links
these layers through one criterion-importance vector and assigns that
information three distinct roles: self-anchor formation,
criterion-specific retention, and final participation. The resulting
operator is diagonally calibrated, internal, monotone, and jointly
continuous on the complete score-weight domain. Exact logarithmic
retention, endogenous cross-effects, nonseparability, failure of
bisymmetry in the equal-weight binary restriction, and state-dependent
marginal substitution demonstrate a genuine profile-mediated geometry
rather than a disguised additive or classical quasi-arithmetic form.

The numerical audit translates this structure into verifiable
computational signatures. Boundary zeros are localized to their own
active contributions; every score is completely decomposable; continuous
importance paths yield certified stability intervals and transition
thresholds; and the escort continuation produces finitely many
constant-ranking phases and a kernel-defined limiting order. The
set-dependence analysis identifies current-set normalization as the
specific pathway through which changes in the alternative set can alter
surviving profiles and strict pairwise orders, while fixed exogenous
bounds close that pathway. Exact permutation enumeration, transition
certificates, manifests, checksums, and deterministic reruns preserve the
trace from reported results to declared inputs.

PEJWAK consequently occupies an important space between direct weighted
aggregation and parameter-intensive nonadditive interaction models. It
introduces alternative-specific context without subset capacities,
preserves fixed criterion identity, separates retention from final
participation, and supports complete contribution-level auditing.
Multi-criteria decision analysis provides its first application
semantics, but not its mathematical limit. Wherever bounded labelled
components must be combined and their meaning depends partly on the
profile of the same observation, profile echoing defines a distinct
aggregation principle: a component enters the final evaluation not merely
as a value, but as a value interpreted within its own profile.


\section*{Acknowledgments}
The author sincerely thanks the anonymous reviewers for their careful,
constructive, and sustained evaluations across successive rounds of
review. Their comments materially improved the clarity, rigor,
organization, and presentation of the manuscript.

\section*{Funding}
This research did not receive any specific grant from funding agencies
in the public, commercial, or not-for-profit sectors.

\section*{Data Availability}
The raw data are reported in the manuscript.

\section*{Conflicts of Interest}
The author declares no conflict of interest.

\section*{Author Contributions}
\textbf{S. A. Edalatpanah:}
Conceptualization; Methodology; Formal analysis;
Investigation; Software; Validation; Visualization; Writing--original
draft; Writing--review and editing.
\section*{Declaration on the Use of Generative Artificial Intelligence}
During the development and preparation of this article, the author used
generative artificial-intelligence tools, including assistants based on
large language models, in a limited capacity and under his full
supervision. These tools were employed to improve the clarity,
readability, linguistic accuracy, internal coherence, and presentation
of parts of the manuscript; to assist in the editorial preparation of
the text in \LaTeX; and to support formal-consistency checks and
computational cross-checking. The conceptual architecture of PEJWAK, the
mathematical definitions, analytical results, proofs, experimental
design, numerical interpretations, and scientific conclusions are the
author's own. No AI-generated material was incorporated into the
manuscript without the author's careful review, verification, and
revision. Full responsibility for the originality, accuracy,
reproducibility, scientific integrity, and final content of the article
rests solely with the author.

\end{document}